\documentclass[10pt, reqno]{amsart}
\usepackage{latexsym,graphics,epic,epsfig,wrapfig}
\usepackage{amsfonts,mathrsfs,amssymb,amsmath,amscd,amsthm,mathbbol,mathrsfs}
\usepackage{oldgerm,units}
\usepackage{blkarray}

 \usepackage[table]{xcolor}
\usepackage{graphics}
\usepackage{pstricks}
\usepackage{youngtab}
\usepackage{mathabx}
\usepackage{multirow}
\usepackage{amssymb,  color,  rotating}
\usepackage{blkarray}

\input amssymb.sty
\input xy
\xyoption{all}

\newcommand{\mresize}[1]{{ \small #1 }}

\definecolor{lgray}{gray}{0.90}
\definecolor{mypink1}{rgb}{0.858, 0.388, 0.578}
\definecolor{mypink2}{RGB}{255, 0, 255}
\definecolor{mypink3}{cmyk}{0, 0.7808, 0.4429, 0.1412}
\definecolor{mygrayI}{gray}{0.6}
\definecolor{mygrayII}{gray}{0.8}
\definecolor{mygrayIII}{gray}{0.9}
\def\grcel{\cellcolor{mygrayII}}

\def\embox{{\phantom{x}}}

\def\dispace{\setlength{\itemsep}{2pt}}

\def\fsmg{forward semigroup}
\def\cloak{cloaktic}
\def\ccloak{co-\cloak}

\def\kmon{\cloak\ monoid}
\def\ckmon{co-\kmon}
\def\keqv{cloaktically equivalent}

\def\hstep{horizontal  step}
\def\vstep{vertical step}
\def\dstep{slant step}

\def\dwalk{descending  walk}
\def\vwalk{falling  walk}
\def\hwalk{horizontal cover}

\def\flt{\pv}
\def\eflat{$\flt$-flat}

\def\lcorn{$\ell$-corner}

\def\pSkip{\vskip 1mm \noindent}
\def\sSkip{$ $ \vskip 2mm}

\newcommand{\joverline}[2]{%
  \mathord{
    \vbox{\offinterlineskip
      \halign{##\cr
        $\scriptscriptstyle#1$\hrulefill\cr
        \noalign{\kern.4ex}
        $\; #2 \,$\cr
      }%
    }%
  }%

}

\newcommand\arrs[2]{\scriptsize{\begin{array}{c}
                                                    #1, \\ #2
                                                  \end{array}}}

\newcommand{\longleft}[1]{\;{\leftarrow%
\count255=0 \loop \mathrel{\mkern-6mu}%
    \relbar\advance\count255 by1\ifnum\count255<#1\repeat}\;}
\newcommand{\longright}[1]{\;{\count255=0 \loop \relbar\mathrel{\mkern-6mu}%
    \advance\count255 by1\ifnum\count255<#1\repeat\rightarrow}\;}
\newcommand{\Right}[2]{\overset{#2}{\longright{#1}}}
\newcommand{\RIGHT}[3]{\mathrel{\mathop{\kern0pt\longright{#1}}
    \limits^{#2}_{#3}}}

\newcommand{\LEFT}[3]{\mathrel{\mathop{\kern0pt\longleft{#1}}\limits^{#2}_{#3}}
}
\newcommand{\longleftright}[1]{\;{\leftarrow\mathrel{\mkern-6mu}%
    \count255=0\loop\relbar\mathrel{\mkern-6mu}%
    \advance\count255 by1\ifnum\count255<#1\repeat\rightarrow}\;}

\newcommand{\onto}[1]{\;{\count255=0 \loop \relbar\joinrel
    \advance\count255 by1
    \ifnum\count255<#1 \repeat \twoheadrightarrow}\;}
\newcommand{\Onto}[2]{\overset{#2}{\onto{#1}}}
\newcommand{\RLEFT}[3]{\mathrel{%
   \mathop{\vcenter{\baselineskip=0pt\hbox{$\kern0pt\longright{#1}$}%
   \hbox{$\kern0pt\longleft{#1}$}}}\limits^{#2}_{#3}}}

\newcommand{\To}{\longrightarrow }
\newcommand{\TO}{\longright{3} }
\newcommand{\mto}{\mapsto }
\newcommand{\mTo}{\longmapsto }

\def\Sim{\,\smash{\raisebox{-0.45ex}{\ensuremath{\scriptstyle\sim}}}\,}
\newcommand{\ONTO}{\onto{3} }
\newcommand{\ISOTO}{\Right{3}{\Sim}}
\newcommand{\Isoto}{\Right{1}{\Sim}}

\newcommand{\SIM}[1]{\underset{\Sim}{#1}}

\newcommand{\hooklongrightarrow}{\lhook\joinrel\longrightarrow}
\newcommand{\Hooklongrightarrow}{\lhook\joinrel\joinrel\joinrel\joinrel\TO}

\newcommand{\Into}{\hooklongrightarrow}
\newcommand{\INTO}{\Hooklongrightarrow}
\newcommand{\Mto}{\longmapsto}

\def\insrt{\looparrowright}
\def\balg{\begin{align}}
\def\ealg{\end{align}}

\def\+{\vee}
\def\-{\wedge}
\def\d+{\, \widehat{+} \, }
\def\bigp{\bigvee}
\def\bigm{\bigwedge}

\def\chra{\chi^+}
\def\chrp{\chi^\times}

\def\ddual{^\-}

\newcommand\oix[2]{{#1}^{(#2)}}
\newcommand\Oix[2]{{#1}_{(#2)}}
\newcommand\sMon[2]{{#1}^{[#2]}}
\newcommand\smon[2]{#1^{[#2]}}
\def\TT{T}

\def\KXa{$\operatorname{K1}$}
\def\KXb{$\operatorname{K2}$}

\def\TPA{PA}
\def\PXa{$\operatorname{\TPA1}$}
\def\PXb{$\operatorname{\TPA2}$}
\def\PXc{$\operatorname{\TPA3}$}
\def\PXd{$\operatorname{\TPA4}$}
\def\dPXa{$\operatorname{\TPA1}'$}
\def\dPXb{$\operatorname{\TPA2}'$}
\def\dPXc{$\operatorname{\TPA3}'$}
\def\dPXd{$\operatorname{\TPA4}'$}

\def\KXa{$\operatorname{KN1}$}
\def\KXb{$\operatorname{KN2}$}

\def\NDXa{$\operatorname{FS}$}

\def\bfc{{\bf c}}
\def\bfr{{\bf r}}

\def\udscr{\underline{\phantom{x}} \; }

\newcommand\mcl[2]{\rwcl{#1}{{\udscr}}{#2}}
\newcommand\mrw[2]{\rwcl{#1}{#2}{{\udscr}}}
\newcommand\clrw[3]{#1[#3,#2]}
\newcommand\rwcl[3]{#1[#2,#3]}

\def\Cl{\operatorname{Col}}
\def\Rw{\operatorname{Row}}

\def\zeroM{(\zero)}

\def\w{k}
\def\a{a}

\def\aa{a}
\def\bb{b}
\def\cc{c}

\def\elm{{\operatorname{elm}}}
\def\srHom{\varphi}
\newcommand\Pv[1]{#1}

\def\pv{\kappa}
\def\ipv{\inv{\pv}}
\def\lt{\a}

\def\Al{\tA}

\def\eM{E}
\def\fM{F}

\def\dom{\rhd}

\def\Iff{\quad \Leftrightarrow \quad }

\def\clcong{\equiv_\clk}
\def\cclcong{{\equiv^{\operatorname{co}}_\clk}}
\def\cpcong{{\equiv^{\operatorname{co}}_\plc}}

\def\tcong{\equiv_\tab}
\def\kcong{\equiv_\knu}
\def\pcong{\equiv_\plc}

\def\bcong{\equiv_\byc}

\def\bsim{\sim_\blk}
\def\bleq{\subseteq_\blk}
\def\bgeq{\supseteq_\blk}

\def\NDSg{\mathcal{F}}

\def\MFrm{\mathfrak{F}}

\def\MPlc{\mfAml}

\def\dMPlcn{\mir{\mfAml_n}}

\def\APlc{\mathfrak{plc}}

\def\dAPlc{\APlc\ddual}

\newcommand\rvs[1]{\overleftarrow{#1}}
\newcommand\std[1]{\overrightarrow{#1}}

\newcommand\lcmp[2]{\co{\mfM}_{#1}(#2)}

\newcommand\ndc[1]{{#1}^{\uparrow}}
\newcommand\nic[1]{{#1}^{\downarrow}}

\newcommand\NDC[1]{\ndc{\operatorname{Sq}}( #1 )}
\newcommand\NIC[1]{\nic{\operatorname{Sq}}( #1 )}
\newcommand\MNDC[1]{\ndc{\operatorname{MSq}}( #1 )}
\newcommand\MNIC[1]{\nic{\operatorname{MSq}}( #1 )}

\newcommand\squ[1]{[ #1 ]}
\newcommand\incs[1]{ {#1}^{\uparrow}}
\newcommand\decs[1]{ {#1}^{\downarrow}}
\newcommand\lensq[1]{\operatorname{len}\big({#1}\big)}

\def\decS{\decs{S}}

\def\incR{\incs{R}}
\def\incS{\incs{S}}

\def\rev{{\operatorname{rev}}}
\def\mirM{{\operatorname{cmr}}}
\def\cmirr{co-mirror}
\def\Cmirr{Co-mirror}

\def\RInj{\operatorname{Inj_r}}
\def\Tab{\mathsf{Tab}}
\def\STab{\mathsf{STab}}
\def\CTab{\mathsf{CTab}}
\def\SCTab{\mathsf{SCTab}}

\def\HC{\Theta}

\def\clk{{\operatorname{clk}}}

\def\tab{{\operatorname{tab}}}
\def\stab{{\operatorname{stab}}}
\def\ctab{{\operatorname{ctab}}}

\def\mat{{\operatorname{mat}}}

\def\tabctab{\scT_\ctab}
\def\stabsctab{\scT_\stab}

\def\ctabmat{\scC_\mat}
\def\ctabcomat{\scC_\mat^{\operatorname{co}}}
\def\tabmat{\scT_\mat}
\def\tabcomat{\scT_\mat^{\operatorname{co}}}

\def\stdxrvs{\scS}

\def\tb{\mathfrak{T}}
\def\ctb{\mathfrak{C}}

\def\rw{\mathfrak{r}}
\def\cl{\mathfrak{c}}
\def\row{\operatorname{row}}

\def\Lin{{\operatorname{Lin}}}
\def\Mat{{\operatorname{Mat}}}
\def\MatnT{\Mat_n(\T)}
\def\TMatnT{\TMat_n(\T)}
\def\dMatnT{\Mat_n(\T_\-)}
\def\MatnB{\Mat_n(\B)}

\def\TMat{{\operatorname{TMat}}}
\def\TMatnT{\TMat_n(\Trop)}

\def\mfa{\mathfrak{a}}
\def\mfb{\mathfrak{b}}
\def\mfc{\mathfrak{c}}
\def\mfd{\mathfrak{d}}
\def\mfe{\mathfrak{e}}
\def\mff{\mathfrak{f}}
\def\mfu{\mathfrak{u}}
\def\mfv{\mathfrak{v}}

\def\mfo{\mathfrak{o}}

\def\delm{\mfd}
\def\felm{\mff}

\def\cent{\lm}
\def\Cent{\cdiag}

\def\cdiag{\tau}

\def\mfAml{\mfA^\times}

\def\mfA{\mathfrak{A}}

\def\mfM{\mathfrak{M}}
\def\mfN{\mathfrak{N}}

\def\mfS{\mathfrak{S}}

\def\mfS{\mathfrak{S}}

\def\mfp{\mathfrak{p}}
\def\mfq{\mathfrak{q}}
\def\mfr{\mathfrak{r}}

\def\p{\mfp}
\def\q{\mfq}
\def\r{\mfr}
\def\tt{t}

\def\tlmu{\widetilde \mu}
\def\tlpi{\widetilde \pi}

\def\PLC{\mathsf{PLC}}
\def\CPLC{\co{\PLC}}
\def\CLK{\mathsf{CLK}}
\def\CCLK{\co{\CLK}}

\def\Cor{\operatorname{Cor}}
\def\CorT{\Cor_n(\Trop)}
\def\Syn{\operatorname{Syn}}

\def\SynT{\Syn_n(\Trop)}

\def\plcM{\mathcal{P}}
\def\clkM{\mathcal{K}}
\def\coclkM{\co{\clkM}}
\def\coplcM{\co{\plcM}}

\newcommand\ltno[2]{\#^{(#1)}_{#2}}

\newcommand\sword[2]{\natural_{#1}}
\newcommand\jmp[2]{\eta_{#1}}
\newcommand\itoj[2]{[#1:#2]}
\newcommand\sitoj[2]{\{#1:#2\}}
\newcommand\tbrest[3]{#1|_{\itoj{#2}{#3}}}

\newcommand\co[1]{{\, ^{\operatorname{co}}\joinrel\joinrel\joinrel\joinrel{#1}}}

\newcommand\mir[1]{{^\diamond\joinrel\joinrel\joinrel{#1}}}

\def\frd{{\operatorname{fw}}}

\def\knu{{\operatorname{knu}}}
\def\plc{{\operatorname{plc}}}
\def\coplc{{\operatorname{cplc}}}

\def\byc{{\operatorname{bcy}}}
\def\blk{{\operatorname{blk}}}

\def\clkrep{\mho}
\def\dclkrep{\Om}
\def\matcmat{\Game}
\def\pMap{\partial}

\def\plktoclk{\eth}
\def\plktocclk{\co{\, \plktoclk}}

\def\plcrep{\wp}
\def\plccplc{\widetilde{\pMap}}

\def\symrep{\scH}
\def\symrepB{\scQ}

\def\chId{\widecheck{\Id}}

\def\rep{\rho}

\newcommand\genr[1]{\langle\, {#1} \,\rangle}

\def\lex{{\operatorname{lx}}}

\def\cnx{{\operatorname{cx}}}
\def\wrd{{\operatorname{wd}}}
\def\cnxsset{\subseteq_\cnx}
\def\wrdsset{\subseteq_\wrd}

\newcommand\inv[1]{#1^{{{\small -1}}}}
\def\B{\mathbb B}
\def\T{\mathbb T}
\def\N{\mathbb N}

\def\Q{\mathbb Q}
\def\Z{\mathbb Z}
\def\nnn{\{1 ,\dots, n  \}}

\def\nxn{n\times n}

\def\grph{G}

\def\per{{\operatorname{per}}}

\def\An{A^{n!}}
\def\Bn{B^{n!}}
\def\Id{\Pi}
\def\htId{\widehat\Id}

\def\tA{\mathcal A}
\def\tB{\mathcal B}

\def\htv{\widehat v}
\def\htu{\widehat u}

\def\N{\mathbb N}
\def\Z{\mathbb Z}

\def\cha{\widecheck{a}}

\def\chA{\widecheck{A}}
\def\chB{\widecheck{B}}
\def\chC{\widecheck{C}}

\def\hta{\widehat{a}}

\def\htA{\widehat{A}}

\def\chw{\widecheck{w}}
\def\chx{\widecheck{x}}
\def\chy{\widecheck{y}}

\def\chC{\widecheck{C}}

\def\lm{\lambda}

\def\sm{\setminus}
\def\tS{\mathcal S}
\def\Om{\Omega}

\def\om{\omega}

\def\ver{\mathcal V}
\def\arc{\mathcal E}
\def\varX{\mathcal A}

\newcommand\factor[2]{#1 | #2}

\def\pre{\operatorname{pre}}
\def\suf{\operatorname{suf}}

\def\wlen{\operatorname{len}}

\def\tlw{\widetilde w}

\newcommand\thmcite[2]{\pSkip\textbf{#1. }\emph{#2}\vskip 1mm}

\def\x{{\underline{x}}}
\def\y{{\underline{y}}}

\def\tM{\mathcal M}

\newcommand\cset[1]{#1^\dagger}

\def\set2{{\cset{2}}}

\def\setW{{\cset{W}}}

\def\setP{{\cset{P}}}

\def\setU{{\cset{U}}}
\def\setUp{{\cset{U'}}}

\def\setW{{\cset{W}}}

\def\l2{x^2y^2x}
\def\ll2{yx^2y^2x}

\def\lllt{xy^3xyx^3y}

\newcommand{\trk}[1]{\operatorname{trk}(#1)}
\newcommand{\len}[1]{\operatorname{len}(#1)}

\newcommand{\mxlen}[2]{\overrightarrow{\operatorname{len}}_{#1}(#2)}

\def\sw{{\operatorname{sw}}}
\newcommand{\sbwd}[2]{\sw_{#1}(#2)}
\newcommand{\incsbwd}[2]{\overrightarrow{\sw}_{#1}(#2)}

\newcommand{\ds}[1]{\ {#1} \ }
\newcommand{\dss}[1]{\quad {#1} \quad }

\def\twt{w}
\newcommand{\tbwt}[1]{\twt(#1) }

\newcommand{\mxtbwt}[1]{\overline{\om} #1 }
\newcommand{\mntbwt}[1]{\underline{\om}#1 }

\def\pth{\gm}
\def\wlk{\gm}
\def\cov{\theta}

\def\Wlk{\Gm}

\def\grph{G}

\def\one{\mathbb{1}}
\def\zero{\mathbb{0}}

\def\mT{\MatnT}
\def\bT{\MatnB}

\def\scC{\mathscr C}

\def\scH{\mathscr H}

\def\scS{\mathscr S}
\def\scT{\mathscr T}
\def\scQ{\mathscr Q}
\def\scP{\mathscr P}

\newtheorem{theorem}{Theorem}[section]
\newtheorem{proposition}[theorem]{Proposition}
\newtheorem{definition}[theorem]{Definition}
\newtheorem{algorithm}[theorem]{Algorithm}

\newtheorem{lemma}[theorem]{Lemma}
\newtheorem{klemma}[theorem]{Key lemma}

\newtheorem{notation}[theorem]{Notation}
\newtheorem{corollary}[theorem]{Corollary}

\newtheorem{example}[theorem]{Example}
\newtheorem{remark}[theorem]{Remark}

\newtheorem{observation}[theorem]{Observation}

\newtheorem{problem}[theorem]{Problem}

\newtheorem{construction}[theorem]{Construction}
\newtheorem{properties}[theorem]{Properties}

\newtheorem{algr}[theorem]{Algorithm}

\newtheorem*{complexity*}{Time complxity}

\newtheorem*{theorem*}{Theorem}
\newtheorem*{problem*}{Problem}

\newcommand{\junk}[1]{}
\newcommand{\etype}[1]{\renewcommand{\labelenumi}{(#1{enumi})}}
\def\eroman{\etype{\roman} \dispace}
\def\ealph{\etype{\alph} \dispace }

\newcommand{\bfem}[1]{\textbf{\emph{#1}}}
\def\({\left(}
\def\){\right)}

\def\al{\alpha}

\def\gm{\gamma}
\def\Gm{\Gamma}

\def\e{\varepsilon}
\def\w{\om}

\def\sig{\sigma}

\def\minf{-\infty}

\def\Real{\mathbb R}
\def\Trop{\mathbb T}
\def\Comp{\mathbb C}

\def\Bool{\mathbb B}
\def\Z{\mathbb Z}
\def\Q{\mathbb Q}

\def\tr{\operatorname{tr}}
\def\mtr{\operatorname{mtr}}

\def\row{\operatorname{row}}

\def\trn{{\operatorname{t}}}

\def\grph{G}

\def\minf{-\infty}

\def\htA{\widehat A}

\title[Tropical plactic algebra]
 {Tropical plactic algebra, \\ the cloaktic monoid, 
 \\  and semigroup  representations}

\author{Zur Izhakian}

\address{  Institute  of Mathematics,
 University of Aberdeen, AB24 3UE,
Aberdeen,  UK.
    }
    \email{zzur@abdn.ac.uk}


\date{January 17, 2017}


\thanks{The research of the author  has been sported by the Research Councils UK (EPSRC), grant no EP/N02995X/1.}

\thanks{\pSkip Institute  of Mathematics,
 University of Aberdeen, AB24 3UE,
Aberdeen,  UK. \\ $ $ \ Email: zzur@abdn.ac.uk.
}

\begin{document}

\begin{abstract}
A new tropical plactic algebra is introduced in which the  Knuth relations are  inferred  from the underlying  semiring arithmetics, encapsulating the ubiquitous plactic monoid ~$\plcM_n$. This algebra manifests a natural framework for    accommodating  representations of $\plcM_n$, or equivalently of Young tableaux, and its moderate coarsening ---  the \kmon \ $\clkM_n$ and the \ckmon\ $\coclkM_n$. The faithful linear representations of $\clkM_n$ and $\coclkM_n$ by
tropical matrices, which constitute a tropical plactic algebra,  are  shown to provide linear representations of the plactic monoid.
To this end the paper develops a special type of configuration tableaux,  corresponding  bijectively to  semi-standard Young tableaux. These special tableaux allow  a systematic encoding of combinatorial properties in numerical algebraic ways,  including  algorithmic benefits.
The interplay between these algebraic-combinatorial structures establishes  a profound machinery for exploring   semigroup attributes, in particular satisfying of semigroup identities. This  machinery is  utilized here to prove that $\clkM_n$ and $\coclkM_n$ admit all the semigroup identities satisfied by $\nxn$ triangular tropical matrices, which  holds also for~$\plcM_3$.
\end{abstract}

\maketitle


\setcounter{tocdepth}{1}
  {\small \tableofcontents}


\section*{Introduction}
\numberwithin{equation}{section}

This paper introduces linear representations of the plactic monoid and its coarsening, called  the \cloak\ and the \ccloak\ monoids, together with a new encapsulating algebra.
The \textbf{plactic monoid} $\plcM_I = \PLC(\tA_I)$  is the presented monoid $\tA_I^*/ _{\kcong}$. That is,  the quotient of the free monoid~ $\tA_I^*$  over an ordered alphabet $\tA_I$ by the congruence $\kcong$ determined by the \textbf{Knuth relations} (also called \textbf{plactic relations})
\begin{equation}\label{eq:knuth.rel}
  \begin{array}{cccc}
   & \aa \; \cc \; \bb = \cc \; \aa \; \bb & \text{if} &  \aa \leq \bb < \cc \; , \\[1mm]
   & \bb \;\aa \; \cc = \bb \; \cc \; \aa & \text{if} & \aa <  \bb \leq \cc \; .  \\  \end{array} \tag{KNT}
\end{equation}
This monoid was  first  appeared  in the context of Young tableaux (Knuth \cite{Knuth} and Schensted \cite{Schensted})  and  has been followed  by an extensive  study  of   Lascoux and Sch\"{u}tzenberger \cite{Las2}, establishing  an important link between combinatorics and algebra via its bijective correspondence to semi-standard \textbf{Young tableaux} ~\cite{Kirillov.Pol}.

 The combinatorics of $\plcM_I$ is framed by Young tableaux and has been studied mainly in  traditional perspective (e.g., Gr\"{o}bner-Shirshov bases ~\cite{CGM}). It has various applications (e.g., in  symmetric functions~ \cite{Mac}, Kostka-Foulkes
polynomials \cite{Las3}, and Schubert polynomials \cite{Las4}). Related representations have looked at monoid algebras $K[\plcM_I]$ over a field $K$ for a finite $I$ \cite{COk.1, KOk.1}.
 At the same time,
 Young tableaux has been found to be an important characteristic patterns
in classical representation  theory of groups \cite{Fulton,Green}, especially in representations of the symmetric group~ \cite{SaganBook} and algebraic combinatorics~ \cite{Lothaire}, as well as in combinatorics ~ \cite{Loehr} and computer science \cite{Romik}.  Therefore, in this challenging  arena,  a direct algebraic description of the plactic monoid
 strengthens the mutual connections to classical representation theory, providing an additional   algebraic-combinatorial approach to group representations.

To address this goal, this paper introduces
\begin{itemize}\dispace
  \item  a semiring structure in which the Knuth relations \eqref{eq:knuth.rel} are inferred from its  arithmetics,
  \item  meaningful   coarser monoids  of the plactic monoid, and

  \item  linear representations of the plactic monoid and its coarsening.
  \end{itemize}
These objects pave a new  way  to study the plactic monoid and Young tableaux systematically. 
To develop linear representation,  we are assisted by several auxiliary objects (i.e.,  semigroups, tropical matrices, digraphs, and configuration tableaux), enhancing the interplay among them.

The fundamental structure of the monoids of our main interest, determined as a quotient by multiplicative congruences,  makes  their  analysis rather difficult. To  approach  such monoids straightforwardly, we  adopt  the  familiar concept of associating  an algebra  to a multiplicative group which allows one to study groups in terms of algebras,   e.g., the  correspondence between Lie groups and Lie algebras. This paper applies a similar concept to the plactic monoid $\plcM_I$ by encapsulating $\plcM_I$  in  the \textbf{tropical plactic algebra} $\APlc_I$ --- a new semiring structure --- abbreviated \textbf{troplactic algebra}. The underlying structure of this algebra is a noncommutative idempotent semiring (Defunition ~\ref{defn:trop.plc.alg}), generated by a set of ordered elements. In this algebra the Knuth relations \eqref{eq:knuth.rel} are intrinsically  followed  from the semiring arithmetics (Theorem ~ \ref{thm:plc.Alg.1}).

An important  property of $\APlc_I$ is that every product of its elements is uniquely decomposable to a sum of non-nested terms, each is a nondecreasing subsequence  (Corollary \ref{cor:facotrize.plc}). Furthermore, the generators of $\APlc_I$  admit the
 Frobenius, property (Lemma ~\ref{lemma:Frob.2}):
\begin{equation*}
    \big(\mfa + \mfb\big)^m = \mfa^m + \mfb^m \qquad \text{for any } m \in \N. 
\end{equation*}
 A dual version  of the troplactic algebra, denoted $\dAPlc_I$, exists (Theorem \ref{eq:d.plac.Alg.2}) and is perfectly  compatible  with the   \cmirr ing map~
\eqref{eq:int.mir.2} below.

An immediate  consequence of the equivalence  $ u \kcong v $ of two words in  $\plcM_I$
is that the lengths $\mxlen{\tA_I}{u}$ and  $\mxlen{\tA_I}{v}$ of longest nondecreasing subsequences (subwords)
in  $u$ and $v$ are the same. Furthermore,
\begin{equation}\label{eq:int.clk}
\mxlen{\tA_J}{u} = \mxlen{\tA_J}{v} \quad \text{for every convex sub-alphabet $\tA_J$ of $\tA_I$,} \tag{CLK}
\end{equation}
 which happens  due to  correspondence of $\plcM_I$  with semi-standard Young tableaux (Proposition \ref{prop:tab-to-cloc}).
The converse implication, however, does not hold in general, which leads  to defining the \textbf{\kmon} $\clkM_I=\CLK(\tA_I)$ as the presented monoid $\tA_I^*/_{\clcong}$, a moderate coarsen of ~$\plcM_I$. That is, $\tA_I^*/_{\clcong}$ is the free monoid $\tA_I^*$  subject to the congruence   $ \clcong $, determined by the relations \eqref{eq:int.clk}. Then,  the monoid homomorphism
\begin{equation*}\label{eq:int.plc.to.clk}
  \plktoclk: \plcM_I \ONTO  \clkM_I 
\end{equation*} is surjective and preserves the congruence $\kcong$.

As expected, the homomorphism  $\plktoclk$ has a critical role in the construction of our  linear representations of the plactic monoid, obtained  from those of $\clkM_I$.  Nevertheless,   sometimes,  $\clkM_I$ appeared to be too
coarser, which  leads us to pursuing additional representable monoids that respect the congruence ~$\kcong$. To receive such a monoid, we restrict the framework to finitely  generated monoids~ $\tA_n^*$ and   introduce the \textbf{\cmirr ing} of letters \begin{equation*}\label{eq:int.mir}
\lcmp{n}{\lt_{\ell}} :=  \lt_n \lt_{n-1} \cdots \lt_{n-\ell+2} \; \lt_{n- \ell} \cdots \lt_1 , \qquad \lt_\ell \in \tA_n,
\end{equation*}
which induces   the  \cmirr ing
\begin{equation}\label{eq:int.mir.2}
  \lcmp{n}{w}  :=   \lcmp{n}{\lt_{\ell_1}} \cdots \lcmp{n}{\lt_{\ell_m}},  \qquad w = \lt_{\ell_1} \cdots \lt_{\ell_m} , \tag{CMR}
\end{equation}
over whole $\tA^*_n$. It also determines the  (lexicographic) order preserving surjective homomorphism
\begin{equation*}\label{eq:int.mir.3}
\pMap_1: \tA_n^* \INTO \sMon{(\tA_n^*)}{1} \subset  \tA_n^*, \qquad w \mTo \lcmp{n}{w}. 
\end{equation*} The homomorphism $\pMap_1$ extends inductively to a chain of endomorphisms
$\pMap_i: \sMon{(\tA_n^*)}{i-1} \Into \sMon{(\tA_n^*)}{i}$ of induced free monoids.

Our  \cmirr ing   construction respects the congruence $\kcong$ (Theorem~ \ref{thm:coplc2plc}). Thus, for the finitely generated plactic monoid $\plcM_n$,  the monoid homomorphism  $$\pMap_1: \plcM_n \INTO \sMon{(\plcM_n)}{1} \subset  \plcM_n $$ is injective --- a monoid  embedding.  On the other hand,  $\pMap_1$ does not preserve the equivalence of  $\clcong$  with the relations  \eqref{eq:int.clk}, but it  determines  the  equivalence $\cclcong$,  defined by
\begin{equation}\label{eq:int.co.clk}
u \cclcong v  \Iff \pMap_1 (u) \clcong \pMap_1(v).  \tag{CCLK}
\end{equation}
   The \textbf{\ckmon} $\coclkM_n = \CCLK(\tA_n)$  is constituted as   $\coclkM_n = \tA_n^* / _{\cclcong} $, that is
    accompanied   with
the surjective monoid homomorphism
 $\plktocclk: \plcM_n \onto{1}  \coclkM_n $, providing a second coarsening  of $\plcM_n$.  

The main goal of this paper is to construct  linear representations of the plactic monoid $\plcM_n$ in terms of tropical representations of the \kmon\ $\clkM_n $ and the \ckmon\ $\coclkM_n$. Therefore, a special focus is given
to  formulating  the correspondences  between $\plcM_n$ to $\clkM_n $ and  $\coclkM_n$ by relying upon  Young tableaux.

We start by  describing the combinatorial structure of the \kmon \ $\clkM_n$  in terms of a troplactic algebra~ $\APlc_I$, obtained by using tropical matrices.  Besides their  conventional algebraic meanings, these matrices  are also combinatorial entities,  corresponding  uniquely to weighted digraphs. As such,  they  intimately compose graph theory in tropical algebraic methodologies,  exhibiting   a useful interplay between algebra and combinatorics. The latter  plays a major role in
theoretical algebraic studies \cite{ABG,IR2} and in applications to combinatorics \cite{Butkovic,Kirillov}, as well as in semigroup representations~ \cite{IDMax,IzhakianMargolisIdentity} and automata theory ~\cite{Simon,Simon2}.

This mutual connection  allows for producing  a special class of tropical matrices that generate the finite tropical plactic  algebra ~$\mfA_n$  (Theorem ~\ref{thm:plac.mat.relations}) along with   recording lengths of the longest pathes in digraphs.  In turn, these paths encode the longest nondecreasing sequences of  the represented  elements  (Lemma~ \ref{klem:clk.represntation}).
 As being a troplactic algebra, the  multiplicative submonoid  $\MPlc_n$ of $\mfA_n$  immediately  admits the Knuth relations~ \eqref{eq:knuth.rel}.  More precisely, $\MPlc_n$ is isomorphic to the \kmon\ ~$\clkM_n$ (Theorem ~ \ref{thm:clk.represntation}). Thus, it  introduces  the
faithful tropical linear representation
$$ \clkrep:\clkM_n  \ISOTO \MPlc_n.$$
This isomorphism  yields an efficient algorithm for computing the maximal lengths of all nondecreasing subwords of $w \in \tA_n^+$ with respect to every  convex sub-alphabet of $\tA_n$ (Algorithm \ref{algr:clk}).
  Similarly, for the  \ckmon\ $\coclkM_n$,   we obtain the monoid isomorphism
$ \dclkrep:\coclkM_n  \Isoto \mir{\MPlc_n}$, whose image~ $\mir{\mfA_n}$ is a dual troplactic matrix algebra $\dAPlc_I$ (Theorem \ref{thm:co.plac.mat.relations}). Both isomorphisms   $\clkrep$ and $\dclkrep$ are allocated  with tropical characters that specify characteristic  invariants.

An  immediate result of the existence of  the isomorphisms  $\clkrep$ and $\dclkrep$ is that both monoids $\clkM_n$ and~
$\coclkM_n$ admit all the semigroup identities satisfied by  the monoid $\TMat_n(\Trop)$ of tropical triangular matrices (Corollaries ~\ref{cor:clk.id} and \ref{cor:co.clk.id}).  A particular form of these semigroup identities  is constructed as  \begin{equation}\label{eq:00} \Id_{(C,p,q)}: \ \tlw  \ds x \tlw  =\tlw  \ds y
\tlw, \tag{SID}
\end{equation}  where $ \tlw :=   \tlw_{(C,p,q)}$ is a fixed word over the variables  $C = \{ x, y\}$ that  contains as  factors all the possible words of length~$q$ over $C$ in which no letter appears sequentially more than $p$ times, and such that the words $\tlw \, x \, \tlw$ and $\tlw \, y \, \tlw$ also satisfy   this law \cite{trID,IDMax}.
With this form in place, $\TMat_n(\Trop)$ satisfies the identities \eqref{eq:00} with  $p = q =  n-1$ by letting   $x = uv$ and $y =vu$.
The identities \eqref{eq:00} generalize the Adjan's identity of the bicyclic monoid \cite{Adjan}, which  is also faithfully represented by tropical matrices \cite{IzhakianMargolisIdentity}.  Also,  $\TMat_n(\Trop)$  satisfies a recursive version of Adjan's identity \cite{OKID}.

With tropical representations of the monoids $\clkM_n$ and $\coclkM_n$ in hand, our next goal is  to formulate the surjective monoid homomorphisms $\plktoclk: \plcM_n \onto{1}  \clkM_n $ and $\plktocclk: \plcM_n \onto{1}  \coclkM_n$ explicitly,  and to explore  their  properties as  reflected in the representations  $ \clkrep:\clkM_n  \Isoto \MPlc_n$ and
$ \dclkrep:\coclkM_n  \Isoto \mir{\MPlc_n}$. 
   To this end, we  rely upon the  well known one-to-one correspondence between the plactic monoid $\plcM_n$ and semi-standard Young tableaux $\Tab_n$ (e.g., see \cite{Las1}). In a sense, tableaux are graphical patterns that  accommodate symbols, i.e., letters of the associated  plactic monoid,   with a deep combinatorial meaning. Nevertheless, an additional machinery is required to canonically  frame Young tableaux by tropical matrices and to convert visual-combinatorial information to suitable numerical-algebraic data.

 This converting  machinery is provided by  \textbf{$n$-configuration tableaux} $\CTab_n$. That is, tableaux of fixed isosceles  triangular shape that  contain  non-negative integers subject to  certain structural laws, called configuration laws  (Definition ~\ref{def:ctab}). Configuration tableaux correspond bijectively to semi-standard Young tableaux  (Theorem ~ \ref{thm:configuration}) and are endowed with  a self  implementation of the Encoding Algorithm~ \ref{algr:conf} that  simulates   the Bumping Algorithm ~\ref{algr:1} of  $\Tab_n$. Their  fixed
 shape enables the introduction of a canonical reference system, employed to state their correspondence to tropical matrices in $\MPlc_n$ and  $ \mir{\MPlc_n}$;  thereby, to link tableaux  with  weighted digraphs.
 In this framework,  encoding a letter  $ \lt_\ell \in \tA_n$ in  a configuration tableaux $\ctb_w \in \CTab_n$ is interpreted  as multiplying the matrix $\clkrep(w)$  by the  matrix $\clkrep(\lt_\ell)$ in $\MPlc_n$ or,  dually,
 $\dclkrep(w)$   by  $\dclkrep(\lt_\ell)$ in $ \mir{\MPlc_n}$. Accordingly,  $\Tab_n$ and $\CTab_n$ are considered as multiplicative monoids whose operations are induced by letter encoding; hence, our tableau correspondences  are realized as monoid homomorphisms.

With all desired  components at our disposal, the tropical linear representation $$ \plcrep_n: \plcM_n \ONTO {\MPlc_n} \times \mir{\MPlc_n} $$ of the plactic monoid $\plcM_n$
is obtained by composing the monoid homomorphisms of our main objects (Theorem ~\ref{thm:plc.rep}),  summarized by the diagram
 
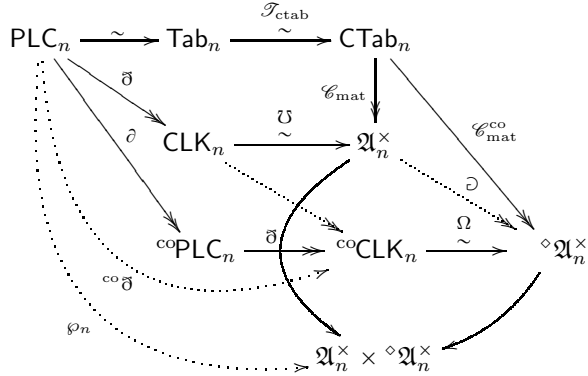
\begin{equation*}\label{eq:diag.1} \begin{gathered}
 \xymatrix{
 \PLC_n  \ar@{->>}[rd]^\plktoclk  \ar@{->>}[rdd]^\pMap \ar@{->}[r]^\Sim \ar@/_5pc/@{..>}[rrddd]_{\plcrep_n}
  \ar@/_4.55pc/@{..>}[rrdd]_{\co{\hskip 2mm \plktoclk}} & \Tab_n  \ar@{->}[r]^{\SIM{\tabctab}} &  \CTab_n \ar@{->>}[rdd]^\ctabcomat \ar@{->>}[d]_\ctabmat&     &
 & \\ & \CLK_n \ar@{..>>}[dr] \ar@{->}[r]^{\SIM{\clkrep}} & {\MPlc_n} \ar@/_3pc/@{->}[dd]  \ar@{..>>}[dr]^\matcmat  &
\\  & \CPLC_n \ar@{->>}[r]^{\plktoclk \quad} &\CCLK_n  \ar@{->}[r]^{\SIM{\dclkrep}}  & \quad \mir{\MPlc_n} \ar@/^1pc/@{->}[dl]  & .
\\  & &  {\MPlc_n} \times \mir{\MPlc_n}  &  
& &
}\end{gathered}
    \end{equation*}
This representation naturally induces a congruence on $\plcM_n$, and therefore also an equivalence relation on  tableaux in $\Tab_n$ and $\CTab_n$.

In the case of the plactic monoid of rank $3$, $\plcM_3$,  the semigroup representation
$ \plcrep_3: \plcM_3 \Isoto {\MPlc_3} \times \mir{\MPlc_3}$
is faithful  (Theorem~ \ref{thm:plc.3.rep}), more precisely, it is  an isomorphism.  Consequently,  we conclude that
$\plcM_3$ admits all the semigroup identities satisfied by the monoid $\TMat_3(\Trop)$ of $3 \times 3$ triangular tropical matrices (Corollary~ \ref{cor:plc.id}); in particular, the identities  $\Id_{(C,2,2)}$ in \eqref{eq:00}. Furthermore,  relying  upon  the correspondence between reversal of words and transposition of tableaux,  standard Young tableaux $\STab_n$ are shown to be faithfully realizable by tropical matrices (Theorem \ref{thm:stab.rep}).
As tableaux in $\STab_n$ bijectively correspond to elements of the symmetric group $S_n$, a tropical realization of $S_n$ is obtained, linking the theory to Heacke algebras.

Tropical representation theory turns out to be
applicable for studying characteristic properties of semigroups in places that the use of  classical
representation theory is limited.
It  provides an alternative
approach for realization and for the exploration of  algebraic-combinatorial objects; especially, their semigroup identities.
These identities are of special interest in the theory of semigroup varieties \cite{SV}, e.g., in  finitely generated semigroups of polynomial growth \cite{Grom,Shn}.
Applications of the  ubiquitous plactic monoid in group representations are well known~\cite{Green},  e.g.,  in computing products of Schur functions in~$n$ variables which  are the irreducible polynomial characters of the general linear group $\operatorname{GL}_n(\Comp)$, cf.  \cite{Littlewood}. Tropical linear representations pave the way to studying combinatorial aspects in classical representation theory, including a geometric perspective  via projective modules \cite{IJK2,merl10}.
This paper ease the use of tropical representation by providing a solid  bridge  between tropical algebra and classical representation theory.

\subsection*{Paper outline}
\sSkip

  For the reader convenience,  this paper  is designed as a self-contained document.
 In ~\S\ref{sec:preliminaries}, we bring all the relevant definitions and properties of objects to be used in this paper,  including  basic examples.
 In \S\ref{sec:troplactic.alg}, we introduce and study the new structure of tropical plactic algebra and its core object, called forward  semigroup (Definition \ref{def:forwardMon}).
  A brief overview on tropical matrices and their associated digraphs opens~ \S\ref{sec:trop.matrices}, followed by the characterization of tropical corner matrices, providing faithful linear representations of forward semigroups.
In \S\ref{sec:troplactic.m.alg} we employ  corner matrices to construct an explicit troplactic algebra $\mfA_n$ and to  establish  the linkage between diagraphs to nondecreasing subwords. The entire~ \S\ref{sec:rep.clocktic} is devoted to the introduction of  the \kmon \ $\clkM_n$ and the \ckmon \ $\coclkM_n$, including their linear representations by $\mfA_n$ and~ $\mir{\mfA_n}$, respectively.
Young tableaux and configuration tableaux are discussed  in ~\S\ref{sec:configuration.tab}, with a special emphasis on numerical functions that allow their mapping to
matrices in $\mfA_n$. Finally, in~ \S\ref{sec:plc.rep} we compose our various components to obtain  representations and co-representations of configuration tableaux which eventually result as linear representations of the plactic monoid.

To  help for a better understanding, our exposition involves many diagrams and examples, including of pathological cases.

\section{Preliminaries}\label{sec:preliminaries}

To make this paper reasonably self contained,
this section, expect parts of \S\ref{ssec:mirror},  recalls the relevant notions and terminology,   as well as  definitions  and properties of the algebraic structures to be used in the paper, starting with our special notations.

\subsection{Notations}\label{ssec:notation}
\sSkip

Unless otherwise is specified,  the capital letters $I,J$ denote subsets of the neutral numbers $\N := \{ 1,2, \dots \}$, while $\N_0$ stands for $\{ 0,1,2, \dots \}$.  The finite set $\nnn$ is often denoted by $N$, for short.  By ``ordered'' we always mean  totaly ordered.
\begin{definition}\label{defn:convex} A  subset $J$ of an ordered set $I$ is called \textbf{convex}, written $J \cnxsset I$,
  if for any $i,j \in J$, every $k \in I$ such that  $i < k < j$, also belongs to $J$. We write $\sitoj{i}{j}$ for the convex subset of $I$ determined by $i \leq j$.
\end{definition}
In other words, a convex subset is an ``interval'', could be empty or a singleton. For a given $n$, we   define  $i':= n-i +1$,  for $i = 1, \dots,n$,  ordered now reversely as
$n' < (n-1)' < \cdots < 1'$. Then $\sitoj{j'}{i'}$ has the same number of elements as $\sitoj{i}{j}$ has, and is again convex in $N$.

 We write $L_m = \squ{\ell_1, \ell_2, \dots, \ell_m}$ for a sequence of elements $\ell_t$ taken from  $\{ 1, \dots, n\}$,  where  $t =1,\dots,m$. This notation  means that the indexing order of elements is preserved.   We  denote the sequence $\squ{1,\dots, n}$ by $\itoj{1}{n}$, and write $\itoj{i}{j}$ for the subsequence $\squ{i, i+1, \dots, j}$ of $\itoj{1}{n}$.

\begin{itemize}\dispace
\item A subsequence $S$ of $L_m$ is notated  as    $S \sqsubseteq L_m $, e.g., $\itoj{i}{j} \sqsubseteq \itoj{1}{n}$, for $1 \leq i \leq j \leq n$.
 \item $\lensq{S}$ denotes the \textbf{length} of $S \sqsubseteq L_m$, i.e., the number of its elements.
\item  $S_k$ denotes a  subsequence of length $k$, $k \geq 0$, where $S_0$ stands for the empty sequence.


  \item The notation $\prod_{s \in S_k} \lt_s  $ stands for the product  $\lt_{s_1} \cdot \lt_{s_1} \ds \cdots  \lt_{s_k}$, respecting the indexing of the sequence $S_k =  \squ{s_1, \dots, s_k } $.

 \item A subsequence $S_k = \squ{s_1, s_2,   \dots , s_k} $ of length $k$ of $L_m $ is
 \begin{itemize}
   \item non-decreasing, denoted $\incS_k \sqsubseteq L_m $,  if $s_1 \leq s_2 \leq \cdots \leq s_k$,

   \item increasing if $s_1 < s_2 <  \cdots < s_k$,

   \item non-increasing, denoted $\decS_k \sqsubseteq L_m $,  if $s_1 \geq s_2 \geq  \cdots \geq s_k$,
 \item decreasing  if $s_1 > s_2 >
   \cdots > s_k$.

 \end{itemize}

 \item $\incS$ and  $\decS$ denote respectively non-decreasing and non-increasing subsequences of an arbitrary length.

 \item  We denote the set of all non-decreasing subsets $\incS \sqsubseteq L_m$ of $L_m$ by $\NDC{L_m}$.
     \item
 A non-decreasing subsequence $\incS \sqsubseteq L_m$ is said to be \textbf{maximal} in $L_m$ if $L_m$ has no other non-decreasing subsequence $\incR \sqsubseteq L_m $ such that $\incS    \sqsubsetneq \incR $.
$\MNDC{L_m}$ denotes the subset of all maximal non-decreasing subsequences of $L_m$.
 \item
 $\NIC{L_m}$ and $\MNIC{L_m}$
are defined similarly for non-increasing subsequences.

\end{itemize}
In this paper we deal only with finite sequences, which are  also termed ``words'' in the context of semigroups.


\subsection{Free monoids}\label{ssec:semigroups}
\sSkip

 A \textbf{semigroup} $\tS:= (\tS, \cdot \,)$ is a set of elements together with an associative binary operation.
A~ \textbf{monoid} is a semigroup with identity element $e$.  Any semigroup $\tS$ can  be formally adjoined with an identity element~ $e$ by declaring that $e a = a e = a$ for
all $a\in \tS,$ so when dealing with multiplicative structures we work with
monoids.
We write $a^i$ for $a \cdot a \cdots a$ with $a\in \tS$
repeated $i$ times, and formally identify $a^0$ with~$e$,
when $\tS$ is a monoid.
An
element $o$ of $\tS$ is said to be \textbf{absorbing}, usually  identified as $\zero$,  if $oa = ao = o$ for
all $a \in \tS$. $\tS$ is a \textbf{pointed semigroup} if it has an
absorbing  element $o$.

 An Abelian semigroup
$\tS$ is \textbf{cancellative} with respect to
a subset $\TT \subseteq \tS$ if $ac = bc$ implies $a=b$ whenever
$a,b \in \tS$ and $c \in \TT$. In this case,
$\TT$ is  called a \textbf{cancellative subset} of $\tS$, and it generates a subsemigroup in $\tS$, also cancellative. Thus, we usually assume that $\TT$ is a subsemigroup. A semigroup ~$\tS$ is
\textbf{strictly cancellative}  if $\tS$ is cancellative with
respect to itself.
The term ``\textbf{congruence}'' refers to an equivalence relation that respects the operation of its underlying semigroup  carrier.

 A \textbf{partially ordered semigroup}
is a semigroup~$\tS $ with a partial order $\leq$ that respects the semigroup operation
\begin{equation}\label{ogr1} a \le b \quad \text{implies}\quad ca
\le cb, \  ac
\le bc \end{equation} for all elements $a,b,c \in \tS$. A semigroup
$\tS $ is   \textbf{ordered} if the order $\leq$ is a total order.
\pSkip

We recall some basic definitions from \cite{trID}, for the
reader's convenience.
%
As customarily, $\varX^*_I$ denotes the free monoid of finite sequences
 generated by a countably infinite totaly ordered set $\varX_I: = \{ \a_\ell \ds : \ell \in I  \}$, called \textbf{alphabet},   of \textbf{letters} $\a_1, \a_2,
\a_3, \dots$ An  element  $w \in \varX^*_I$ is  called a \textbf{word} and, unless it is empty, is written uniquely as
\begin{equation}\label{eq:word} w =  \a_{\ell_1} ^{q_1} \cdots
\a_{\ell_m}^{q_m} \in \varX^*_I, \qquad \ell_t \in I, \ q_t \in
\N,
\end{equation}
where $ \a_{\ell_t} \neq  \a_{\ell_{t+1}}$ for every $t$.
We assume that the empty word, denoted  $e$, belongs to $\varX^*_I$, serving as the identity.
 As customarily,   $\varX^+_I$ stands for the free sub-semigroup obtained from $\varX^*_I$ by excluding the empty word $e$. When $|I| = n$ is finite,  we write $\tA_n$
for the finite alphabet $\tA_I = \{\a_1, \dots, \a_n \}$.
A word is read from left to right,  to    distinguish this reading direction, when needed,  we write $\std{w}$ for $w$.

The free monoid
$\tA_I^*$ is endowed with the familiar \textbf{lexicographic order}, denoted $<_\lex$,  induced by the order of $\tA_I$.
A sub-alphabet $\tA_J$ of $ \tA_I$ is said to be \textbf{convex}, written $\tA_I \cnxsset \tA_I$, if $J \cnxsset I$, i.e., $J$ is a convex subset of $I$ (Definition \ref{defn:convex}). We denote by
$\tA_{\itoj{i}{j}}$ the convex sub-alphabet
$\{\a_i , \dots, \a_j \}$ of $\tA_n = \tA_{\itoj{1}{n}}$.

The \textbf{length} of a word $w \in \varX^*_I$  of the form  \eqref{eq:word}  is defined (by the standard summation) as
 $\wlen(w) := \sum_{t =1}^m  q_t \, .
$  Then, since a word $w$ is a finite sequence of letters,  $\wlen(w)$ is well defined and finite, with $\wlen(w)$ iff $w =e $ is the empty word $e$.
 A word $w \in \varX^*_I $ is called
\textbf{$k$-uniform} if each of its letters appears exactly $k$ times. We say that $w$ is \textbf{uniform} if $w$ is $k$-uniform for some $k$.

A word $u \in \varX^*_I$ is a \textbf{factor} of $w\in
\varX^*_I$, written $\factor {u}{ w}$, if $w = v_1 u v_2$ for
some $v_1, v_2 \in \varX^*_I$.
When $w = v_1 v_2$, 
the factors $v_1$ and $v_2$ are respectively called the
\textbf{prefix} and \textbf{suffix} of~$w$.  If $v_1$ is of length $k$ we say that $v_i$ is the \textbf{$k$-prefix} of $w$, and denote it by  $\pre_k(w)$.
Similarly, if $\wlen(v_2) = k$, we say that $v_i$ is the \textbf{$k$-suffix} of $w$, and denote it by  $\suf_k(w)$.

%
%
A word $u$  is a \textbf{subword} of~$w$, written $u \wrdsset w$,
if $w$ can be written as $w = v_0 u_{1} v_1 u_{2} v_2 \cdots u_{m}
v_m$ where $u_i$ and $v_i$ are words (possibly empty) such that $
u = u_1u_2 \cdots u_m $, i.e., the $u_i$ are factors of $u$. Clearly, any factor of $w$ is also a subword, but not conversely.

The following notion of $n$-power words was introduced
in its full generality in \cite[\S3.1]{trID}. However,  for the purpose of the present paper, it suffices to consider a restricted version of this notion, as described below.
As seen later, these $n$-power words are the cornerstone of our systematic construction of semigroup identities.
\begin{definition}\label{def:npw} Let $C :=C_m$ be a finite (nonempty) alphabet, and let $n \in \N$.
An~\textbf{$n$-power word $\tlw := \tlw_{(C,p,n)}$}  is a nonempty word in $C_m^+$ such that:
\begin{enumerate} \ealph \dispace
  \item Each letter $\a_\ell \in C_m$ may appear in $\tlw$ at most $p$-times sequentially, i.e., $\a_\ell^q \nmid \tlw$ for any  $q >  p$ and $\a_\ell \in C_m$;
  \item Every  word $u \in C^+_m$ of length $n$ that satisfies rule (a) is a factor of $\tlw$.
\end{enumerate}
An $n$-power word is \textbf{uniform} if it is uniform as a word.

\end{definition}
\noindent An $n$-power word needs not be unique in $C^+_m$, and different $n$-power words may  have  different length. Often, $n$-power words $\tlw_{(C,p,n)}$  can be concatenated to a new $n$-power word.

 In the present paper, in the view of Theorem \ref{thm:2id} below, we are mostly interested in the case of $n$-power words over the 2-letter alphabet  $C_2 = \{ x,y\}$. As we consider powerwords as generic words, we call the letters $x,y$  \textbf{variables}.
\begin{example}\label{exmp:2-word}
Suppose $C = \{ x, y \} $.
\begin{enumerate} \eroman \dispace
  \item  $ \tlw_{(C,2,2)} = yx^2y^2x$ is a $2$-power uniform word  of length $6$.   \pSkip

  \item
$ \tlw_{(C,3,3)} = \lllt $
 is a  $3$-power uniform word of   length $10$.

\end{enumerate}

\end{example}
For more details  on $n$-power words see \cite{trID,IDMax}.

\subsection{Co-mirroring and  reversal}\label{ssec:mirror}
\sSkip Recall that the binary operation of a semigroup is associative.
\begin{definition}\label{def:orderedMonoid}
A \textbf{semigroup homomorphism} is a map
$\phi : \tS  \To \tS' $  that preserves the semigroup operation, i.e.,
$$\phi(a \cdot b) = \phi(a) \cdot \phi(b)$$
for all
$a$ and $b$ in $\tS$.
A \textbf{monoid homomorphism} is a semigroup homomorphism for which
     $\phi(e) = e'$.
   It is an
 \textbf{endomorphism} when $\tS = \tS'$.
\end{definition}

A \textbf{presentation} of a monoid $\tM =  (\tM, \cdot \,)$ (resp. semigroup) is a description of $\tM$ in terms of a set of generators $\tA_I:= \{ a_\ell \ds | \ell \in I\} $   and a set of relations $\Xi \subset \tA_I^* \times \tA_I^* $ on the free monoid $\tA_I^*$ (resp. on the free semigroup $\tA^+_I$) generated by $\tA_I$. The monoid $\tM$ is then presented as the quotient $\tA_I^*/_\Xi$ of the free monoid $\tA_I^* $ by the set of relations $\Xi$, i.e.,
by a monoid  homomorphism
$$ \phi : \tM \TO \tA_I^*/_\Xi, \qquad \Xi \subset \tA_I^* \times \tA_I^*.$$
When $|I| =n$ is finite, we say that $\tM$ is \textbf{finitely presented}.

\begin{definition}\label{def:wordreversing}
The \textbf{reversal}  of a word $\std{w}= \lt_{\ell_1} \lt_{\ell_2} \cdots \lt_{\ell_m}$ of length $m$ in $\tA^*_I$, $\ell_i \in I$,  is the word
$\rvs{w}= \lt_{\ell_m} \lt_{\ell_{m-1}} \cdots \lt_{\ell_1}$ of the same length $m$.
\end{definition} \noindent
The reversal $\rvs{w}$ of a word $\std{w}$ is therefore the rewriting of $\std{w}$ from right to left, while we formally  set $\rvs{e} := e$. It defines a bijective map
$$ \rev: \tA^*_I \TO \tA^*_I, \qquad \std{w} \Mto \rvs{w}, $$
which is not a monoid homomorphism, since  $\rev(uv) \neq \rev(u) \rev(v)$.  However,  $\rev(uv) = \rev(v) \rev(u)$ for any $u,v \in \tA_I^*$.

\pSkip

Restricting our ground alphabets to finite alphabets,  we introduce the following operation:
\begin{definition}\label{def:mirror}
 The \textbf{\cmirr} of a letter $\lt_\ell$ in a finite alphabet $\tA_n$  is defined to be  the word
 \begin{equation}\label{eq:letter.mir}
\lcmp{n}{\lt_{\ell}} :=  \lt_n \lt_{n-1} \cdots \lt_{n-\ell+2} \lt_{n- \ell} \cdots \lt_1 = \prod^1_{\arrs{t = n}{ t \neq n-\ell+1}} \hskip -4mm \lt_t.
\end{equation}
(The \cmirr\ of the empty word $e$ is formally set to be $e$.)

The \textbf{\cmirr} of a word $w = \lt_{\ell_1} \cdots \lt_{\ell_m}$ in $\tA_n^*$ is the word defined as
\begin{equation}\label{eq:word.mirr}
\lcmp{n}{w} :=  \lcmp{n}{\lt_{\ell_1}} \cdots \lcmp{n}{\lt_{\ell_m} },
\end{equation}
written $\lcmp{}{w}$ when $n$ is clear from the context.
\end{definition}
Accordingly, for the concatenation $w = uv$ of two words  $u, v$ in $\tA_n^*$ we then have
$$\lcmp{n}{w} := \lcmp{n}{u}\lcmp{n}{v},$$
where $\wlen(\lcmp{n}{w}) = (n-1)\len{w}$ for any  $w \neq e$.  Thus, the \cmirr ing  of words determines a monoid endomorphism
\begin{equation}\label{eq:mir.map}
\mirM: \tA_n^*  \TO  \tA_n^*, \qquad w \mTo \lcmp{}{w},
\end{equation}
with $e \mTo e$.
\pSkip

\begin{remark}\label{rmk:mir.smon} The \cmirr ing of a finite alphabet $\tA_n= \{ \aa_1, \dots, \aa_n\}$ introduces a submonoid in the free monoid $\tA_n^*$, which we denote by $\sMon{(\tA_n^*)}{1}$, whose generators $\aa'_\ell = \lcmp{n}{\lt_{\ell}} $ are again (totaly) ordered
as $ \aa'_1 < \cdots < \aa'_n$. In fact, as can be seen form \eqref{eq:letter.mir}, this order is just the lexicographic order $<_\lex$ of $\tA_n^*$, and thus is compatible with the initial order of ~$\tA_n$.

Applying the \cmirr ing  map \eqref{eq:mir.map} inductively, we have a chain of submonoids, a filtration, $$\tA_n^*  \ds \supset \sMon{(\tA_n^*)}{1} \supset \sMon{(\tA_n^*)}{2} \ds \supset \sMon{(\tA_n^*)}{3} \ds \supset \cdots , $$
together with the surjective homomorphisms
$$\tA_n^*  \Onto{3}{\pMap_1} \sMon{(\tA_n^*)}{1} \Onto{3}{\pMap_2} \sMon{(\tA_n^*)}{2} \Onto{3}{\pMap_3} \sMon{(\tA_n^*)}{3} \Onto{3}{\pMap_4}  \cdots . $$
Each surjection $\pMap_i: \sMon{(\tA_n^*)}{i-1} \onto{1} \sMon{(\tA_n^*)}{i}$ is also  injective,  preserving  the lexicographic order of $\tA_n^*$,  and thus it is an isomorphism.
\end{remark}

Note that $\pMap_i$ does not necessarily preserve relations nor presentations over $\tA_n^*$.



\subsection{Semigroup identities}\label{sec:3.2} \sSkip
A (nontrivial) \textbf{semigroup identity} is a formal equality
 of the form  \begin{equation}\label{eq:Id}
\Id : u=v,
\end{equation} where
$u$ and $v$ are two  different (nonempty) words in the free semigroup $\varX^+_I$, cf. Form
\eqref{eq:word}. For a monoid
identity, $u$ and $v$ are allowed to be the empty word as well.
%
Therefore, a semigroup identity $\Pi$  determines a single relation  $\Pi \in  \tA^+_I \times \tA^+_I$
on the free semigroup $\tA^+_I$.

 We say that an identity $\Id: u =v $ is an
\textbf{$n$-variable identity} if $u$ and $v$ involve together exactly $n$ letters of $\varX_I$.
An identity $\Id$ is  said to be
\textbf{balanced} if the number of occurrences of each letter $\a_i\in \varX_I$ is the same in $u$ and in~$v$.  $\Id$ is called \textbf{uniformly balanced} if furthermore the words $u$ and~$v$ are $k$-uniform for some $k$.    The \textbf{length} $\ell(\Id)$  of $\Id$ is defined to
be $\ell(\Id) :=  \max\{\ell(u), \ell(v)\}$, clearly $\ell(u) = \ell (v)$, when  $\Id$ is balanced.

%

 A semigroup $\tS := (\tS, \cdot \;)$ \textbf{satisfies the
semigroup identity}~\eqref{eq:Id} if $$\text{ $ \phi(u)= \phi(v)$ \quad for every homomorphism $\phi
:\varX^+_I \To \tS$.}$$
Concerning the existence of identities it is suffices to consider $2$-variable identities:

\begin{theorem}[{\cite[{Theorem 3.10}]{trID}}]\label{thm:2id} A semigroup  $\tS:= (\tS, \cdot \;)$ that satisfies an
 $n$-variable identity $\Id: u = v$, for $n \geq 2$,  also satisfies  a refined
$2$-variable identity $\htId : \htu =\htv$ of exponents $\{1,2\},$ i.e., each letter in $\htu$ and in $\htv$ appears sequentially at most twice.
\end{theorem}

The $n$-power words  $\tlw_{(C,p,n)}$ (Definition \ref{def:npw}) are utilized to construct a class of nontrivial semigroup identities $\Id_{(C,p,n)}$. (We use the notation $\x$ and $\y$ to mark a specific instance
of the variables $x$ and $y$ in a given expression, although these instances stand for the same variables $x$ and $y$, respectively.)
\begin{construction}\label{construct:main} Let $ \tlw_{(C,p,n)}$ be a  uniform $n$-power word over $C = \{ x, y  \}$ such  that   the words
$\  \tlw_{(C,p,n)} \ds {x}
\tlw_{(C,p,n)} $ and $  \tlw_{(C,p,n)} \ds
 {y} \tlw_{(C,p,n)}$ are both $n$-power words over $C $.
 Define the 2-variable  balanced identity
\begin{equation}\label{eq:iduniv2}
\Id_{(C,p,n)}: \quad   \tlw_{(C,p,n)} \ds {
\x}
\tlw_{(C,p,n)} \ds =  \tlw_{(C,p,n)} \ds
 {\y}
\tlw_{(C,p,n)}.
\end{equation}
Then, substitute
\begin{equation}\label{eq:subt1} \text {$x:= \chx \ds{\chy} $ \dss{and}  $y :=
\chy \ds{\chx}$}
\end{equation} to refine \eqref{eq:iduniv2} to the uniformly balanced identity
\begin{equation}\label{eq:iduniv3}
\chId_{(\chC,p,n)}: \quad   {\chw}_{(\chC,p,n)} \ds {{\chx
\chy}} \chw_{(\chC,p,n)} \ds =  \chw_{(\chC,p,n)} \ds
 {{\chy \chx}}
\chw_{(\chC,p,n)},
\end{equation}
  where $\chC = \{ \chx,\chy\}$ and  $\chw_{(\chC,p,n)}$ is the word obtained from $\tlw_{(C,p,n)}$ by substitution~\eqref{eq:subt1}.
\end{construction}
Note that $\chw_{(\chC,p,n)}$ needs not be an $n$-power word  over $\chC = \{ \chx,\chy\}$, yet it satisfies law~ (a) of Definition~\ref{def:npw}. This summarize Construction \ref{construct:main}, when a semigroup $\tS$ satisfies the identity \eqref{eq:iduniv3}, we shortly say that
it admits the identity \eqref{eq:iduniv3} by taking $x = uv$ and $y = vu$ for all elements $u,v$ in $\tS$.

\begin{example}[{\cite[Example 3.9]{trID}}]\label{exmp:2-word:2}
Let  $C = \{x,y\}$.
\begin{enumerate} \eroman \dispace
  \item Using the uniform $2$-power  word $\tlw_{(C,2,2)} =\ll2$ as in
Example~\ref{exmp:2-word}.(i), we receive the  identity
\begin{equation}\label{eq:id2} \Id_{(C,2,2)}: \quad
    \ll2 \ds{\underline{x}}  \ll2 \ds =  \ll2 \ds{\underline{y}}
     \ll2 \ .
\end{equation}

  \item
Taking the uniform $3$-power word $ \tlw_{(C,3,3)} = \lllt$
in Example~\ref{exmp:2-word}.(ii), we obtain  the  identity
\begin{equation}\label{eq:id3} \Id_{(C,3,3)}: \quad
     \lllt \ds{\underline{x}}  \lllt \ds =  \lllt \ds{\underline{y}}
     \lllt \ .
\end{equation}
\end{enumerate}
By substituting $x:= \chx {\chy},$  $y :=
\chy {\chx}$, these identities \eqref{eq:id2} and \eqref{eq:id3} become uniformly balanced.
\end{example}

\subsection{Semirings and semimodules} \sSkip  Equipping a set of elements simultaneously  by two binary (monoidial)  operations,  one  obtains the following structure.

\begin{definition}\label{def:semiring}
A \textbf{semiring} $R := (R,+ \, ,\cdot \;)$ is a set $R$ equipped with two
binary operations $+$ and~$\cdot \;$, called addition and
multiplication, such that:
\begin{enumerate} \eroman
    \item $(R, + \,)$ is an Abelian monoid with identity element $\zero = \zero_R$;
    \item $(R, \cdot \, )$ is a monoid with identity element
    $\one = \one_R$;
    \item Multiplication distributes over addition.
\end{enumerate}
A semiring $R$ is called \textbf{idempotent} if $a + a = a$ for every $a\in R$, while it is said to be \textbf{bipotent} (or \textbf{selective}) if $a + b \in \{ a,b\}$ for any $a, b \in R$. When the multiplicative monoid  $(R, \cdot \, )$ is an abelian group, $R$ is called a \textbf{semifield}.
\end{definition}
Clearly, idempotence implies bipotence, but not conversely.
 A standard general reference for the structure of
semirings is~\cite{golan92}.

\begin{remark}\label{makesemi} Any  ordered monoid $(\tM, \cdot \; )$
  gives rise to
a semiring, where we define the addition $a+b$ to be $\max\{ a, b\}$.
  Indeed,
associativity is clear, and distributivity  follows from
\eqref{ogr1}. In the same way from  an abelian group we obtain a semifield.
\end{remark}

Monoid homomorphisms extend naturally to semirings.

\begin{definition}\label{def:smr.hom}
A \textbf{homomorphism} of semirings is  a map
$\srHom : R \To R' $  that preserves addition and multiplication.
To wit, $\srHom  $   satisfies the following properties for all
$a$ and $b$ in $R$:
\begin{enumerate} \eroman \dispace
    \item $\srHom(a + b) = \srHom(a) + \srHom(b)$;
    \item $\srHom(a \cdot b) = \srHom(a) \cdot \srHom(b)$;
    \item  $\srHom(\zero_R) = \zero_{R'}$.
\end{enumerate}
A \textbf{unital} semiring homomorphism is a semiring homomorphism that
preserves  the multiplicative identity, i.e., $\srHom(\one_R) = \one_{R'}$.
\end{definition} \noindent
In the sequel, unless otherwise is specified, our homomorphisms are assumed to be unital.
\pSkip

In analogy to the case of rings,  one defines module over semiring in a straightforward way.

\begin{definition}\label{def:semimodule}
An \textbf{$R$-module} $V$ over a semiring $R$ is an additive  monoid $(V,+)$ together
with a scalar multiplication $R\times V \to V$ satisfying the
following properties for all $r_i \in R$ and $v,w \in V$:
\begin{enumerate}\eroman
    \item $r(v + w) = rv + rw;$ \pSkip
    \item $(r_1+r_2)v = r_1v + r_2 v;$ \pSkip
    \item $(r_1r_2)v = r_1 (r_2 v);$ \pSkip
    \item $\one_R v =v;$ \pSkip
    \item $\zero_R v =\zero_V = r \zero_V .$
  \end{enumerate}
Modules over semirings are also called \textbf{semimodules}.
 \end{definition}

\section{Tropical plactic algebra}\label{sec:troplactic.alg}

On of the central ground objects of the current paper is the following well-known monoid \cite{Las2}:

\begin{definition}\label{def:plactic.mon}

The \textbf{plactic monoid} is the monoid $\plcM_I := \PLC(\tA_I)  $, generated by an ordered set of  elements $\tA_I : = \{\aa_\ell \ds | \ell \in I  \}$,  subject to the equivalence relation $\kcong$ (known as the  \textbf{elementary Knuth relations} or the \textbf{plactic relations}) defined
by
%
\begin{equation}\label{eq:knuth.rel}
  \begin{array}{cccc}
   \text{\KXa}. & \aa \; \cc \; \bb = \cc \; \aa \; \bb & \text{if} &  \aa \leq \bb < \cc , \\[2mm]
   \text{\KXb}. & \bb \;\aa \; \cc = \bb \; \cc \; \aa & \text{if} & \aa <  \bb \leq \cc , \\  \end{array} \tag{KNT}
\end{equation}
for all triplets  $\aa ,\bb,\cc \in \tA_n$.
Namely, $\plcM_I := \tA_I^*/ _{\kcong}$ with the identity element $e$.
When $|I| = n$ is finite, $\plcM_I$ is said to be of \textbf{rank} $n$, or finitely generated,  and is  denoted by $\plcM_n$. \footnote{In the literature, $\PLC_I$  is often assumed to finitely generated, however for large parts of our study are more general and finiteness is not required there.}
\end{definition}
\noindent  We write $\PLC_I$  for $\PLC(\tA_I)$ when the alphabet $\tA_I$ is arbitrary.   Henceforth,  we always assume that $e < \aa_\ell$ for all $\aa_\ell \in \plcM_I$. The relation $\kcong$ is a congruence on the free monoid $\tA_I^*$, e.g. see  $\tA_I^*$ \cite{Las1}, which is also denoted  by $\pcong$ to indicate its relevance to $\PLC_I$.


We aim for an algebra whose multiplicative structure comprises that of the plactic monoid, and in which the Knuth relations are  inferred from the algebraic operations. Towards this goal, we begin with  an axillary semigroup  structure, to be employed  for a construction of such a desirable algebra -- a special semiring. (In what follows we let $I \subset \N$ be a nonempty subset.)

\begin{definition}\label{def:forwardMon}
A \textbf{(partial) \fsmg } $\NDSg_I = \genr{ \felm_i \ds | i \in I }$ is a pointed semigroup with an  absorbing element~$\mfo$,  generated by the (partial) ordered set of elements \footnote{For simplicity, we assume a countable set of generators, where the  generalization to an arbitrary  ordered set of generators is obvious. Moreover, one can also generalize this structure by considering a partially ordered set of generators, but such theory is more involved.}
$\{\felm_i \ds|  i \in I \}$, subject to the  axiom
\begin{description} \eroman \dispace
 \item[\NDXa] \qquad$ \felm_j \felm_i  = \mfo$ whenever $\felm_j > \felm_i$.
\end{description}
When $I$ is finite with $|I| =n$, we write $\NDSg_n$ for $\NDSg_I$.
\end{definition}

Note that the semigroup  $\NDSg_I $  is not assumed to be a (partial) ordered semigroup, i.e., having an order preserving operation  $\felm_i \geq  \felm_j \Rightarrow \felm_i \felm_k \geq \felm_j \felm_k   $ (Definition \ref{def:orderedMonoid}), nor a cancellative semigroup, i.e., $\felm_k \felm_i = \felm_k \felm_j \not \Rightarrow \felm_i =  \felm_j$.  Usually, $\NDSg_I $ is non-commutative, since  otherwise we would always have $\felm_i \felm_j = \felm_j \felm_i = \mfo$, implying that  $\NDSg_I $  consists of $\{ \felm_i \ds | i \in I \} \cup \{ \mfo \}.$

\begin{remark} The elements of a forward semigroup   $\NDSg_I $ are realized as nondecreasing sequences over $\{ \felm_i \ds | i \in I \}$. Then it easy to verify that $\NDSg_I $  naively admits the Knuth relations \eqref{eq:knuth.rel}, since  in $\NDSg_I $  each term of   \eqref{eq:knuth.rel} equals $\mfo$.
\end{remark}

Clearly, the \fsmg\ $\NDSg_I$ is  a very degenerated version of the plactic monoid, but useful for the construction of semirings, as seen below in Construction \ref{construct:plc.alg}.

\subsection{Tropical plactic algebra}\label{ssec:troplactic.alg}\sSkip

We open this section by introducing a new  semiring structure --  a key object  in this paper -- that  encapsulates the plactic monoid (cf. Theorem \ref{thm:plc.Alg.1} below).

\begin{definition}\label{defn:trop.plc.alg}
A \textbf{tropical plactic algebra} $\APlc_I$ is a (noncommutative) idempotent semiring
$(\APlc_I, +, \cdot \,)$    with multiplicative identity element $\mfe$ and zero $\mfo$,  generated by an ordered set of elements \footnotemark[2]  $\{ \mfa_i \ds | i \in I \}$ subject to the  axioms (for every $\mfa \leq  \mfb \leq \mfc$ in $\{ \mfa_i \ds | i \in I \}$):
\begin{description} \eroman \dispace
 \item[\PXa] $ \mfa  = \mfe  + \mfa$;
  \item[\PXb] $\mfb \mfa = \mfa + \mfb$  when  $\mfb >  \mfa $;
 \item[\PXc] $\mfa (\mfb +  \mfc) = \mfa \mfb + \mfc;$
  \item[\PXd] $(\mfa + \mfb) \mfc = \mfa + \mfb \mfc.$
\end{description}
When $|I| =n$ is finite we write $\APlc_n$ for $\APlc_I$, and say that $\APlc_n$ is finitely generated and has \textbf{rank} $n$.
(For indexing matter, $\mfe$ is also denoted by $\mfa_0$.)
\end{definition}
We abbreviate our terminology and also  write  \textbf{troplactic algebra} for tropical plactic algebra. Arbitrary elements of $\APlc_I$  are denoted by Gothic letters $\mfu, \mfv, \mfa$, while $\mfa, \mfb, \mfc,$ are devoted for its generators.
When it is clear from the context, we  suppress the notation of the indexing set $I$ and  write $\APlc$ for $\APlc_I$, for short.

As $\APlc_I$ is an idempotent semiring, then $\mfu + \mfu = \mfu$ for all $\mfu \in \APlc_I$ (not only for  generators $\mfa \in \{ \mfa_i \ds | i \in I \}$).
While Axioms \PXa\ and \PXb\ are easily figured out, Axioms \PXc\ and \PXd\ are more complicated.  To better understand them, they are illustrated by the diagrams:
$$ \begin{array}{l|c|c|c|}   \cline{2-4}
     \mfa \mfb  & \square & \square &  \\[1mm]  \cline{2-4}
     \mfa \mfc  & \blacksquare & &  \square \\  \cline{2-4}
   \end{array} \dss = \begin{array}{l|c|c|c|}   \cline{2-4}
     \mfa \mfb  & \square & \square &  \\[1mm]  \cline{2-4}
      \mfc  &  & &  \square \\  \cline{2-4}
   \end{array}
$$
$$ \begin{array}{l|c|c|c|}   \cline{2-4}
     \mfa \mfc  & \square &  & \blacksquare \\  \cline{2-4}
     \mfb \mfc  &  &  \square &  \square \\  \cline{2-4}
   \end{array} \dss = \begin{array}{l|c|c|c|}   \cline{2-4}
     \mfa   & \square &  &  \\  \cline{2-4}
      \mfb \mfc  &  &\square &  \square \\  \cline{2-4}
   \end{array}
$$ in which the columns stand for products of the letters $\mfa, \mfb$ and $\mfc $ and the rows for the additive terms.

\begin{remark}\label{rmk:lex.order} Given a product $\mfu = \mfa_{\ell_1} \cdots \mfa_{\ell_m}$ in $\APlc_I$, with $\ell_1, \dots, \ell_m \in I$, we write $\cset{\mfu}$ for the realization of $\mfu$  as a word in the free semigroup $\tA^+_I$ over the set of symbols $\tA_I = \{ \mfa_i \ds | i \in I \}$.
 Then $\tA_I^+$  -- the multiplicative structure of $\APlc_I$ -- is equipped with a \textbf{lexicographic order}, determined by  order of the symbol  set   $\tA_I$. We denote this lexicographic order by  $<_ \lex.$
\end{remark}

Troplactic algebras appear naturally as semimodules over a semiring (Definition~ \ref{def:semimodule}), often  as a projective semimodule realized as projective matrices, where then the equalities in axioms \PXa--\PXd\ are taken with  respect to the appropriate setting.
A particular abstract construction of plactic algebra can be obtained  by the use of  \fsmg s (Definition \ref{def:forwardMon}).
\begin{construction}\label{construct:plc.alg} Let  $\NDSg_I = \genr{ \felm_i \ds | i \in I }$ be a \fsmg \ and $\delm$ an adjoint distinguished additively idempotent element  $\delm = \delm + \delm $ that commutes with each  element $\felm_i  $, i.e.,  $\delm\felm_i = \felm_i\delm$. Set  $\mfa_i  := \delm +  \felm_i$,   whose order is induced  by  the order   $\NDSg_I$, i.e.,
$$ \mfa_i < \mfa_j   \ds{\Iff} \felm_i  <  \felm_j,$$ to define a troplactic algebra $\APlc_I$ as the algebra generated by $\mfa_i$.
We often take $\delm = \mfe $ to be the multiplicative identity of $\APlc_I$.

In this setup, if $\mfa_i < \mfa_j $, then
   $\mfa_j \mfa_i = \delm(\mfa_i + \mfa_j)$.
 Indeed,   $\mfa_j \mfa_i = (\delm + \felm_j)( \delm + \felm_i) = \delm^2 + \delm\felm_i + \delm \felm_j = \delm( \delm + \felm_i + \felm_j) = \delm(\mfa_i + \mfa_j).$
Thus, for the case $\delm = \mfe$,   only axioms  \PXc \ and \PXd\ are enforced, while Axiom \PXa\ is obtained by construction and Axiom \PXb \ follows from the structure of $\NDSg_I$.

\end{construction}
Construction \ref{construct:plc.alg} turns out to be very useful in applications, especially for troplactic matrix algebra, as its arithmetics provides a natural additive decomposition for multiplicative expressions, followed from Axiom \PXb.  (Note that $\APlc_I$ is not a bipotent semiring and its addition is not necessarily   determined  as maximum.)

\begin{remark}\label{rmk:plac.1} Let $\mfa, \mfb, \mfc \in \{ \mfa_i \ds | i \in I \}$ be the generators of $\APlc_I$, and assume that  $\mfa \leq \mfb \leq \mfc$.
\begin{enumerate}\eroman \dispace
\item Declaring that  $\mfe < \mfa_i$ for all $i \in I$,   Axiom \PXa\ is then derived directly from Axiom \PXb. Furthermore, in this case, the equality  $\mfe = \mfa + \mfe $ implies that $\mfa = \mfe$. Indeed $\mfe = \mfa + \mfe = \mfa \mfe = \mfa$.
  \item   The multiplicative identity  $\mfe$ is also an additive  identity (by Axiom \PXa)
  $$ \mfa \mfe = \mfe \mfa = \mfa = \mfa + \mfe = \mfe + \mfa,$$
  and in particular
    it is idempotent with respect to both
    addition and multiplication, i.e., $$\mfe + \mfe = \mfe = \mfe \cdot \mfe \; .$$

  \item Axiom \PXa \ implies $$\mfa\mfb + \mfa = \mfa\mfb + \mfb =  \mfa\mfb,$$ since  $\mfa\mfb = \mfa(\mfb+ \mfe) = \mfa\mfb + \mfa$, and  $\mfa\mfb = (\mfa + \mfe)\mfb = \mfa \mfb + \mfb.$  Therefore, by idempotence of $\mfu = \mfa \mfb$,
      $$ \mfa\mfb + \mfa + \mfb = \mfa\mfb,$$
 which implies that
      $$ \mfa\mfb + \mfb \mfa  = \mfa\mfb,$$
and furthermore that
    \begin{equation}\label{eq:plc.power}
       \mfa^m + \mfa^{m-1} + \cdots + \mfa = \mfa^m.
  \end{equation}
      Inductively,  together with the use of Axiom \PXb, for an arbitrary product $\mfa_{\ell_1}\cdots  \mfa_{\ell_m} $ we obtain
     \begin{align*}
     \sum_{\incS \sqsubseteq L_m } \prod_{s \in \incS}\mfa_{s}  = \mfa_{\ell_1}  \cdots    \mfa_{\ell_m}, & \qquad  L_m = \squ{\ell_1, \dots, \ell_m}.
     \end{align*}
     Namely,  a decomposition of products to sums of non-decreasing terms.

\item When $\mfa = \mfb,$ Axiom \PXc\ reads as $\mfa \mfa + \mfa \mfc = \mfa \mfa + \mfc,$ and thus
by (iii) we have
\begin{equation}\label{eq:plc.aa+ac}\mfa^2 + \mfa \mfc + \mfc^2 = \mfa ^2 + \mfc + \mfc^2 = \mfa ^2 +  \mfc^2,\end{equation}
hence $(\mfa + \mfc)^2 = \mfa ^2 +  \mfc^2$.

\item When $\mfb = \mfc,$ Axiom \PXd\ reads as $\mfb \mfb + \mfa \mfb = \mfb \mfb + \mfa.$

\item
 If $\mfa \mfb = \mfe $, then, by Axioms \PXa, \PXb, and part (iii), we have
  $$\mfb \mfa = \mfb + \mfa = \mfb + \mfe  + \mfa + \mfe =  \mfb + \mfa \mfb  + \mfa + \mfa \mfb = \mfa \mfb + \mfa \mfb = \mfe + \mfe = \mfe,$$
  and thus also  $\mfb \mfa = \mfe$. Accordingly, when $\mfa \mfb = \mfe $ for all $\mfb > \mfa$,  all the elements of $\APlc_I$ are of the form $\mfa_\ell^k$, with $k \in \N.$



\end{enumerate}
\end{remark}

The main virtue of the troplactic algebra is that the (multiplicative) Knuth relations can be deduced  from its semiring structure. Therefore, multiplicatively, $\APlc_I$  has the structure of the plactic monoid ~$\plcM_I$ (Definition \ref{def:plactic.mon}), which in general is not degenerated.
\begin{theorem}\label{thm:plc.Alg.1} The troplactic algebra  $\APlc_I$ admits the Knuth relations \eqref{eq:knuth.rel}:
\begin{equation*}\label{eq:plac.Alg.2}
  \begin{array}{ccccccc}
   \text{\KXa}:&& \mfa \; \mfc \; \mfb = \mfc \; \mfa \; \mfb & \text{if} &  \mfa \leq \mfb < \mfc , & \qquad & \\
   \text{\KXb}:&& \mfb \;\mfa \; \mfc = \mfb \; \mfc \; \mfa & \text{if} & \mfa <  \mfb \leq \mfc ,   \end{array}
\end{equation*}
where  $\mfa, \, \mfb, \, \mfc \in \{ \mfa_i \ds| i \in I \}$, and thus mutiplictively forms a plactic monoid.
\end{theorem}
\begin{proof}
We employ Axiom  \PXa \ via Remark \ref{rmk:plac.1}.(ii).
\begin{enumerate} \dispace
  \item[\KXa:] By axiom  \PXb:  $$\mfa \mfb +  \mfa \mfc = \mfa(\mfb + \mfc) = \mfa \mfc \mfb,$$  while Axiom \PXa \  gives (cf. Remark \ref{rmk:plac.1}.(iii))
  $$\mfa \mfb +  \mfc = \mfa \mfb + \mfb + \mfc =  \mfa \mfb + \mfc \mfb = (\mfa  + \mfc) \mfb =
  \mfc \mfa \mfb.$$ Composing both by using Axiom \PXc, we obtain
  $$\mfc \mfa \mfb= \mfa \mfb +  \mfa \mfc = \mfa \mfb +  \mfc=  \mfa \mfc  \mfb, $$
 as desired.
  \item[\KXb:] By axiom \PXd  $$\mfa + \mfb \mfc = \mfb \mfc + \mfa \mfc = (\mfb  + \mfa) \mfc =  \mfb \mfa \mfc.$$   On the other hand,  Axiom \PXa \  and Axiom \PXb  \ imply
  $$\mfa + \mfb \mfc = \mfa + \mfb +  \mfb \mfc = \mfb\mfa +  \mfb \mfc = \mfb( \mfa +  \mfc) = \mfb \mfc \mfa ,$$ yielding the relation $\mfb \mfa \mfc =  \mfb \mfc \mfa$.
\end{enumerate}
Thus, $\APlc_I$ admits the Knuth relations \eqref{eq:knuth.rel}, and therefore  the congruence $\kcong$.
\end{proof}

Theorem \ref{thm:plc.Alg.1} also implies
 that one can produce a troplactic algebra form a given plactic monoid $\plcM_I = \tA_I / _{\kcong}$ (Definition \ref{def:plactic.mon})  by adjoining a zero element and introducing a formal additive (commutative) operation $+$, subject to Axioms
\PXb-\PXd.

\begin{example} $ $
\begin{enumerate}\eroman
  \item $\APlc_1$ is the \textbf{trivial  troplactic algebra} has rank $1$ and it is generated by a single,
    and thus $\APlc_1 = \{ \mfa ^k \ds | k \in \N_0 \}. $

  \item The troplactic algebra $\APlc_2$ with two  generators  $\mfb > \mfa $ consists of the elements of the forms
  $$
  \begin{array}{lll}
   \mfa^i \mfb^j, \quad  \mfa^i \mfb^j + \mfa^p + \mfb^q, & & p > i, q > j,  \\
  \end{array}$$
  where $i,j ,p, q \in \N_0$.
   If $\mfa \mfb = \mfe$, then also $\mfb \mfa = \mfe$
  by Remark \ref{rmk:plac.1}.(vi), and $\APlc_2$ contains only powers of $\mfa$ and $\mfb$.

\end{enumerate}

\end{example}

In  troplactic algebra, by the proof of Theorem \ref{thm:plc.Alg.1}, we see that the terms of the form
$\mfa  \mfc \mfb = \mfc  \mfa  \mfb = \mfa \mfb +\mfc $ and $ \mfb \mfa \mfc = \mfb  \mfc  \mfa = \mfa +\mfb\mfc$
decompose in $\APlc_I$ to sums of products of length $1$ and $2$. On the other hand $\mfc \mfb \mfa = \mfa+\mfb +\mfc$ decomposes to a sum of generators, while $\mfa\mfb\mfc$ is not necessarily decomposable. Thus, in  general, the products
$\mfa \mfb \mfc$ and $\mfc \mfb \mfa$ in $\APlc_I$ are different from
$\mfa  \mfc \mfb = \mfc  \mfa  \mfb$ and $ \mfb \mfa \mfc = \mfb  \mfc  \mfa$, and multiplicatily $\APlc_I$ has a non-degenerated plactic structure.
\pSkip

In the forthcoming study of troplactic algebra, we intensively use our notations for sequences, cf.  \S\ref{ssec:notation}.

\begin{theorem}\label{thm:plc.inc.terms} Suppose $L_m = [\ell_1, \dots, \ell_m]$ is nondecreasing, with $\ell_1, \dots, \ell_m \in I$ and $m \geq 1$,  then
\begin{equation}\label{eq:fis.dom}  \mfa_{\ell_1} \cdots \mfa_{\ell_m} =   \sum_{\incS \sqsubseteq L_m}\prod_{s \in \incS } \mfa_{s}. \end{equation}
\end{theorem}

Note that, as $L_m$ is assumed nondecreasing, this means that we only care about maximal nondecreasing sequences and  ignore all shorter subsequences. That is, there are no nested nondecreasing sequences.
\begin{proof}

Proof by induction on $m$. The case of $m=1$ is trivial, and we start $m = 2$. Let $L_2 = [i, j]$, where  $i = \ell_1 \leq  \ell_2 = j $.
Then, by Remark \ref{rmk:plac.1} we have
\begin{align*}\sum_{\incS \sqsubseteq L_2}\prod_{s \in \incS } \mfa_{s} & =  \mfa_i +  \mfa_j + \mfa_i \mfa_j = \mfa_i \mfa_j.
\end{align*}

Assuming the implication for $m-1$, since $L_m$ assumed nondecreasing,  we have
\begin{align*}\sum_{\incS \sqsubseteq L_m}\prod_{s \in \incS } \mfa_{s} & =
\sum_{\incS \sqsubseteq L_{m-1}}\prod_{s \in \incS } \mfa_{s} + \bigg(\sum_{\incS \sqsubseteq L_{m-1}}\prod_{s \in \incS } \mfa_{s}\bigg)\mfa_{\ell_m}
\\
&= \mfa_{\ell_1} \cdots \mfa_{\ell_{m-1}} + \big(\mfa_{\ell_1} \cdots \mfa_{\ell_{m-1}}\big)
\mfa_{\ell_{m}} \\ &=
\mfa_{\ell_1} \cdots \mfa_{\ell_{m-1}}\big(\mfe + \mfa_{\ell_{m}}\big) =
\mfa_{\ell_1} \cdots \mfa_{\ell_{m}},
\end{align*}
as required.
\end{proof}

\begin{remark}\label{rmk:lexorder.2}
From Theorem \ref{thm:plc.inc.terms} we learn  that for a nondecreasing  sequence (i.e., word) the addition of $\APlc_I$ is compatible with the lexicographic order (cf. Remark \ref{rmk:lex.order}),
$$ \cset{\mfu} \leq_\lex \cset{\mfv} \ds{\text{and}}  \cset{\mfu} \wrdsset \cset{\mfv} \dss{\Rightarrow} \mfu + \mfv = \mfv.$$
Note that in general $\mfa < \mfb$ does not imply $\mfa + \mfb = \mfb$, and also that the stand alone condition $\cset{\mfu} \leq_\lex \cset{\mfv}$ does not imply $\mfu + \mfv = \mfv$.

\end{remark}

The linkage between inclusion of words to their lexicographic order allow to consider maximality in a suitable  algebraic context.

\begin{corollary}\label{cor:facotrize.plc} Any product $\mfa_{\ell_1} \cdots \mfa_{\ell_m}$ in $\plc_I$, with $\ell_1, \dots, \ell_m  \in I$ and $m \geq 1$,
 decomposes as
 \begin{equation}\label{eq:facotrize.plc}
    \mfa_{\ell_1} \cdots \mfa_{\ell_m} =    \sum_{ \incS \in \MNDC{L_m}}  \prod_{s \in \incS } \mfa_{s}  , \qquad L_m = \squ{\ell_1, \dots, \ell_m }.
 \end{equation}

 \end{corollary}

 \begin{proof}  If $S \sqsubseteq L_m$ is not nondecreasing then $\prod_{s \in S } \mfa_{s}$  splits to a sum of terms by Axiom \PXb, otherwise for the nondecreasing sequences  $\incS  \sqsubseteq L_m$ it is enough to restrict to the maximal ones, by Theorem ~\ref{thm:plc.inc.terms} and Remark \ref{rmk:lexorder.2}.
\end{proof}

In particular cases of strictly decreasing words, this Corollary provides us the following useful computational simplification.

\begin{corollary}\label{cor:facotrize.plc.2} Any finite   product $(\mfa_{\ell_1})^{q_1} \cdots (\mfa_{\ell_m})^{q_m}$ in $\plc_I$ with $\mfa_{\ell_1}  > \mfa_{\ell_2} > \cdots >  \mfa_{\ell_m}  $
 decomposes  to
 \begin{equation}\label{eq:facotrize.plc.2}
   (\mfa_{\ell_1})^{q_1} \cdots (\mfa_{\ell_m})^{q_m} =    \sum_{t =1} ^m   (\mfa_{\ell_t})^{q_t},
 \end{equation}
 and thus
 $$ (\mfa_{\ell_1})^{q_1} \cdots (\mfa_{\ell_m})^{q_m} + (\mfa_{\ell_m})^{q_m} \cdots (\mfa_{\ell_1})^{q_1} = (\mfa_{\ell_1})^{q_1} \cdots (\mfa_{\ell_m})^{q_m}.$$
 \end{corollary}

We also have the following generalization to arbitrary finite products.
\begin{corollary} $\mfu^m + \mfu^{m-1} + \cdots + \mfu = \mfu^m  $ for every finite product  $\mfu \in \APlc_I $, i.e., a word $\cset{\mfu}$ over $\tA_I$.

\end{corollary}

\begin{proof} For every  $k < m$,  each  term in decomposition \eqref{eq:facotrize.plc} of $\mfu^k$ is also contained in one of the terms of  decomposition \eqref{eq:facotrize.plc} of $\mfu^m$, both are nondecreasing,  and we are done by  Theorem \ref{thm:plc.inc.terms}.
\end{proof}

Recall that the \textbf{elementary symmetric polynomials} in $n$ variables $X_1, \dots, X_n$   are defined as
\begin{align}
   \elm_{k}(X_1, \dots, X_n) & = \sum_{1 \leq i_1< i_2 < \cdots < {i_k} \leq n} X_{i_1} X_{ i_2} \cdots X_{i_k} \nonumber \\ & = \sum_{\incS_k \sqsubseteq N}\prod_{s \in \incS_k } X_{s} \ ,&  \qquad N =  \squ{1, \dots, n } \label{eq:elmpoly},
\end{align}
for $ k = 1, \dots, n$. In analogy, the \textbf{elementary symmetric expressions} on a finite alphabet $\Al_n$ are syntactic formulas,   defined as
\begin{equation}\label{eq:elmexp}
    \elm_{k}(\Al_n) = \sum_{1 \leq i_1< i_2 < \cdots < {i_k} \leq n} \lt_{i_1} \lt_{ i_2} \cdots \al_{i_k} = \sum_{\incS_k \sqsubseteq N}\prod_{s \in \incS_k } \lt_{s} \ , \qquad N =  \squ{1, \dots, n },
\end{equation}
for $ k = 1, \dots, n$. By  Remark \ref{rmk:plac.1}.(iii),  for expressions taken over $\APlc_n$,  we have
$$ \elm_{\ell}(\APlc_n) + \elm_{k}(\APlc_n) = \elm_{\ell}(\APlc_n) $$
for any $\ell \geq k$.

\begin{lemma}\label{lemma:Frob.2}  Any pair of generators $\mfa$ and $\mfb$ of $\APlc_I$ admits the \textbf{Frobenius property}:
\begin{equation}\label{eq:genr.Frob.2}
    \big(\mfa + \mfb\big)^m = \mfa^m + \mfb^m \qquad \text{for any } m \in \N.
\end{equation}

\end{lemma}

\begin{proof} Proof by induction on $m$. The case of  $m =2$ is given by Remark \ref{rmk:plac.1}.(iv):
\begin{align*}  (\mfa + \mfb)^2 & = \mfa^2 + \mfa\mfb +  \mfb \mfa + \mfb^2 \\ &  = \mfa^2 + \mfa + \mfb +  \mfa \mfb  + \mfb^2  = \mfa^2 + \mfb^2.\end{align*}
Assuming the implication holds for $m- 1$, write
\begin{align*} (\mfa + \mfb)^m  & = (\mfa + \mfb)(\mfa + \mfb)^{m-1}  = (\mfa + \mfb)(\mfa^{m-1} + \mfb^{m-1}) \\ & = \mfa^m  + \mfa \mfb^{m-1} + \mfb \mfa^{m-1} + \mfb^{m} \\& =  \mfa^m  + \mfa \mfb^{m-1} + \mfb + \mfa^{1} + \mfa^{2} + \cdots +  \mfa^{m-1} + \mfb^{m}  & \text{[by Theorem \ref{thm:plc.inc.terms}]}\\ & =
 \mfa^m  + \mfa \mfb^{m-1} + \mfb^{m}  \\ & =
 \mfa^m  + (\mfa \mfb + \mfb^2 )\mfb^{m-2} & \text{[by Axiom \PXc]}
 \\ &  = \mfa^m  + (\mfa + \mfb^2)\mfb^{m-2}
 =\mfa^m  + \mfa \mfb^{m-2}  +  \mfb^{m}
 \\& = \mfa^m  + \mfa\mfb^{m-2}  + \mfb^{m-1} + \mfb^{m} & \text{[by \eqref{eq:plc.power}]}
  \\ &  = \mfa^m  + (\mfa\mfb+ \mfb^{2})\mfb^{m-3} + \mfb^{m}
  \\ &  = \mfa^m  + (\mfa + \mfb^{2})\mfb^{m-3} + \mfb^{m}
  \\& \qquad \vdots   \\
   & = \mfa^m  + (\mfa\mfb+ \mfb^{2}) + \mfb^{m} \\ & = \mfa^m  + \mfa+ \mfb^{2} + \mfb^{m} =
   \mfa^m  +  \mfb^{m}
  \end{align*}
  as desired.
\end{proof}

As an immediate consequence of this proof we obtain the following.
\begin{corollary} Suppose  $\mfa <   \mfb$ are generators of   $\APlc_I$, then
   \begin{equation}\label{eq:bmam}
    \big(\mfa + \mfb\big)^m = \mfa^m + \mfb^m  = \mfb^m\mfa^m = (\mfb \mfa)^m
\end{equation}
for any  $ m \in \N.$
\end{corollary}

The troplactic algebra
$\APlc_I$  in general does not admit the Frobenius property \eqref{eq:genr.Frob.2}, i.e., for elements that are not generators; for example
  $$ \mfu = (\mfa \mfb + \mfc)^2 = (\mfa \mfb)^2 + \mfa \mfb \mfc + \mfc^2 \ds \neq (\mfa \mfb)^2 + \mfc^2 = \mfv, $$
since $\mfa \mfb  \mfc$ appear in $\mfu$ but not in $\mfv.$

\begin{remark}
  Not all the elements of $\APlc_I$ can be rewritten as products, i.e., $\APlc_I$  has  elements that are not multiplicatively generated.
  For example, the element $ad + bc $  is not provided as any word $\cset{u}$ in $\APlc_I$. Indeed, otherwise the terms  $ad$ and $bc$ would be  subwords of $\cset{u}$,  which does not contain other increasing subwords of length 2, but this impossible.
\end{remark}

\subsection{Dual tropical plactic algebra}\label{ssec:du.troplactic.alg}\sSkip
 We turn to a dual version of a troplactic algebra $\APlc_I$ (cf. Definition~ \ref{defn:trop.plc.alg}), to be employed later, carrying analogous properties.
\begin{definition}\label{defn:d.trop.plc.alg}
A \textbf{dual  troplactic algebra}, denoted $\dAPlc_I$,  is a (noncommutative) idempotent semiring $(\dAPlc_I, \d+, \cdot \,)$  with  (multiplicative) identity element $\mfe$ and zero $\mfo$,  generated by an ordered set of elements 
  $\{ \mfa_i \ds | i \in I \}$ subject to the  axioms (for every $\mfa \leq  \mfb \leq \mfc$ in $\{ \mfa_i \ds | i \in I \}$):
\begin{description} \eroman \dispace
 \item[\dPXa] $ \mfa  = \mfe  \d+ \mfa$;
  \item[\dPXb] $ \mfa \mfb = \mfa \d+ \mfb$  when  $\mfb >  \mfa $;
 \item[\dPXc] $ (\mfb  \d+  \mfc) \mfa = \mfb \mfa  \d+ \mfc;$
  \item[\dPXd] $\mfc(\mfb  \d+  \mfa)  = \mfa \d+ \mfc \mfb .$
\end{description}
For indexing matter, we denote $\mfe$ also by $\mfa_0$.
\end{definition}
\noindent
(We use  the special notation ``$\d+$'' for  the addition in $\dAPlc_I$ to distinguish it from the addition of $\APlc_I$.) \pSkip

In analogy to Remark \ref{rmk:plac.1}, by changing nondecreasing  sequences by nonincreasing sequences,  we now have following properties.
\begin{remark}\label{rmk:d.plac.1} Let $\mfa \leq \mfb \leq \mfc $  be elements in the generating set $\{ \mfa_i \ds | i \in I \}$ of $\dAPlc_I$.
\begin{enumerate}\eroman \dispace
\item Axiom \dPXa\ is  derived directly from Axiom \dPXb\ by declaring that $\mfe < \mfa_i$ for all $i$.
  \item   The multiplicative identity  $\mfe$ is also an additive  identity (by Axiom \dPXa)
  $$ \mfa \mfe = \mfe \mfa = \mfa = \mfa \d+ \mfe = \mfe \d+ \mfa.$$

  \item  Since  $\mfa\mfb = \mfa(\mfb \d+ \mfe) = \mfa\mfb \d+ \mfa$, and  $\mfa\mfb = (\mfa \d+ \mfe)\mfb = \mfa \mfb \d+ \mfb$,
      by Axiom \dPXa, the idempotence of $\dAPlc_I$ implies that $$\mfa\mfb \d+ \mfa \d+ \mfb = \mfa\mfb \d+ \mfa  = \mfa\mfb  \d+ \mfb = \mfa\mfb.$$ This gives
          \begin{equation}\label{eq:plc.power}
       \mfa^m \d+ \mfa^{m-1} \d+ \cdots \d+ \mfa = \mfa^m,
  \end{equation}
while using Axiom \dPXb \ inductively for an arbitrary product
        $\mfa_{\ell_1} \cdots   \mfa_{\ell_m} $ we get
     \begin{align*}
     \sum^{\wedge}_{\decS \sqsubseteq L_m } \prod_{s \in \decS}\mfa_{s}  = \mfa_{\ell_1}  \cdots    \mfa_{\ell_m}, & \qquad  L_m = \squ{\ell_1, \dots, \ell_m}.
     \end{align*}

\item When $\mfa = \mfb,$ Axiom \dPXc\ reads  $\mfa \mfa \d+ \mfc \mfa  = \mfa \mfa \d+ \mfc,$ and thus
by (iii) we have
\begin{equation}\label{eq:plc.aa+ac}\mfa^2 \d+ \mfc \mfa  \d+ \mfc^2 = \mfa ^2 \d+ \mfc \d+ \mfc^2 = \mfa ^2 \d+  \mfc^2, \end{equation}
and thus $(\mfa  \d+  \mfc)^2 = \mfa ^2 \d+  \mfc^2$.

\item When $\mfb = \mfc,$ Axiom \dPXd\ reads  $\mfb \mfb \d+ \mfb \mfa  = \mfa \d+ \mfb \mfb .$

\item
 If $\mfb \mfa  = \mfe $, then, by Axioms \dPXa, \dPXb, and part (iii), we have
  $$\mfa \mfb  = \mfa \d+ \mfb = \mfa \d+ \mfe  \d+ \mfb \d+ \mfe =  \mfa \d+ \mfb \mfa  \d+ \mfb \d+ \mfb \mfa = \mfb \mfa \d+ \mfb \mfa = \mfe \d + \mfe = \mfe,$$
  and thus also  $\mfa \mfb = \mfe$. This implies that all the elements of $\APlc_I$ are of the form $\mfa_\ell^k$, with $k \in \N.$



\end{enumerate}
\end{remark}

\begin{theorem}\label{thm:d.plc.Alg.1} The dual troplactic algebra  $\dAPlc_I$ admits the Knuth relations \eqref{eq:knuth.rel}
\begin{equation*}\label{eq:d.plac.Alg.2}
  \begin{array}{ccccccc}
   \text{\KXa}:&& \mfa \; \mfc \; \mfb = \mfc \; \mfa \; \mfb & \text{if} &  \mfa \leq \mfb < \mfc , & \qquad & \\
   \text{\KXb}:&& \mfb \;\mfa \; \mfc = \mfb \; \mfc \; \mfa & \text{if} & \mfa <  \mfb \leq \mfc ,   \end{array}
\end{equation*}
where  $\mfa, \; \mfb, \; \mfc \in \{ \mfa_i \ds| i \in I \}$.
\end{theorem}
\begin{proof} We fellow similar arguments as in the proof of Theorem \ref{thm:plc.Alg.1}.
\begin{enumerate} \dispace
  \item[\KXa:]

  By axiom \dPXd  $$\mfa \d+ \mfc \mfb  = \mfc \mfb  \d+ \mfc \mfa  = \mfc (\mfb  \d+ \mfa)  =  \mfc \mfa \mfb.$$   On the other hand  Axiom \dPXa \  and Axiom \dPXb  \ imply
  $$\mfa \d+ \mfc \mfb  = \mfa \d+ \mfb \d+  \mfc \mfb  = \mfa \mfb \d+  \mfc \mfb  = ( \mfa \d+  \mfc)\mfb = \mfa \mfc \mfb,$$ yielding $ \mfc \mfa \mfb  =  \mfa \mfc  \mfb$.

  \item[\KXb:]
  By axiom  \dPXb:  $$ \mfb \mfa \d+   \mfc \mfa= (\mfb \d+ \mfc)\mfa =  \mfb \mfc  \mfa,$$  while Axiom \dPXa \  gives
  $$ \mfb \mfa  \d+  \mfc =  \mfb  \mfa \d+ \mfb \d+ \mfc = \mfb  \mfa  \d+ \mfb \mfc  = \mfb (\mfa  \d+ \mfc)  =
   \mfb \mfa \mfc .$$ Composing both by using Axiom \dPXc, we obtain
  $$\mfb \mfc \mfa =   \mfb \mfa  \d+  \mfc \mfa  = \mfb \mfa  \d+  \mfc =  \mfb \mfa  \mfc$$
  as required.
\end{enumerate}
Thus, $\dAPlc_n$ admits the Knuth relations \eqref{eq:knuth.rel}, and therefore  the congruence $\kcong$.
\end{proof}

With similar proofs, adapted to Definition \ref{defn:d.trop.plc.alg}, the results of \S\ref{ssec:troplactic.alg} also hold for  dual troplactic algebras, replacing nondecreasing sequences by nonincreasing sequences.

\section{Tropical matrices}\label{sec:trop.matrices}

The tropical  \textbf{max-plus semiring} $(\Trop, \+, \cdot \,)$ is the set  $\Trop := \Real\cup\{-\infty\}$ endowed with the operations of maximum
and summation (written in the standard algebraic way),
$$a \+  b :=\max\{a,b\}, \qquad a b :=a \underset{\operatorname{sum}}{+}b,$$
serving respectively as the semiring addition and multiplication
\cite{pin98}.  (We save the sign $``+"$  for the standard summation, and use ``$\+$'' for addition.)     This semiring is additively idempotent, i.e., $a \+ a = a$ for every $a\in \Trop$, in which
$\zero := -\infty$ is the zero element and $\one := 0$ is the
multiplicative  unit. (Equivalently one can set $\Trop := \N \cup\{-\infty\}$,  $\Trop := \Z \cup\{-\infty\}$, or  $\Trop := \Q \cup\{-\infty\}$).

Dually, the \textbf{min-plus semiring} $(\Trop_\-, \-, \cdot \,)$ is the set  $\Trop_\- := \Real\cup\{\infty\}$ equipped with the operations
$$ a \-  b :=\min\{a,b\}, \qquad a b :=a \underset{\operatorname{sum}}{+}b,  $$
addition and multiplication respectively.
The addition $\-$ over $\Real$  can be written  in terms of $\+$ as
 $ \displaystyle{ a \- b  = \frac{ab}{a \+b },}$
or as
 \begin{equation}\label{eq:minimum.trop.2}
 a \- b  = (a^{-1} \+ a^{-1})^{-1} = -((-a) \+ (-b)),
\end{equation}
where $-a$ stands the standard negation of $a \in \Real$ -- the tropical multiplicative inverse of $a$. (When $a = -\infty$ in \eqref{eq:minimum.trop.2} we implicitly take $-a = \infty$.)
We write $\Trop$ for the max-plus semiring $(\Trop, \+, \cdot \,)$  and $\Trop_\-$ for the min-plus semiring $(\Trop, \-, \cdot \,)$.

The \textbf{boolean semiring} is the two element
idempotent semiring $\B := (\{ 0,1\}, \+ , \cdot \,),$ whose addition and
multiplication are given for $a,b \in \{ 0,1 \}$ by:
$$ a \+ b = \left\{ \begin{array}{ll}
                    1 & \text{ unless } a = b = 0, \\
                    0 & \text{ otherwise}, \end{array} \right.
                     \qquad
                      a b = \left\{ \begin{array}{ll}  0 & \text{ unless } a =b=1,\\
                      1 & \text{ otherwise}. \end{array} \right.$$
$\Bool$ embeds naturally in $\T$ via $0 \mTo \zero$ and $1 \mTo \one$, while the projection
\begin{equation}\label{eq:trop.to.bool}
\pi : \T \ONTO \B, \qquad a \longmapsto \left\{ \begin{array}{lll}
                                                1 &&  a \neq \zero,   \\
                                                0 && a = \zero,
                                              \end{array}
\right.
\end{equation}
is a semiring homomorphism (Definition \ref{def:smr.hom}).

\subsection{Tropical matrices}\label{sec:mathGr} \sSkip
Tropical matrices are matrices with entries in $\Trop$, whose multiplication is
induced from the operations of $\Trop$ as in the familiar matrix construction. The set of all $\nxn$ tropical matrices forms a noncommutative semiring, denoted by  $\MatnT$. The \bfem{unit} $I$ of
 $\MatnT$ is the matrix
with $\one := 0$ on the main diagonal and whose off-diagonal
entries are all $\zero := \minf$. The zero matrix is denoted by $(\zero)$, the absorbing element of $\MatnT$.
Formally,  for any nonzero matrix $A \in \MatnT$ we set $A^{0} :=
I$. A  matrix $A \in  \MatnT$ with entries $a_{i,j}$, $i,j = 1,\dots,n$, is written
as $A = (a_{i,j})$.
For a pair of  matrices $A,B \in \Mat_n(\T)$, we define
\begin{equation}\label{eq:mat.order}
A \geq   B \quad \text{ if $a_{i,j} \geq  b_{i,j}$ for all $i,j$},
\end{equation}
and $A > B$ if $A \geq B$ with  $a_{i,j} >  b_{i,j}$ for some $i,j$. Equivalently, we have
$$ A \geq B \dss{\text{iff}} A + B = B.$$
Therefore, in general, $\MatnT$ is a partially ordered semiring with $(\zero) < A$ for all $A$.
We denote the monoid of all upper tropical triangular matrices by $\TMat_n(\Trop)$.
\begin{remark}\label{rmk:dual.mat}
Given a matrix  $A = (a_{i,j})$ in $ \MatnT$ we write $-A$ for the matrix
\begin{equation}\label{eq:math.neq}
-A := (\inv{a}_{i,j}) = (-a_{i,j}).
\end{equation}
Then, by
\eqref{eq:minimum.trop.2}, the entry-wise operation gives
$$ A \- B = -(-A \+ -B),$$ i.e.,
the entry-wise minimum of the matrices $A$ and $B$.
\end{remark}

The \textbf{structure  map} is the onto monoid homomorphism
\begin{equation}\label{eq:struc.map}
\tlpi: \MatnT \ONTO \MatnB
\end{equation}
sending  a tropical matrix $A \in  \MatnT$ to the boolean matrix $\tlpi(A) \in \MatnB$ defined as
$\tlpi(A) = (\pi(a_{i,j}))$,  by the entry-wise mapping \eqref{eq:trop.to.bool}.   (It is a ``value forgetful'' homomorphism that maps nonzero entries to ~$1$ and zero-entries to~ $0$.) The matrix $\tlpi(A)$ is called the \textbf{structure matrix} of $A$.

Tropical matrix algebra was systematically  studied in its generalized context of supertropical algebra \cite{IR1}, providing many tropical analogues  to classical results \cite{zur05TropicalAlgebra,IR2,IR3,IR4}, while subgroup and semigroups structures of these matrices were discussed in  \cite{merl10,Simon2}.
As a matrix  monoid, in comparison to matrices over a field, important  submonmoids of $\MatnT$ have a very special behavior,
as have been dealt  in \cite{trID,IDMax,IzhakianMargolisIdentity}, yielding the following results.

\begin{theorem}[{\cite[{Theorem 4.11}]{trID}}]\label{thm:trMat.Id}The submonoid $\TMatnT \subset \MatnT$ of all upper (or dually lower)
triangular tropical  matrices satisfies the semigroup identities defined in Construction \ref{construct:main}:
  \begin{equation}\label{eq:iduniv2.2}
\Id_{(C,n-1,n-1)}: \quad   \tlw_{(C,n-1,n-1)} \ds {x}
\tlw_{(C,n-1,n-1)} \ds =  \tlw_{(C,n-1,n-1)} \ds
 {y}
\tlw_{(C,n-1,n-1)},
\end{equation}
for $C = \{ x,y\}$, with $x = AB$ and $y = BA$.
\end{theorem}
\begin{theorem}[{\cite[Theorem 4.14]{IDMax}}]\label{thm:simDia} Any nonsingular subsemigroup $\tM_n \subset \MatnT$ satisfies the identities ~ \eqref{eq:iduniv2.2} by letting   $x =  \An \Bn$ and $y =  \Bn \An$.
\end{theorem}
\noindent Clearly, these results hold also to boolean matrices.

\pSkip

We recall some basic definitions from \cite{ultimateRank,zur05TropicalAlgebra,IR1,IR2}.

\begin{definition}\label{def:matOper} Given  a tropical matrix $A =  (a_{i,j})$ in $\MatnT$.
\begin{enumerate}  \eroman \dispace
\item The \textbf{transpose}  of $A$ is defined as $A^\trn = (a_{j,i})$  and satisfies  $(AB)^\trn = B^\trn A^\trn$ for any  $B \in \mT$.

 \item The \textbf{permanent} of  $A$ is
defined as:
  \begin{equation*}\label{eq:tropicalDet}
 \per(A) = \bigp_{\sig \in \mfS_n}    \prod_i a_{i,\sig(i)},
\end{equation*}
where $\mfS_n$ denotes the set of all the permutations over $\{1,\dots,n\}$.
 \item $A $ is called \textbf{nonsingular}
 if there exist a unique permutation $\tau_A \in \mfS_n $ that reaches $\per(A)$; that is,
$ \per(A)  =  \prod_i a_{i,\tau_A(i)} \;.$
Otherwise, $A$ is called \textbf{singular}.

\item The \textbf{(tropical) rank} of $A$ is the largest $k$ for which $A$ has a $k \times k$ nonsingular sub-matrix. (Equivalently,  the rank is the  maximal number of independent columns (or rows) of $A$, cf.  \cite{IR1}.)


\item The \textbf{trace} and the \textbf{multiplicative trace}  of $A$ are defined respectively as
 $$ \tr(A) = \bigp_i a_{i,i}, \qquad  \mtr(A) = \prod_i a_{i,i}.$$

\end{enumerate}
\end{definition}

\noindent
It is easily seen that
 $\per(A) \geq \mtr(A)$ and $\mtr(AB) \geq \mtr(A) \mtr(B)$,
 for any matrices $A,B \in \mT$.

\subsection{Digraph view to tropical matrices}\label{ssec:digraph}
\sSkip

Square tropical matrices correspond uniquely to weighted diagraphs, whose products interpreted as weights of paths \cite{ABG}, as  described below.
An $n\times n$ tropical matrix $A = (a_{i,j})$ is associated to the \textbf{weighted digraph} $\grph_A := (\ver, \arc)$ on the vertex set $\ver :=\{ 1, \dots, n\}$ with   a directed edge
$\e_{i,j} := (i,j) \in \arc$ of \textbf{weight} $a_{i,j}$ from $i$ to $j$  for every
$a_{i,j} \ne \zero$.
The digraph of the transpose matrix $A^\trn$ is obtained by redirecting the edges of $\grph_A $, i.e., an edge $(i,j)$ is replaced by the edge $(j,i)$ of the same  weight.

A \textbf{path} $\gm$ is a sequence of edges
$\e_{i_1, j_1}, \dots, \e_{i_m, j_m} $, with $j_{k} = i_{k+1}$ for every
$k = 1,\dots, m-1$. We write $\gm := \gm_{i,j}$ to indicate that
$\gm$ is a path from $i = i_1$ to $j=j_m$.
The \textbf{length} $\len{\pth}$ of a path $\pth$ is the number of
its edges. 
Formally, we consider also paths of length $0$, which we call
\textbf{empty paths}. The \textbf{weight} $\w(\pth)$ of a path
$\pth$ is  the tropical product of the weights of all
the edges $\e_{i_k, j_k}$ composing $\pth$, counting repeated edges. The weight of an empty path is formally set to be $\zero$.

A path is \textbf{simple} if it crosses each vertex at most once -- it does not have repeating vertices.
(Accordingly, an empty path is considered also as simple.) A (simple) path
that starts and ends at the same vertex is called a
\textbf{(simple) cycle}.
An edge $\rho_i := \e_{i,i}$ is called a \textbf{loop}. We write
$(\rho)^k$ for the composition $\rho \circ \dots \circ \rho$ of a loop
$\rho$ repeated $k$ times, and call it a \textbf{multiloop}. The
notation $(\rho)^0$ is formal, stands for the  empty loop, which
can be realized as a vertex. A graph is \textbf{acyclic} if it has no cycles of length $> 1$ (i.e., it may have loops).

In combinatorial view, entries in powers of a tropical matrix $A \in \MatnT $
correspond to paths of maximal weights in the associated digraph $G_A$ of $A$. Namely, the $(i,j)$-entry of the matrix  $A^m$ corresponds to the highest weight
over  all the paths~$\gm_{i,j}$ from $i$ to $j$  of length $m$ in
$G_A$.

The graph view of products $A_{\ell_{1}} \cdots A_{\ell_m}$  of different $n \times n$ matrices  $A_{\ell_\tt} \in \{ A_\ell \ds | \ell \in I \}$ is more involved and includes paths with edges contributed by  different  digraphs $\grph_{A_\ell} = (\ver, \arc_\ell) $ occurring on the common vertex set  $\ver$. To cope  with this generalization,  we  assign the weighted edges $\e_{A_\ell} \in \arc_\ell$ of each digraph ~ $G_{A_\ell}$ with
a unique color, say~$c_\ell$, and define the \textbf{colored  digraph}
$$\grph_{A_{\ell_1} \cdots A_{\ell_m}} := \bigcup_{\tt =1 }^m  \grph_{A_{\ell_\tt}} \qquad  (\ell_\tt \in I) $$ to have  the  vertex set
$\ver = \{1, \dots, n \}$ and a set of  edges $\bigcup_{\tt =1 }^m \arc_{\ell_\tt}$  colored $c_{\ell_\tt}$, obtained from $\grph_{A_{\ell_\tt}}$.  This digraph may  have
multiple edges with different colors, an edge of $\grph_{A_{\ell_1} \cdots A_{\ell_m}}$ is denoted by  $[\e_{A_\ell}]_{i,j}$.

In this setting, the
$(i,j)$-entry of a matrix product $A_{\ell_1} \cdots A_{\ell_m}$
corresponds to the highest weight of all colored paths $[\e_{A_{\ell_1}}]_{i_1,
j_1}, \dots, [\e_{A_{\ell_m}}]_{i_m, j_m} $ of length $m$ from $i = i_1$ to $j=
j_m$ in the digraph $\grph_{A_{\ell_1} \cdots A_{\ell_m}}$, where each edge
$[\e_{A_{\ell_\tt}}]_{i_\tt, j_\tt}$ has color $c_{\ell_\tt}$, $\tt =1, \dots,m$. Namely, every edge of the path is
contributed uniquely by the  digraph $G_{A_\ell}$, respecting the color ordering determined by the multiplication  concatenation $A_{\ell_1} \cdots
A_{\ell_m}$.
Working with a colored digraph $\grph_{A_{\ell_1} \cdots A_{\ell_m}}$, we always restrict
to colored paths,  called \textbf{properly colored paths},   that respect the sequence of coloring $c_{\ell_1},
\dots, c_{\ell_m}$.
For this reason, we use the awkward notation
$\grph_{A_{\ell_1} \cdots A_{\ell_m}}$ that records the multiplication concatenation $A_{\ell_1} \cdots
A_{\ell_m}$.

\begin{notation}\label{nott:1}
A matrix product $U = A_{\ell_1} \cdots A_{\ell_m}$ is denoted  as  $\setU$ to
indicate that $U$ is realized as the word obtained by the
concatenation of the symbols $``A_1",  \dots, ``A_m"$, accordingly  $\grph_{A_{\ell_1} \cdots A_{\ell_m}}$ is denoted by $\grph_\setU$, while $U = (u_{i,j})$ stands for the actual result  of the   matrix product.
A path $\gm_{i,j}$ in $\grph_\setU$ is  denoted  by $[\gm_\setU]_{i,j}$ to indicate the sequence of its edges' coloring determined by $\setU$.
\end{notation}
In this notation, an  entry
$u_{i,j}$ of $U = A_{\ell_1} \cdots A_{\ell_m}$ encodes the highest weight over all properly colored paths $[\gm_\setU]_{i,j}$, determined by the word $\setU$, of length $m$ from
$i$ to $j$ in the colored digraph $\grph_\setU$.  Conversely,
the word $\setU$ can be recovered from the coloring  sequence of the edges composing any path
$\gm_{i,j}$ of length ~$m$ in $\grph_\setU$.

\subsection{Synoptic matrices}\label{ssec:synoptic} \sSkip
We begin with a certain class of tropical matrices of a special characteristic.
\begin{definition}
A \textbf{synoptic matrix} is a  matrix $A= (a_{i,j})$ in $\mT$ in which
\begin{equation}\label{eq:synp.matrix.1}
    a_{i,j} \geq  a_{i+1,j} \+  a_{i,j-1},
\end{equation}
for all $i = 1, \dots, n-1$ and $j = 2, \dots, n$.
\end{definition}
We denote the set of all synoptic matrices by $\SynT$.
By definition, in any synoptic matrix:
\begin{enumerate}\ealph  \dispace
  \item Each row has a nondecreasing order (from left to right);

  \item  Each column  has a nonincreasing order (from top to bottom).
\end{enumerate}
In particular,  the zero matrix $(\zero)$ is also synoptic,

\pSkip

\begin{lemma}
$\SynT$ is a subsemiring in $\mT$.
\end{lemma}
\begin{proof}
To see that $\SynT$ is closed for multiplication,  let $C = AB$ and compute $C =(c_{i,j})$ as
\begin{align*} c_{i,j}  = \bigp_{t=1}^n a_{i,t}b_{t,j} & \geq  \bigg( \bigp_{t=1}^n a_{i,t}(b_{t,j-1} \+ b_{t+1,j} )\bigg) \+ \bigg(\bigp_{t=1}^n  (a_{i+1,t} +  a_{i,t-1}) b_{t,j}\bigg) \\
&  \geq  \bigg( \bigp_{t=1}^n a_{i,t}b_{t,j-1 } \bigg)   \+ \bigg( \bigp_{t=1}^n  a_{i+1,t}  b_{t,j} \bigg) = c_{i, j-1} \+ c_{i+1,j}.
\end{align*}
Hence property \eqref{eq:synp.matrix.1} is preserved.
The verification  that $\SynT$ is closed for addition is immediate.
\end{proof}

\subsection{Corner  and flat  matrices}\label{ssec:corner} \sSkip
In the next subsections we develop a theory of special types of tropical matrices, brought here in its full generality, to be used also for future applications.

\begin{notation}\label{nott:row-col}
Given a matrix $A$,  we denote by $\Rw(A) $ and $\Cl(A)$ the set of its rows and columns, respectively.
We write
 $\mrw{A}{I}$  for the restriction of $A$ to rows
 $I \subseteq \Rw(A)$ and $\mrw{A}{i}$ for the row $\bfr_i \in \Rw(A)$.  Similarly,  $\mcl{A}{J}$ stands  for the restriction of $A$ to columns
 $J \subseteq \Cl(A)$, and $\mcl{A}{j}$ for the column $\bfc_j \in \Cl(A)$.   $\clrw{A}{J}{I}$ denotes the restriction  of $A$
  the rows $I
\subseteq \Rw(A)$ and to the columns $J \subseteq \Cl(A)$.  When $A$ is a square matrix we say that  $\clrw{A}{J}{I}$ is  \textbf{principal submatrix} if $J = I$.

\end{notation}

We start with structure matrices, provided by the structure map $\tlpi : \mT \To \bT$, cf. \eqref{eq:struc.map}, and consider first  the combinatorial shape of matrices. In this setup, we ignore  the values of nonzero entries of  matrices (these values will show up later) and identify all of them with $1 \in \B$.
\begin{definition}
A \textbf{corner matrix} is a matrix whose top-right  \textbf{corner block} (possibly empty) contains only nonzero  entries and all entries out of this block are $\zero$.
A matrix $A \in \mT$ is called \textbf{$(\p,\q)$-corner} if its corner-block $B_{\p,\q}$ is $\rwcl{A}{I_\p}{J_\q}$, where $$I_\p = \{1, \dots, \p \} \dss{\text{ and }} J_\q = \{\q,  \dots,  n\}, \qquad \p,\q \in N := \{ 1, \dots, n \} ,$$ i.e., $a_{i,j} \neq \zero$ for all $(i,j) \in I_\p \times J_\q$, otherwise $a_{i,j} = \zero$.

Two corner matrices $A$ and $B$ in $\mT$ are said to be \textbf{block-similar}, written $A \bsim B$, if
 $\tlpi(A) = \tlpi(B)$, i.e.,  both are $(\p,\q)$-corner for some $(\p,\q) \in N \times N$.
When $\tlpi(A) \leq  \tlpi(B)$,  we write  $A \bleq B$,  which means that the corner block of $A$ is contained in that of $B$.
\end{definition} \noindent
The block indexing $(\p,\q)$ indicates the position of the bottom left corner of a  block, and thus uniquely determines the corner block. Note that for boolean matrices the block inclusion $\bleq$ is compatible with the matrix (partial) order  \eqref{eq:mat.order}, in the sense that $\tlpi(A) \bleq \tlpi(B)$ implies $\tlpi(A) \leq  \tlpi(B)$.

It easy to verify that the set of all corner matrices forms a multiplicative monoid in $\mT$,  which we denote by $\CorT$.
\begin{remark}\label{rem:cor.prod} Suppose $ {A}_{{t_u}} \in \CorT$, $t_u \in N$ with  $u =1,2,3$,   are  $(\p_{t_u},\q_{t_u})$-corner matrices, i.e., matrices whose corner blocks are given by $\rwcl{{A}_{t_u}}{I_{t_u}}{J_{t_u}}$, with $I_{t_u} = \{1, \dots \p_{t_u} \} $ and $J_{t_u} = \{ \q_{t_u}, \dots, n\} $.     Then, letting $s = t_1$ $t = t_2$, and $r = t_3$, we have the following:
\begin{enumerate} \eroman \dispace
\item If $\p_s < \p_t $ then $I_s \subset I_t$, while if $\q_s <  \q_t $ then  $J_s \supset J_t$.
  \item If $\q_s > \p_t $, that is $J_s \cap I_t = \emptyset$,  then $A_s A_t = \zeroM$.
  \item If $\q_s \leq \p_t $ then  $A_s A_t$ is  $(\p_s,\q_t)$-corner, whose corner block is $\rwcl{(A_s A_t)}{I_s}{J_t}$.

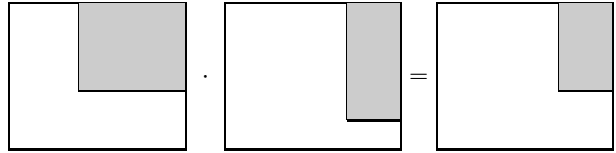
$$ \mresize{ \begin{array}{|l|llll|}
\hline
\quad \phantom{w}& \grcel & \grcel &\grcel&\grcel \\
\embox & \grcel & \grcel & \grcel & \grcel \\
\embox & \grcel & \grcel & \grcel & \grcel\\\cline{2-5}
\multicolumn{5}{|l|}{\embox} \\
\multicolumn{5}{|l|}{\embox}\\
\hline
\end{array} } \ \cdot \ \mresize{\begin{array}{|lll|ll|}
\hline
\embox & \embox & \embox &\grcel&\grcel \\
\embox & \embox & \embox & \grcel & \grcel \\
\embox & \embox & \embox & \grcel & \grcel \\
\embox & \embox & \embox & \grcel & \grcel\\\cline{4-5}
\multicolumn{5}{|l|}{\embox}\\
\hline
\end{array}}  = \mresize{
\begin{array}{|lll|ll|}
\hline
\embox & \embox & \embox &\grcel&\grcel \\
\embox & \embox & \embox & \grcel & \grcel \\
\embox & \embox & \embox & \grcel & \grcel\\\cline{4-5}
\multicolumn{5}{|l|}{\embox}\\
\multicolumn{5}{|l|}{\embox}\\
\hline
\end{array}}$$ \pSkip
 In particular, for any $(\p_t,\q_t)$-corner matrix $A_t$ with $\q_t\leq \p_t$,  $(A_t)^m$ is $(\p_t,\q_t)$-corner for every $m \in \N$.  Otherwise, for $\q_t >  \p_t$, we get   $(A_t)^m = \zeroM.$

  \item The block inclusion $A_r A_t \bleq A_r A_s$ holds iff $\q_t \geq \q_s$ (i.e.,  $J_t \subseteq J_s$), while $A_r A_t \bleq A_s A_t$ iff $\p_r \leq \p_s$ (i.e., $I_r \subseteq I_t$).
      In general for arbitrary nonzero products (i.e., $ \q_{t_{i}} \leq  \p_{t_{i+1}}$ and $ \q_{s_{j}} \leq \p_{s_{j+1}}$)
       we have
\begin{equation}\label{eq:bleq.1}
 A_{s_1} \cdots  A_{s_u} \bleq A_{t_1} \cdots  A_{t_v}  \dss{\Iff} \p_{s_1} \leq  \p_{t_1} \text{ and } \q_{s_u} \geq \q_{t_v} \  ,
\end{equation}
and in particular
\begin{equation}\label{eq:bsim.1}
A_{s_1} \cdots  A_{s_u}  \bsim A_{t_1} \cdots  A_{t_v}  \dss{\Iff} \p_{s_1} = \p_{t_1} \text{ and } \q_{s_u} = \q_{t_v} \ .
\end{equation}
\end{enumerate}

\end{remark}

In this paper we are especially interested in  $(\p,\q)$-corner matrices obtained as products of $(\p,\p)$-corner matrices, we call the latter  \textbf{$\p$-corner matrices}, for short, and denote them by $\Oix{A}{\p}$.
These are a special type of triangular  $(\p,\q)$-corner matrices with $\p \geq \q$.
 However, we have the following general observation, obtained by Remark~ \ref{rem:cor.prod}.(ii).
\begin{remark}\label{rmk:gen.fw.mon} The  monoid $\CorT$ of all  $(\p,\q)$-corner $\nxn$ matrices, ordered  by inclusion, forms a  partial \fsmg\ (Definition \ref{def:forwardMon}).
Restricting to a fixed generating subset of $\p$-corner matrices with different $\p = 1,\dots, n,$ one  obtains a finitely generated \fsmg.
\end{remark}

So far we have concerned  only with the structure of corner matrices, rather than the values of their nonzero entries.
Considering suitable values for these entries, block inclusions can be  translated to matrix inequalities \eqref{eq:mat.order}. To obtain this view we restrict to corner matrices whose nonzero entries all have a same fixed value $\flt \neq \zero$. We call these matrices \textbf{\eflat \ corner matrices}.
For every \eflat\ $\p$-corner matrix ~$\Oix{\fM}{\p}$ we have
\begin{equation}\label{eq:eflat.pow}
  (\Oix{\fM}{\p})^m = \flt^{m-1} \Oix{\fM}{\p}, \qquad \p = 1,\dots,n,
\end{equation}
for any $m \in \N$. In particular $\Oix{\fM}{\p}$ is idempotent when $\flt = \one$.

\begin{lemma}\label{lem:order.1} Let $\Oix{\fM}{\p}, \Oix{\fM}{\q}, \Oix{\fM}{\r}$ be  \eflat\  corner  matrices with  $\p \leq \q \leq \r$ then:
\begin{enumerate} \eroman \dispace
  \item $\Oix{\fM}{\q} \Oix{\fM}{\p} = \zeroM$ for $\p < \q$;
  \item $\Oix{\fM}{\p} \Oix{\fM}{\q}$ is  $(\p,\q)$-corner for $\p \leq \q$;
  \item $\Oix{\fM}{\p} \Oix{\fM}{\q} > \Oix{\fM}{\p} \Oix{\fM}{\r}$ for $\q < \r$;
  \item $ \Oix{\fM}{\q} \Oix{\fM}{\r} > \Oix{\fM}{\p} \Oix{\fM}{\r} $ for $\p < \q$.
\end{enumerate}

\end{lemma}
\begin{proof} 

Since $ \p \leq \q \leq \r$ we have
\begin{equation}\label{eq:str}
   I_\p \subseteq I_\q \subseteq I_\r, \qquad J_\r \subseteq J_\q \subseteq J_\p, \tag{$*$}
\end{equation}
with strict inclusions when $  \p < \q < \r$.

\pSkip
(i) and (ii) are a special case of products of corner matrices, pointed in  Remak \ref{rem:cor.prod} (i) and (iii), respectively.

\pSkip
(iii):  By \eqref{eq:str} and Remark \ref{rem:cor.prod}.(iv)
we have $\Oix{\fM}{\p} \Oix{\fM}{\q} \bgeq \Oix{\fM}{\p} \Oix{\fM}{\r}$, where each nonzero entry  of
 $\Oix{\fM}{\p} \Oix{\fM}{\q}$ and  $\Oix{\fM}{\p} \Oix{\fM}{\r}$ equals $\flt^2$.

\pSkip
 (v):  Use a similar argument as in (iii),  where now $ \Oix{\fM}{\q} \Oix{\fM}{\r} \bgeq \Oix{\fM}{\p} \Oix{\fM}{\r} $.
\end{proof}

\subsection{Tropical linear representations}\sSkip

We write $\Trop^{(n)}$ for the Cartesian product $\Trop \times \dots \times \Trop $, with $\Trop$ repeated $n$ times, considered as a $\Trop$-module over the semiring $\Trop$ (Definition \ref{def:semimodule}), whose operations induced by the operations of $\Trop$.
As in this paper we focus on combinatorial aspects, we identify the associative algebra $\Lin(\T^{(n)})$ of all
tropical linear operators on $\Trop^{(n)}$  with $\Mat_n(\Trop)$.

 A finite dimensional \textbf{tropical (linear) representation} of a
semigroup $\tS$, over $\Trop^{(n)}$, is a semigroup homomorphism
$$\rep: \  \tS \TO \Mat_n(\Trop).$$
A representation $\rep$ is said to be \textbf{faithful} if
it is an injective homomorphism.  As in classical representation theory
one should think of a representation as a \textbf{tropical linear
action} of $\tS$ on the space $\Trop^{(n)}$ (since to every $a \in
\tS$, there is an associated tropical linear operator $\rep(a)$ in $\Lin(\T^{(n)})$
that acts on $\Trop^{(n)}$). Tropically, these representations have an extra digraph meaning (cf. \S\ref{ssec:digraph}), associating a semigroup element to a weighted digraph where the semigroup operation is interpreted as an action of one digraph on another diagraph.

Tropical linear representations were introduced in \cite{IzhakianMargolisIdentity}, applied there to prove that the  bicyclic monoid admits the Adjan identity \cite{Adjan}.  These representations are a major method in the present paper, implemented first to a
finitely generated \fsmg\  (Definition \ref{def:forwardMon}).

\begin{theorem}\label{thm:rep.forward.rep}
  A finitely generated \fsmg\ $\NDSg_n = \genr{ \felm_\ell \ds | \ell = 1,\dots, n }$   has a  tropical linear representation
  $$\rep: \  \NDSg_n \TO \TMat_n(\Trop), \qquad \felm_\ell \longmapsto \Oix{\fM}{\ell}, \ \mfo \longmapsto (\zero),  $$ determined by generators' mapping $\felm_\ell \longmapsto \Oix{\fM}{\ell}$,
where $\Oix{\fM}{\ell}$ are \lcorn\ matrices in $\TMat_n(\Trop)$.
\end{theorem}

\begin{proof} Associativity is clear, while, by Lemma \ref{lem:order.1}.(i), $\rep(\felm_j \felm_i) = \rep(\felm_j) \rep(\felm_i) = \Oix{\fM}{j} \Oix{\fM}{i} = (\zero)$ for any $j > i $.
\end{proof}

Conversely, a finitely  generated matrix semigroup  $\MFrm_n = \genr{ \Pv{\Oix{\fM}{1}}, \;  \Pv{\Oix{\fM}{2} }, \; \dots,  \; \Pv{\Oix{\fM}{n} } } \subset \TMat_n(\Trop)$, where $\Oix{\fM}{\ell}$  are  \eflat\  $\ell$-corner matrices of fixed $\flt$,  is a \fsmg,  cf. Remark \ref{rmk:gen.fw.mon}. Its
nonzero elements are all singular matrices of rank $1$ (Definition \ref{def:matOper}), and thus
$\MFrm_n$ is a singular matrix subsemigroup of $\TMatnT$, cf.  \cite{IDMax}.

\begin{corollary}\label{cor:rep.forward.rep}
The map
  $$\rep_\frd: \  \NDSg_n \TO \MFrm_n, \qquad \felm_\ell \longmapsto \Oix{\fM}{\ell}, \ \mfo \longmapsto (\zero),  $$ determined by generators' mapping $\felm_\ell \longmapsto \Oix{\fM}{\ell}$,
  is a   tropical linear representation of the finitely generated \fsmg\ $\NDSg_n$.
  \end{corollary}
Note that $\rho = \rho_\frd$ is not injective, since by \eqref{eq:eflat.pow}  we have
$$\begin{array}{rll}
 \rho(\felm_i \felm_i \felm_j) & = \rho(\felm_i) \rho( \felm_i) \rho (\felm_j) = (\Oix{\fM}{i})^2
\Oix{\fM}{j}  = \flt \Oix{\fM}{i}\Oix{\fM}{j} \\[1mm] & =\Oix{\fM}{i}\flt \Oix{\fM}{j} = \Oix{\fM}{i}(\Oix{\fM}{j})^2 =
 \rho(\felm_i) \rho( \felm_j) \rho (\felm_j) & =  \rho(\felm_i \felm_j \felm_j), \end{array}
$$
and thus it is not a faithful representation.

\section{Troplactic matrix algebra}\label{sec:troplactic.m.alg}

In this section we utilize the \eflat\ corner matrices to construct a specific  troplactic algebra,  carrying a digraph meaning, to be utilized  later for monoid representations.
For mater of generality we use $\pv$ as a parameter,  assumed taking generic values $> \one$, as it also supports a polynomial view to matrix invariants, e.g., traces and permanent.

\pSkip
Our setup consists of the following auxiliary matrices in $\TMatnT \subset \MatnT$:
\begin{enumerate} \ealph \dispace

\item A collection of  \eflat\ \lcorn\ matrices
$ \Pv{\Oix{\fM}{\ell}}  = \big(\oix{ f} {\ell}_{i,j}\big)$, $\ell = 1,\dots, n$, with fixed $\flt$,
 defined by
\begin{equation}\label{eq:F.mat.2}
\oix{ f} {\ell}_{i,j}= \left\{
            \begin{array}{ll}
              \flt  & \hbox{if } \   i \leq  \ell \leq j \leq n ,\\[2mm]
              \zero & \hbox{otherwise}  .
            \end{array}
          \right.
\end{equation}
These matrices ordered as $\Pv{\Oix{\fM}{1}} < \Pv{\Oix{\fM}{2}}< \cdots < \Pv{\Oix{\fM}{n}}$, and by Lemma \ref{lem:order.1}  generate the forward matrix semigroup (Definition \ref{def:forwardMon})
\begin{equation}\label{eq:fMon}
    \Pv{\MFrm_n} := \genr{ \Pv{\Oix{\fM}{1}}, \;  \Pv{\Oix{\fM}{2} }, \; \dots,  \; \Pv{\Oix{\fM}{n} } } ,
\end{equation}
cf. (Remark \ref{rmk:gen.fw.mon}). (Note that $\MFrm_n$  is a partially ordered semigroup.)

  \item A   \textbf{layout matrix} $\Pv{\eM}$ that is a triangular  idempotent  matrix such that $E  \Pv{\Oix{\fM}{\ell}} = E\Pv{\Oix{\fM}{\ell}}$ for every $\ell = 1, \dots, n$, for simplicity this matrix  is taken to be  the upper triangular matrix $E = (e_{i,j})$ defined as
      \begin{equation}\label{eq:E.mat}
      e_{i,j} = \left\{
            \begin{array}{ll}
              \one  & \hbox{if } \   i  \leq j  ,\\[2mm]
              \zero & \hbox{otherwise}  .
            \end{array}
          \right.
    \end{equation}
 Note that $E$ is both multiplicatively and additively idempotent, i.e, $E = E + E = E^2$.
\end{enumerate}
The above  matrices are used to define the finitely generated matrix algebra
\begin{equation}\label{eq:mAlg}\mfA_n := \genr{  \Oix{A}{\ell} \ds | \Oix{A}{\ell} : = E \+ \Oix{F}{\ell},  \ \ell = 1, \dots, n} \ ,
\end{equation}
whose generators are the triangular matrices $\Oix{A}{\ell} := E \+ \Oix{F}{\ell}$, ordered as
$$\Pv{\Oix{A}{1}} < \Pv{\Oix{A}{2}}< \cdots < \Pv{\Oix{A}{n}},$$
and has the multiplicative identity $E$ and zero $(\zero)$.
We write
\begin{equation}\label{eq:fmMon}
\mfAml_n := (\mfA_n \sm \{(\zero) \}, \cdot \, )  \end{equation} for the multiplicative monoid of $\mfA_n$. (The whole $\mfA_n$ is a pointed multiplicative semigroup, with absorbing element $(\zero)$.)
The monoid $\mfAml_n$ is a nonsingular partially ordered monoid, whose members all have rank~$n$,  but it is not a cancellative monoid.

We start with the  combinatorial structure of the multiplicative monoid $\mfAml_n$, that establishes an important linkage to   digraphs.
The special structure of the generating matrices $\Oix{A}{\ell}$, together with  their digraph realization,
allows to record lengthes  of nondecreasing subsequences of letters (i.e., nondecreasing subwords) by the means of matrix multiplication,
leading  to the following key lemma.

\begin{klemma}\label{klem:clk.represntation}  Let
$L_m = \squ{\ell_1, \ell_2, \dots, \ell_m}$ be a sequence of indexes $\ell_\xi \in \{ 1, \dots, n\}$ with   $\xi =1,\dots,m$.  Let  $U = \prod_{\ell_\xi \in L_m} \Pv{\Oix{A}{\ell_\xi}}$ be a product of generating matrices $\Pv{\Oix{A}{\ell_\xi}}$ in $\mfAml_n$, and write $U = (u_{i,j})$.
  Then, for $i \leq j$,  the $(i,j)$-entry $u_{i,j}$ of the matrix $U$ encodes as the power of $\flt$ the length of the longest  nondecreasing subsequence of $L_m$ that involves only terms from the (convex) subsequence   $\itoj{i}{j} \sqsubseteq \itoj{1}{n}$.
\end{klemma}
In other words, the $(i,j)$-entry $u_{i,j}$ of the matrix $U$ records the length of longest nondecreasing subword of
$\setU$ restricted to the convex sub-alphabet $``\Oix{A}{i}", \dots, ``\Oix{A}{j}"$, cf. Notation \ref{nott:1}.
\begin{proof} We employ the digraph realization of tropical matrices, cf. \S\ref{ssec:digraph}. In this unique realization  the  digraph $\grph_{\Oix{A}{\ell}}$ associated to the matrix  $\Oix{A}{\ell}$ is an acyclic weighted  digraph having
the following properties:
\begin{enumerate} \dispace \ealph
  \item  Every vertex $1, \dots, n$ is assigned with a loops of weight $\one = 0,$ except the vertex $\ell$ whose loop $\rho_\ell$ has weight  $\w(\rho_\ell) = \flt$.

  \item The only directed edges of $\grph_{\Oix{A}{\ell}}$ are
  \begin{enumerate} \dispace
    \item[(i)] $\e_{s,\ell} = (s,\ell)$ for every $s < \ell$;
    \item[(ii)] $\e_{\ell,t} = (\ell,t)$ for every $t > \ell$;
    \item[(iii)] $\e_{s,t} = (s,t)$ only for  $s < \ell < t$;
  \end{enumerate}
   and all of these matrices  have weight $\flt$.
\end{enumerate}

By these properties, given fixed indices $i \leq j$,  we see that for all $s < i \leq j  < t$   the  digraphs $\grph_{\Oix{A}{s}}$ and $\grph_{\Oix{A}{t}}$ have no directed edges $(p,q)$ for any  $i \leq  p <q \leq j$. Therefore, when considering paths between pairs of vertices in $\itoj{i}{j}$ we can ignore the  digraphs $\grph_{\Oix{A}{s}}$ and $\grph_{\Oix{A}{t}}$ with $s < i $, $t >j$.

Let  $\setU$ denote the restoring of the multiplication sequence
$\Oix{A}{\ell_1} \cdots \Oix{A}{\ell_m}$ as a word  (Notation \ref{nott:1}), and let $\grph_\setU$ denote the graph
 $$\grph_{\Oix{A}{\ell_1} \cdots \Oix{A}{\ell_m}} = \bigcup_{\ell = 1}^n\grph_{\Oix{A}{\ell}},$$ while  $U$ stands for the matrix product.
Recall from  \S\ref{ssec:digraph} that any  considered path $\gm_{i,j}$ in $\grph_\setU$ always
respects the edges' coloring determined by the sequence $\squ{\ell_1, \ell_2, \dots, \ell_m}$,
$\ell_\xi \in \{ 1, \dots, n\}$, or equivalently by $\setU$, and is  denoted  by
$[\gm_\setU]_{i,j}$ to indicate this coloring.

Taking $i \leq j$, an edge  $\e_{i,j}$ from $i$ to $j$ in $\grph_U$ corresponds
  to a path $[\pth_\setU]_{i,j}$ of highest weight
   from $i$ to $j$ of length ~$m$ in $\grph_\setU$.
     By  properties (a)-(b) above, we learn that the possible contribution of the diagraphs  $\grph_{\Oix{A}{s}}$ and $\grph_{\Oix{A}{t}}$,   $s < i \leq j < s$, to $[\pth_\setU]_{i,j}$  could only be loops, all having weight $\one =0$. This means, that we can reduce $[\pth_\setU]_{i,j}$ to a sub-path $[\pth_\setUp]_{i,j}$ (exactly of the same weight, but perhaps shorter)
   whose edges coloring is determined by the subsequence $L'_{m'} = \squ{\ell'_1,\dots, \ell'_{m'}} \sqsubseteq \squ{\ell_1, \ell_2, \dots, \ell_m}$,  $m' \leq m$, with $\ell'_\xi \in \{ i, \dots, j\}$.  Writing $\setUp$ for the multiplication sequence
$\Oix{A}{\ell'_1} \cdots \Oix{A}{\ell'_{m'}}$, realized a word, then  $ [\pth_\setUp]_{i,j}$  is a colored path from $i$ to $j$ of highest  weight and length $m'$ in  $\grph_\setUp$.

   Let $\incS =  \squ{\ell''_1, \ell''_2, \dots, \ell''_{m''}} \sqsubseteq \squ{\ell'_1, \ell'_2, \dots, \ell'_{m'}}$, $\ell''_\xi \in \{ i, \dots, j\}$, $m'' \leq m'$, be a longest nondecreasing subsequence of $L'_{m'}$.
   Let $$i \leq p_1 < p_2 < \cdots < p_r \leq  j, \qquad r \leq j-i+1, $$ be the (distinct)  elements of $\{ i, \dots, j\}$ that take part in $\incS$, and let $q_\xi$, $\xi = 1, \dots, r$, be the number of occurrences of $p_\xi$ in $\incS$.  (These occurrences must be sequential occurrences, as  $\incS$ is nondecreasing.)
Accordingly,
\begin{equation}\label{eq:str}
  \len{\incS} = \sum_{\xi = 1}^r q_\xi  = m''. \tag{$*$}\end{equation}
We denote by $\setP$  the multiplication sequence
$\Oix{A}{\ell''_1} \cdots \Oix{A}{\ell''_{m''}}$, determined by $\incS$ and realized as a word. (Namely $\setP \wrdsset \setW$.)

    Assuming first that $\flt = 1$,   we claim that  $\w(\pth'_{i,j}) = \len{\incS}$ for  $\pth'_{i,j} := [\pth_\setUp]_{i,j}$.
  To prove   that  $\w(\pth'_{i,j}) \geq \len{\incS}$, consider the (colored) path
  $$ [\mu_\setP]_{i,j} = (\rho_{p_1})^{q_1 -1 } \circ \e_{p_1, p_2} \circ (\rho_{p_2})^{q_2-1} \circ \e_{p_2, p_3 } \circ \cdots \circ (\rho_{p_{r-1}})^{q_{r-1} -1 } \circ \e_{p_{r-1}, p_r, } \circ (\rho_{p_r})^{q_r}$$ of length $m''$ in $\grph_\setP$, obtained from $\setP$ by \eqref{eq:str}.
  All the edges and loops in $\mu_{i,j} : = [\mu_\setP]_{i,j} $ have weight~$1$, and thus $\w(\mu_{i,j}) = m''$. Clearly,
  $\w(\pth'_{i,j}) \geq \w(\mu_{i,j})  $, since otherwise by plugging in loops of weight $0$ we could expand the path $ \mu_{i,j}$ to  a properly colored path $ \tlmu_{i,j}$ in $\grph_\setUp$ with  $\w(\tlmu_{i,j}) = \w(\mu_{i,j})$ to get that $\w(\tlmu_{i,j}) > \w(\pth'_{i,j})  $, contradicting   the weight maximality of the path $\gm'_{i,j}$ in $\grph_\setUp$.

  On the other hand,  if $\w(\pth'_{i,j}) > \w(\mu_{i,j})$, as $\pth'_{i,j}$ is an acyclic path, by excluding all loops of weight~ $0$ we could extract from $\pth'_{i,j}$ a (properly colored) acyclic sub-path $\pth''_{i,j}$  with $\w(\pth''_{i,j}) =\w(\pth'_{i,j})$. Then, the coloring of edges in $\pth''_{i,j}$ determines a nondecreasing sequences in $L'_{m'}$ of length greater than $ \w(\mu_{i,j}) = \len{\incS} $, contradicting the length maximality of $\incS$.
Composing together, we therefore have   $\w(\pth_{i,j}) = \w(\pth'_{i,j}) =  \w(\mu_{i,j}) = \len{\incS}$.

To complete the proof for a generic  $\flt$, just observe that $m = 1^m$ and that the map $\flt^m \mTo m
$, $m \in \N$, is a  bijection.
  \end{proof}
 Key Lemma \ref{klem:clk.represntation} has dealt with the  multiplicative (monoid) structure of  $\mfAml_n$, and  now we turn to discuss  the  semiring structure of $\mfA_n$.

\begin{theorem}\label{thm:plac.mat.relations} The finitely generated matrix algebra $\mfA_n := \genr{ \Oix{A}{1}, \dots, \Oix{A}{n} }$ of \eqref{eq:mAlg}
is a troplactic algebra $\APlc_n$ (Definition~\ref{defn:trop.plc.alg}).  Thus,
every triplet of it generators $$A= \Oix{A}{p}, \quad  B = \Oix{A}{q}, \quad  C=\Oix{A}{r}, \qquad p \leq q \leq r,$$  admit the Knuth relations \eqref{eq:knuth.rel}:
\begin{equation*}\label{eq:plac.Alg.2}
  \begin{array}{ccccccc}
   \text{\KXa}:&& A  C B =  C  A  B  & \text{if} &  p \leq q < r , & \qquad & \\
   \text{\KXb}:&& B A  C  = B C  A  & \text{if} & p <  q \leq r ,   \end{array}
\end{equation*} and these matrix products are different from   $ABC$ and $CBA$.

\end{theorem}
\begin{proof}
We verify the axioms of troplactic algebra in  Definition \ref{defn:trop.plc.alg}.
\begin{description} \eroman \dispace
 \item[\PXa] $(\mfa =  \mfe + \mfa.)$    \quad
  By construction, and the additive idempotence  of $\MatnT$ we have   $$  \Oix{A}{p} =  \eM + \Oix{\fM}{q} = \eM + \eM + \Oix{\fM}{p} =  \eM  + \Oix{A}{p}.$$

  \item[\PXb] ($\mfb \mfa = \mfa + \mfb$  if  $\mfb >  \mfa $.) \quad
   By Lemma \ref{lem:order.1}.(i) we have \begin{align*}
        \Oix{A}{q} \Oix{A}{p} & = (\eM + \Oix{\fM}{p})(\eM + \Oix{\fM}{p})
         = \eM + \Oix{\fM}{p} + \Oix{\fM}{q} \\ & = (\eM + \Oix{\fM}{p}) +  (\eM+ \Oix{\fM}{q}) =
        \Oix{A}{p} + \Oix{A}{q}.
      \end{align*}
 \item[\PXc] ($\mfa (\mfb +  \mfc) = \mfa \mfb + \mfc.$) \quad
   By Lemma \ref{lem:order.1}.(iii) we have \begin{align*}
        \Oix{A}{p} ( \Oix{A}{q} +  \Oix{A}{r})   &= (\eM + \Oix{\fM}{p})(\eM + \Oix{\fM}{q}  + \Oix{\fM}{r})
        \\
        & =\eM + \Oix{\fM}{p} + \Oix{\fM}{q} + \Oix{\fM}{r} + \Oix{\fM}{p}\Oix{\fM}{q} + \Oix{\fM}{p}
        \Oix{\fM}{r}  \\ &   =\eM + \Oix{\fM}{p} + \Oix{\fM}{q} + \Oix{\fM}{r} + \Oix{\fM}{p}
        \Oix{\fM}{q}  \\ &   =(\eM + \Oix{\fM}{p} + \Oix{\fM}{q}  + \Oix{\fM}{p}
        \Oix{\fM}{q})+(E + \Oix{\fM}{r})   =  \Oix{A}{p}  \Oix{A}{q} +  \Oix{A}{r}.
      \end{align*}

  \item[\PXd] ($(\mfa + \mfb ) \mfc = \mfa + \mfb \mfc.$) \quad
  By Lemma \ref{lem:order.1}.(iv) we have \begin{align*}
        (\Oix{A}{p} + \Oix{A}{q})  \Oix{A}{r}
         &= (\eM + \Oix{\fM}{p} + \Oix{\fM}{q} )(E + \Oix{\fM}{r})
        \\&=\eM + \Oix{\fM}{p} + \Oix{\fM}{q} + \Oix{\fM}{r} + \Oix{\fM}{p}\Oix{\fM}{r} + \Oix{\fM}{q}
        \Oix{\fM}{r}  \\ &   =\eM + \Oix{\fM}{p} + \Oix{\fM}{q} + \Oix{\fM}{r} + \Oix{\fM}{q}
        \Oix{\fM}{r} \\ & = (\eM + \Oix{\fM}{p}) + (E +  \Oix{\fM}{q} + \Oix{\fM}{r} + \Oix{\fM}{q}
        \Oix{\fM}{r})  =  \Oix{A}{p}  + \Oix{A}{q}  \Oix{A}{r}.
      \end{align*}

\end{description} Use Key Lemma~ \ref{klem:clk.represntation}, to see that the products $U= ABC$ and $V = CBA$ are differ from the products $X = A  C B$ and $Y = B A  C $, which gives
$u_{p,r} = \pv^3$, $v_{p,r} = \pv$ for the $(p,r)$-entry of $U = (u_{i,j})$, $V=(v_{i,j})$ respectively, while $x_{p,r} = y_{p,r} = \pv^2$ for the $(p,r)$-entry of $X= (x_{i,j})$ and $Y= (y_{i,j})$.
\end{proof}

The theorem shows that for three letter words
$u \clcong v$ iff $u \kcong v$, cf. \eqref{eq:clk.cong},
 but in general this relation does not hold for words of an arbitrary length, as seen later in
Example \ref{exmp:run:1}.

\begin{proposition}\label{prop:Frob.2} Any two generators  $\Oix{A}{k}$ and $\Oix{A}{\ell}$ admit the Frobenius property:
\begin{equation*}\label{eq:Frob.2}
    \big(\Pv{\Oix{A}{k}} + \Pv{\Oix{A}{\ell}}\big)^m = \Pv{\Oix{A}{k}}^m + \Pv{\Oix{A}{\ell}}^m.
\end{equation*}

\end{proposition}
\begin{proof} Immediate by Lemma \ref{lemma:Frob.2}, since $\mfA_n$ is a troplactic algebra.
\end{proof}

Let $W = \Oix{A}{\ell_1} \cdots \Oix{A}{\ell_m} $, $W = (w_{i,j})$,  be a matrix product in $\MPlc_n$, realized as a word $\setW$ over the symbols $``\Oix{A}{\ell_1}",  \dots, ``\Oix{A}{\ell_m}"$. Setting $\pv = 1$, the multiplicative  monoid $\MPlc_n$ has  two main characters:
\begin{enumerate} \eroman
  \item The additive trace
   \begin{equation}\label{eq:char.add} \chra : \MPlc_n \TO \N_0, \qquad  W \longmapsto \tr(W),
\end{equation}
that provides the maximum occurrences of a letter in $\setW$.

  \item The multiplicative trace
\begin{equation}\label{eq:char.mlt}
\chrp : \MPlc_n \TO \N_0, \qquad  W \longmapsto \mtr(W),
\end{equation}
that gives the length of the matrix word  $\setW$.
\end{enumerate}

Considering $\N_0$ as the tropical semiring $(\N_0, \+, \cdot \, )$, these characters preserve  their type,  additive  and multiplicative respectively,
\begin{align*}
  \chra(U \+ V) & = \chra(U) \+ \chra(V),  \qquad
  \chrp(U V) = \chrp(U) \chrp(V),
\end{align*}
for any $U,V \in \MPlc_n$.
 Note that $\chra(W) = \chrp(W)$ iff $\setW$ consists of a single letter,
 where $\chra(W) = \chrp(W) = 0$ iff
$\setW = \eM$. For $\chra$ and $\chrp$ we also have the properties
\begin{align*}
  \chra(U  V)  \leq \chra(U) \chra(V),  & \qquad
  \chrp(U \+ V)  \leq   \chrp(U) \+  \chrp(V),
\end{align*}
and hence
\begin{align*}
  \chra(U \+ V) \leq \chra(U  V) , \qquad
  \chrp(U \+ V)  \leq   \chrp(U V).
\end{align*}
The following conclusion is then evident.
\begin{corollary}\label{cor:factorziation}
  If $\Oix{A}{\ell_1} \cdots \Oix{A}{\ell_m } = \Oix{A}{\ell'_1} \cdots \Oix{A}{\ell'_{m'} } \neq \eM,$ $\ell_t, \ell'_s \in \{1,\dots, n \},$ then $m =m'$.
\end{corollary}

\section{The \cloak\  and the \ckmon s}\label{sec:rep.clocktic}

\pSkip

We gently coarser the Knuth relations \eqref{eq:knuth.rel} to introduce a new monoid structure over the alphabet $\tA_I = \{\aa_\ell \ds | \ell \in I  \}$ whose underlying equivalence is  based on the forthcoming function.
Given a word $w \in \tA_I$ and a convex set $J \cnxsset I$ (Definition \ref{defn:convex}),  we define
 the set $\sbwd{J}{w}$ of all subwords $u$ of  $w$ with letters in the convex sub-alphabet $\tA_J \cnxsset \tA_I$, i.e.,
$$ \sbwd{J}{w} := \{\, u \wrdsset w \ds| u \in \tA^+_J \, \}, \qquad  J \subseteq I \text{ is convex},$$
and specify  the subset  $\incsbwd{J}{w} $ consisting  of all nondecreasing  words in $\sbwd{J}{w}.$
The function $\mxlen{J}{w}$ gives that length of the longest nondecreasing  word  $u \in \incsbwd{J}{w}$, possibly non-unique,  that is
\begin{equation}\label{eq:max.inc.len}
 \mxlen{J}{w}:= \max\{\, \len{u} \ds | u  \in \incsbwd{J}{w} \, \}.
\end{equation}
Relying upon on this function, rather than on recombining the underlying relations of $\kcong$ (e.g., as in the Chinese monoid \cite{Cass}),  we define the following new monoid, slightly coarsening the plactic monoid (Definition ~ \ref{def:plactic.mon}).

\begin{definition}\label{def:cloaktic.mon}

The \textbf{\kmon} is the monoid $\clkM_I := \CLK(\tA_I)  $, generated by an ordered  set of elements \footnote{For simplicity, we assume a countable set of generators, where the  generalization to an arbitrary  ordered set of generators is obvious. Moreover, one can also generalize this monoid by considering a partially ordered set of generators, but such theory is more involved.} $\tA_I : = \{\aa_\ell \ds | \ell  \in I  \}$,  subject to the equivalence relation $\clcong$, defined as 
\begin{equation}\label{eq:clk.cong}
u \clcong v \Iff \mxlen{J}{u} = \mxlen{J}{v} \quad  \text{for all convex subsets } J \cnxsset I. 
\end{equation}
Namely $\clkM_I := \tA_I^*/ _{\clcong}$.
  When $|I| = n$ is finite, we write $\clkM_n$ for $\clkM_I$ and say that $\clkM_n$ is finitely generated of \textbf{rank} $n$.

\end{definition}
\noindent  We write $\CLK_I$  for $\CLK(\tA_I)$ when the alphabet $\tA_I$ is arbitrary.  Henceforth,  we always assume that $e < \aa_\ell$ for all $\aa_\ell \in \clkM_I$.

\begin{lemma}\label{lem:clk.cong}
  The equivalence relation $\clcong$  is a congruence on the free monoid $\tA_I^*$.
\end{lemma}
\begin{proof}
  Suppose that $u_1 \clcong v_1$ and $u_2 \clcong v_2$,  and assume that  $u_1 u_2 \not \clcong v_1 v_2$.
  Then, say,  $\mxlen{J}{u_1 u_2} > \mxlen{J}{v_1 v_2}$ for some $\tA_J \cnxsset \tA_I$. Let
  with  $w \wrdsset u_1 u_2$ a longest nondecreasing subword of $u_1 u_2$ with  letters from $\tA_J$.  As such  $w$ decomposes as $w = w_1 w_2$, where $w_1 \wrdsset u_1$, $w_2 \wrdsset u_2$ are nondecreasing.   Accordingly $w_1 \in \tA_{J_1}^*$,  $w_2 \in \tA_{J_2}^*$ such that
  $\tA_{J_1} \cap \tA_{J_2} \subset \tA_J$ where
   \begin{equation}\label{eq:str}
\text{
  $\tA_{J_1} \cap \tA_{J_2} = \{ \lt_\ell \}$ for some letter $\lt_\ell \in \tA_J$, or
  $\tA_{J_1} \cap \tA_{J_2} = \emptyset$.} \tag{$*$}
\end{equation}
Since  $u_1 \clcong v_1$ and $u_2 \clcong v_2$, there are nondecreasing subwords  $w'_1 \wrdsset v_1$ in $\tA_{J_1}^*$ and  $w'_2 \wrdsset v_2$ in $\tA_{J_2}^*$  satisfying  $\wlen(w'_1) = \wlen(w_1)$ and  $\wlen(w'_2) = \wlen(w_2)$. Then \eqref{eq:str} ensures that the concatenation $w'_1 w'_2 $ is a nondecreasing subword of $u_2 v_2$ and therefore has length $\wlen(w_1) \wlen(w_2)$, and hence $\mxlen{J}{u_1 u_2} = \mxlen{J}{v_1 v_2}$ -- a contradiction.
\end{proof}

When $u \clcong v$, for $u,v \in \tA_I^*$,  we say that the words $u$ and $v$ are \textbf{\keqv}.
That is $u \clcong v$, if over any convex sub-alphabet $\tA_J$ of $\tA_I$ the length of longest nondecreasing subwords in $u$ and $v$ is the same. In particular each letter $\lt_\ell \in \tA_I$ is \keqv\ only to itself.

\begin{example}\label{exmp:bicyc-clk}
The \textbf{bicyclic monoid} is the monoid $\tB := \tA_2^*/_{\bcong}$ generated by two ordered elements $\aa < \bb$, subject to  the  relation $\bcong$ determined by
$\aa \bb = e$.
As well  known (e.g. \cite{CP}), each  element of $w \in \tB$ can  be canonically written as
\begin{equation*}\label{eq:bcy}
w = b^i a^j, \quad \text{for unique } i,j \geq 0\; . 
\end{equation*}
Accordingly, for any $w \in \tB$ we have $$\mxlen{\{ 1,2 \} }{w} = \max\big\{ \, \mxlen{\{ 1\}}{w}, \;  \mxlen{\{ 2 \}}{w} \, \big  \} \; .$$
Thus, we see that on $\tA_2^*$ the congruence $\bcong$ implies $\kcong$.

\end{example}

By Definition \ref{def:cloaktic.mon} we have the following obvious properties.
\begin{properties}\label{proper:clk} For two \keqv\ words $u \clcong v$ in $\clkM_I$  the below properties hold.
\begin{enumerate} \eroman
  \item  Each  letter $\lt_i\in \varX_I$ appears in $u$ and in~$v$ exactly the same times, i.e., the formal relation $u \clcong v$ is  balanced (cf. \S\ref{sec:3.2}).
  \item The total length of $u$ and $v$ is the same, i.e., $\wlen(u) = \wlen(v)$.
  \item If $u$ is nondecreasing, or nonincreasing,  then $u = v$.
  \item The equivalence $u \clcong v$ does not imply  that set theoretically  $\incsbwd{J}{u} =  \incsbwd{J}{v} $, $ J \cnxsset I$, nor even that longest words in $\incsbwd{J}{u}$ and  $\incsbwd{J}{v} $ are the same.
\end{enumerate}

\end{properties}

Recall that $\pcong$ denotes congruence determined by the plactic relations, which are the Knuth relations, and hence $\pcong$ is exactly $\kcong$.

\begin{remark}
\label{rmk:clk.factor} The monoid  homomorphism $\tA_I^* \To \clkM_I$ factors through the plactic monoid (Definition~ \ref{def:plactic.mon})
$$ \tA_I^* \ONTO \tA_I^*/_ {\pcong} \ONTO \tA_I^*/ _{\clcong},$$
as will be seen later in  Proposition \ref{prop:tab-to-cloc} and Corollary \ref{cor:plc.to.clk}.
Furthermore, from this factorization,  we can conclude that elements $w \in \clkM_2$ of the  \kmon\ of rank $2$ have a canonical form
 $$ w= b^i a^j b^k, \qquad i \leq j, \ i,j,k \in \N_0,$$
 generalizing  Example \ref{exmp:bicyc-clk}.

\end{remark}

\subsection{Linear representations of the \kmon}\label{ssec:clk.rep}
\sSkip

The troplactic matrix algebra $\mfA_n = \genr{\Oix{A}{1}, \dots \Oix{A}{n}}$, as defined in \eqref{eq:mAlg}, is now utilized to introduce a linear representation of the finitely generated \kmon\ $\clkM_n = \genr{\lt_1, \dots, \lt_n}$ of rank $n$ (Definition~\ref{def:cloaktic.mon}).

\begin{theorem}\label{thm:clk.represntation}
The map \begin{equation}\label{eq:clk.rep}
\clkrep: \clkM_n \TO \mfAml_n, \qquad \lt_\ell \longmapsto \Pv{\Oix{A}{\ell}}, \quad \ell  = 1, \dots, n,
\end{equation}
defined by mapping of generators,  i.e.,
$$ \clkrep(w) = \clkrep(\lt_{\ell_1}) \cdots \clkrep(\lt_{\ell_m}), \qquad w = \lt_{\ell_1} \cdots \lt_{\ell_m} \in \tA_n^*,$$
and sends the empty word $e$ to $E$, is a monoid isomorphism -- a  faithful linear representation of the \kmon\ $\clkM_n$ of rank $n$.

\end{theorem}

\begin{proof}
  Let  $L_m = \squ{\ell_1, \ell_2, \dots, \ell_m}$ and  $L'_{m'} = \squ{\ell'_1, \ell'_2, \dots, \ell'_{m'}}$ be two sequences with $\ell_t, \ell'_t \in \{ 1, \dots, n\}$. Write   $u = \lt_{\ell_1} \cdots \lt_{\ell_m}$ and  $v = \lt_{\ell'_1} \cdots \lt_{\ell'_{m'}}$  for   the corresponding  words in  $\clkM_n$ and consider  their linear representations
  \begin{align*}
  U = \clkrep(u) & =\clkrep(\lt_{\ell_1}) \cdots \clkrep(\lt_{\ell_m}) =\Oix{A}{\ell_1} \cdots \Oix{A}{\lt_m},
   \\  V = \clkrep(v) & =\clkrep(\lt_{\ell'_1}) \cdots \clkrep(\lt_{\ell'_{m'}}) =\Oix{A}{\ell'_1} \cdots \Oix{A}{\ell'_{m'}},   \end{align*}
 written as  $U= (u_{i,j})$ and $V = (v_{i,j})$.

  By Key Lemma \ref{klem:clk.represntation},
 the $(i,j)$-entry $u_{i,j}$, $i \leq j$,  of the matrix $U$ gives the length of the longest nondecreasing subsequence of $L_m$ that involves  terms only from  the convex subsequence  $\itoj{i}{j} \sqsubseteq \itoj{1}{n}$.  The same holds for the $(i,j)$-entry $v_{i,j} $ with respect to $L'_{m'}$.
  Taking  the generators corresponding to $L_{m}$ and $L'_{m'}$ in $\clkM_n$, by the definition of the congruence  $\clcong$, it  implies that $u \clcong v$ iff
$U = V$.
\end{proof}

\begin{corollary}\label{cor:clk.knuth} The \kmon\ $\clkM_n$ satisfies the Knuth relations  \eqref{eq:knuth.rel}.
\end{corollary}
\begin{proof}
$\clkrep: \clkM_n \To \mfAml_n$ is isomorphism by Theorem \ref{thm:clk.represntation},  where $\mfAml_n$ satisfies the Knuth relations \eqref{eq:knuth.rel} by Theorem~ \ref{thm:plac.mat.relations}.
  \end{proof}

\begin{corollary}\label{cor:clk.id} The \kmon\ $\clkM_n$ admits all the semigroup identities satisfied by  $\TMatnT$, in particular the semigroup identities  \eqref{eq:iduniv2.2}.
\end{corollary}
\begin{proof}
  $\clkM_n$ is faithfully represented by $\mfAml_n$ -- a submonoid of $\TMatnT$ -- which by Theorem \ref{thm:trMat.Id} satisfies the identities  \eqref{eq:iduniv2.2}.
\end{proof}

The  map $\clkrep$ does not record explicitly longest convex subwords,
but only their lengths. However, by Key Lemma \ref{klem:clk.represntation}, we see that each diagonal entry $u_{\ell,\ell}$ of $U = \clkrep(w)$ records precisely the number of times that the letter $\lt_\ell$ appears in $w$, and thus also  the total length of $w$ as the multiplicative trace \eqref{eq:char.mlt}. The additive trace \eqref{eq:char.add} gives the maximal occurrence of a letter,

\begin{lemma}\label{lem:injec.ndc} Let $T_n \subset \tA_n^*$ be the subset of all nondecreasing words over the alphabet  $\tA_n$.
Then the restriction
$$ \clkrep|_{T_n} : T_n \TO  \MPlc_n,
$$
of the homomorphism \eqref{eq:clk.rep} to $T_n$ is a bijective map.
\end{lemma}

\begin{proof} A (nondecreasing) word  $w \in T_n$ is  written uniquely as  $$w = \lt_{1} ^{q_1} \cdots
\lt_n^{q_n}, \qquad  q_k \in
\N_0.$$
Let $U= \clkrep(w)$, $U = (u_{i,j})$,  be the image of $w$ in $\mfAml_n$.
By  Key Lemma \ref{klem:clk.represntation} we have
$$u_{1,k} = \sum_{t =1 }^k q_t, \qquad \text{for every  } k = 1,\dots,n,$$
and therefore the $q_k$'s are uniquely computed in terms of  the entries of $U$
as
$$q_1 = u_{1,1} \ \text{ and } \ q_k = u_{1,k}-u_{1,k-1} \text{ for }  k =2,\dots,n. $$
(Alternatively, $q_k = u_{k,k}$, for every $k = 1, \dots, n$.)
\end{proof}

\begin{example}\label{exmp:mat.clk.3x3}
We provide the full details of the tropical linear  representation
\begin{equation*}
\clkrep: \clkM_3 \TO \mfAml_3, \qquad \lt_\ell  \mTo \Oix{A}{\ell}, \quad \ell = 1,2,3,
 \end{equation*}
 of the \kmon\ $\clkM_3 = \CLK(\tA_3)$ of rank $3$,
given by the generators'  map
\begin{align*}
\lt_1 \mTo
A= \mresize{
\begin{bmatrix}
 1 & 1 & 1 \\
  & 0  & 0  \\
   &   & 0
\end{bmatrix}},   \quad
\lt_2 \mTo  B= \mresize{
\begin{bmatrix}
 0 & 1 & 1 \\
  & 1 & 1 \\
   &   & 0
\end{bmatrix} } ,
 \quad
\lt_3 \mTo C = \mresize{
\begin{bmatrix}
 0 & 0 & 1 \\
  & 0 & 1  \\
   &   & 1
\end{bmatrix} } ,
\end{align*}
where $A = \Oix{A}{1}$, $B = \Oix{A}{2}$, $C = \Oix{A}{3}$, and $\pv = 1$.
(Recall that $\one := 0$, and the empty space stand for $\zero := -\infty$.)
For these matrix generators  we obtain
\begin{align*}
ACB = CAB = \mresize{
\begin{bmatrix}
 1 & 2 & 2 \\
   & 1 & 1 \\
   &   & 1
\end{bmatrix} }
,  \qquad  BAC = BCA = \mresize{
\begin{bmatrix}
 1 & 1 & 2 \\
   & 1 & 2 \\
   &   & 1
\end{bmatrix}},
\end{align*}
where for pairs of generators we have
\begin{align*}
ABA = BAA = \mresize{
\begin{bmatrix}
 2 & 2 & 2 \\
   & 1 & 1 \\
   &   & 0
\end{bmatrix} }
 , \qquad BBA = BAB = \mresize{
\begin{bmatrix}
 1 & 2 & 2 \\
   & 2 & 2 \\
   &   & 0
\end{bmatrix} }
,
\end{align*}

\begin{align*}
ACA = CAA=\mresize{
\begin{bmatrix}
 2 & 2 & 2 \\
   & 0 & 1 \\
   &   & 1
\end{bmatrix}},
\qquad   CCA= CAC = \mresize{
\begin{bmatrix}
 1 & 1 & 2 \\
   & 0 & 2 \\
   &   & 2
\end{bmatrix}}
,
\end{align*}

\begin{align*}
BCB = CBB= \mresize{
\begin{bmatrix}
 0 & 2 & 2 \\
   & 2 & 2 \\
   &   & 1
\end{bmatrix}},
\qquad  CCB= CBC
= \mresize{
\begin{bmatrix}
 0 & 1 & 2 \\
   & 1 & 2 \\
   &   & 2
\end{bmatrix}},
\end{align*}
while
\begin{align*}  ABC = \mresize{
\begin{bmatrix}
 1 & 2 & 3 \\
   & 1 & 2 \\
   &   & 1
\end{bmatrix}},
\qquad CBA = \mresize{
\begin{bmatrix}
 1 & 1 & 1 \\
   & 1 & 1 \\
   &   & 1
\end{bmatrix}}
.
\end{align*}

Finally, by Corollary \ref{cor:clk.id}, the \kmon\   $\clkM_3$ of rank $3$ satisfies the semigroup identities \eqref{eq:iduniv2.2}, in particular  \eqref{eq:id2}:
\begin{equation*} \Id_{(C,2,2)}: \quad
    \ll2 \ds{\underline{x}}  \ll2 \ds =  \ll2 \ds{\underline{x}}
     \ll2 \
\end{equation*}
by letting  $x = uv$ and $y=vu$, for any   $u,v \in  \clkM_3$.
\end{example}

\begin{remark}\label{rem:reverasl.rep}
Let $\std{w} = \lt_{\ell_1} \cdots \lt_{\ell_m}$ be a word in the free monoid $\tA^*_n$, represented  in $\mfAml_n$  by the matrix product $\Oix{A}{\ell_1} \cdots \Oix{A}{\ell_m }$, via the homomorphism
$ \clkrep:  \tA_n^* \To \mfAml,$
 cf. \eqref{eq:clk.rep}.
The reversal $\rvs{w}$ of the word $\std{w}$  is represented in $\mfAml_n$ by the product
$$ \clkrep:  \rvs{w} \longmapsto \big((\Oix{A}{\ell_1})^\trn \cdots (\Oix{A}{\ell_m })^\trn\big)^\trn.$$
 Indeed, by standard composition of transpositions we have:
\begin{align*}
  \clkrep(\rvs{w}) =  \Oix{A}{\ell_m} \cdots \Oix{A}{\ell_1}&  = \big((\Oix{A}{\ell_m})^\trn\big)^\trn \cdots \big((\Oix{A}{\ell_1})^\trn\big)^\trn \\ & = \big((\Oix{A}{\ell_1})^\trn \cdots (\Oix{A}{\ell_m })^\trn\big)^\trn.
\end{align*}
\end{remark}

In digraph view, cf. \S\ref{ssec:digraph}, the transposition of a matrix $A = (a_{i,j})$ is interpreted as the redirecting of all edges in the associated digraph $\grph_A := (\ver, \arc)$ of $A$. Namely,
replacing each edge $\e_{i,j} \in \arc$ by the edge $\e_{j,i}$ with opposite direction, but with the same weight. Thus, by Key Lemma \ref{klem:clk.represntation}, it means that $\clkrep(\rvs{w})$ gives the lengths of longest  nonincreasing subswords in $\std w$ over the  convex sub-alphabets  of $\tA_n$.

\begin{remark}\label{rem:reveras.clack}
Two words $u,v \in \clkM_n$  can be \keqv, i.e., $ u \clcong v $, while their reversals $\rvs{u}$ and $\rvs{v}$ are not \keqv, or vise versa.  For example, take the words  $ u = c \; b^2 c^2 \; a^2 b^2 c $ and    $v = c^2 \; b^2 c \; a^2 b^2 c $ for which $\rvs{u} = c b^2 a^2 c^2 b^2 c$ and
$\rvs{v} = c b^2 a^2 c b^2 c^2$.  Hence $ u \clcong v $ but $ \rvs u \not \clcong \rvs v $.
\end{remark}
The  example in the  remark is a pathological example, and it will have further uses in the paper.

\begin{observation}\label{obs:clk.additive} Let $u, v \in \clkM_n$ be elements represented in  $\mfAml$ by  $U  = \clkrep(u)$, $V= \clkrep(v)$, and consider the matrix sum $W = U + V$. Since
  by  Theorem  \ref{thm:clk.represntation}
the representation   $\clkrep: \clkM_n \To \mfAml_n$, cf. \eqref{eq:clk.rep},  is isomorphism then $w = \inv{\clkrep}(W)$ is word in $\clkM_n$ that contains  $u$ an $v$ as subwords, in which the lengths of longest  nondecreasing subwords \eqref{eq:max.inc.len} satisfy
$$ \mxlen{J}{w} \leq  \mxlen{J}{u } + \mxlen{J}{v}, \qquad \mxlen{J}{w} =  \max \{ \, \mxlen{J}{u } , \, \mxlen{J}{v} \, \},  $$
for every convex sub-alphabet $\tA_J \cnxsset \tA_n.$
\end{observation}

From this observation it follows that the additive monoid of $\mfA_n$ has the structure of a lattice.

\subsection{Algorithm and complexity}
\sSkip

We utilize  our faithful representation of the \kmon\ (Theorem \ref{thm:clk.represntation}, \S\ref{ssec:clk.rep}) to provide the following efficient algorithm.
By ``roughly similar length'' we mean lengths that are equal up to $\pm 1$.

\begin{algorithm}\label{algr:clk}  Given a word $w \in \tA_n^+$ over a finite alphabet $\tA_n = \{\lt_1, \dots, \lt_n \} $, find the maximal lengthes  of its nondecreasing subwords over each  convex sub-alphabet $\tA_J \cnxsset \tA_n$.
\begin{enumerate} \dispace
  \item Decompose $w$ into two subwords of roughly similar length, repeat this decomposition as  long as possible.

  \item Represent each letter $\lt_\ell$ of $w$ by the matrix $\Pv{\Oix{A}{\ell}}$ as in \eqref{eq:mAlg}.

  \item Roll  back the recursive decomposition of step (1) where in each backward step multiply the resulting matrices obtained from the previous step.
  \end{enumerate}
\end{algorithm}

For complexity considerations,  to be compatible with the customarily notation, we switch notation and denote the  number of letters in the underlining alphabet $\tA_m$ by $m$; the length of the input (i.e., the length of the input   word) is then denoted by  $n$.
We assume that $\tA_m$ is a finite alphabet.
\begin{complexity*} Algorithm \ref{algr:clk} has a  linear time complexity $O(n)$.
\end{complexity*}
\begin{proof}
Reading an input word $w \in \tA_m^+$ of length $n$ over a finite alphabet $\tA_m$, is performed in linear time.
  Computing the product of two $m \times m$ \eflat\ corner matrices requires  a constant time of  order ~ $m^2$. As the algorithm  is recursive, we have $\log(n)$ matrix multiplications, which in total takes   $m^2 \log(n)$ operations. Putting all together, as $m$ is constant we have $O(\log(n))$ effective operations and reading in $O(n)$, which sum up to $O(n)$.
\end{proof}
To the best of our knowledge, no other known  algorithm offers such efficient time complexity for this problem.

\subsection{The \ckmon}\label{ssec:co.clk.rep} \sSkip

In the view of Remark \ref{rmk:mir.smon}, we employ  the \cmirr\ of words (Definition  \ref{def:mirror}) to introduce the following monoid -- a second coarsening of the plactic monoid.

\begin{definition}\label{def:co.cloaktic.mon}
The \textbf{\ckmon} of rank $n$ is the monoid $\coclkM_n := \CCLK(\tA_n)  $ generated by a finite ordered  set of elements $\tA_n : = \{\aa_1, \dots, \aa_n \}$,  subject to the equivalence relation $\cclcong$, defined by \eqref{eq:letter.mir} as
$$ u \cclcong v \Iff \lcmp{n}{v} \clcong \lcmp{n}{u}.$$
Namely $\coclkM_n := \tA_n^*/ _{\cclcong}$.

\end{definition}
\noindent  We write $\CCLK_n$  for $\CCLK(\tA_n)$ when the alphabet $\tA_n$ is arbitrary. As the relation $\cclcong$ is defined in terms of $\clcong$, which is a congruence (Lemma \ref{lem:clk.cong}), respected by  the \cmirr ing operation \eqref{eq:word.mirr}, then it is a congruence.
In general the equivalence $\clcong$ on $\tA_n^*$  does not implies the equivalence $\cclcong$,
i.e.,  $u \clcong v$ does not imply $u \cclcong v$ on $\tA_n^*$, or vise versa. This is shown below in Example \ref{exmp:inject.3}, after developing additional methods.

\begin{properties}\label{proper:clk} For any two words $u \cclcong v$ in $\coclkM_n$ we have the following properties:
\begin{enumerate} \eroman
  \item  Every letter $\lt_i\in \varX_n$ appears in $u$ and in~$v$ exactly the same times, i.e., the formal relation $u \cclcong v$ is  balanced (cf. \S\ref{sec:3.2}).
  \item The length of $u$ and $v$ is the same.
\end{enumerate}

\end{properties}

As the \ckmon\ $\coclkM_n$ arises from the equivalence of the \kmon\ $\clkM_n$, we manipulate the  linear representation of $\clkM_n$, cf. \S\ref{ssec:clk.rep}, to construct a representation
of the (finitely generated) \ckmon\ $\coclkM_n$. 
That is, $\coclkM_n = \genr{\lt_1, \dots, \lt_n}$ of rank $n$ is represented by the monoid homomorphism (determined by generators' mapping)
\begin{equation}\label{eq:d.clk.rep} \dclkrep_\pv: \coclkM_n \TO \mfAml_n, \qquad
\lt_\ell \mTo \lcmp{n}{A_{\ell}}, \ e \mTo E,
\end{equation}
where $\mfA_n := \genr{  \Oix{A}{1}, \dots, \Oix{A}{n} }$ is the troplactic matrix algebra defined in \eqref{eq:mAlg}. Since $\mfA_n$ is a troplacitc algebra
  (Theorem \ref{thm:plac.mat.relations}), then Corollary \ref{cor:facotrize.plc.2} implies
 the  explicitly additive decomposition
 \begin{equation}\label{eq:co.mat.elm}
                                                                                                \Oix{M}{\ell} := \lcmp{n}{A_{\ell}} =
\prod^1_{\arrs{t =n}{t \neq n - \ell + 1}} \hskip -5mm  \Oix{A}{t}
=\bigvee^n_{\arrs{t =1}{t \neq n - \ell + 1}} \hskip -5mm  \Oix{A}{t} \; . \end{equation}
Then  $\Oix{M}{\ell}$ can be written as  $$ \Oix{M}{\ell} = \pv \Oix{\chA}{\ell},$$
where $\Oix{\chA}{\ell}$ is the $\nxn$ matrix of the form
 \begin{equation}\label{eq:co.cl.A.mat}
 \Oix{\chA}{\ell}  = \mresize{
     \left(\begin{array}{llllllll}
    \one  & \cdots & \one  & \one &  & \cdots & \one \\
       & \ddots & \vdots  & \vdots &  & & \vdots \\
     &   &  \one & \one& &   &  \\
     &  &  & \ipv &  \one &  \cdots  & \one \\
     &  & & & \one & \cdots &   \one\\
    \vdots &  & & &   & \ddots &   \vdots \\
    \zero & \cdots &  & & & \cdots &   \one
    \end{array}\right)}, \qquad \ell =1 , \dots, n,
\end{equation}
 whose  $(\ell',\ell')$-entry, $\ell' = n - \ell + 1$,  is $\ipv$ and its other nonzero entries are all  $\one$.

 Letting $\sMon{(\mfAml_n)}{1} := \genr{ \Oix{M}{\ell}, \dots, \Oix{M}{\ell} }$ be the matrix submonoid of $\mfAml_n$ generated by the matrices $\Oix{M}{\ell}, \dots, \Oix{M}{\ell}$ in \eqref{eq:co.mat.elm} and with identity $E$, from \eqref{eq:d.clk.rep} we  draw  the surjective homomorphism
\begin{equation}\label{eq:d.clk.rep.2} \smon{\dclkrep_\pv}{1}: \coclkM_n \TO \sMon{(\mfAml_n)}{1}, \qquad
\lt_\ell \mTo \Oix{M}{\ell}, \ e \mTo E,
\end{equation}
by restricting the image of $\dclkrep_\pv$ to $\sMon{(\mfAml_n)}{1}$.

\begin{remark}\label{rem:embed.coc}
Using  the endomorphism $ \pMap_1: \tA^*_n \onto{1} \sMon{(\tA^*_n)}{1}$ in Remark \ref{rmk:mir.smon},
 the \ckmon\ $\coclkM_n:= \tA_n^*/_{\cclcong}$ can be  defined equivalently as $u \cclcong v \Leftrightarrow \pMap_1(v) \clcong \pMap_1(v) $, and thus $ \sMon{(\tA^*_n)}{1}/_{\clcong} $ can be realized as a monoid isomorphic to a  submonoid of $\coclkM_n$. (Note that $ \sMon{(\tA^*_n)}{1}/_{\clcong} $ and
 $\sMon{\clkM_n}{1}$ are not necessarily equal as monoids.)
On the other hand, $\clkrep: \clkM_n \Isoto \mfAml_n,$ $\lt_\ell \mto \Oix{A}{\ell}$, is an isomorphism by Theorem \ref{thm:clk.represntation}, and its composition with the surjection $\pMap_1$ shows that  $\dclkrep_\pv: \coclkM_n \To \mfAml_n$, cf. \eqref{eq:d.clk.rep}, is an injective homomorphism. Hence, the map  $\smon{\dclkrep_\pv}{1}: \coclkM_n \To \sMon{(\mfAml_n)}{1}$  is an isomorphism. The  following diagram summarizes these homomorphisms
\begin{equation}\label{eq:diag.0} \begin{gathered}
 \xymatrix{ \tA_n^* \ar@{->>}[r]  \ar@{->>}[rd]&
 \clkM_n := \tA^*_n/_{\clcong}   \ar@{_{(}->>}[rr]^{\clkrep } & &  \MPlc_n  \ar@{->>}[d]^{\pMap_1} \ar@/^3pc/@{->}[dd]^{\matcmat} &     &
 \\ & \coclkM_n:= \tA_n^*/_{\cclcong} \ar@{_{(}->>}[rrd]^{\dclkrep}
 \ar@{_{(}->>}[rr]^{\smon{\dclkrep_\pv}{1}}
 \ar@{->}[rru]^{\dclkrep_\pv}& & \sMon{(\MPlc_n)}{1} \ar@{_{(}->>}[d]^\vartheta
 \\  & &  & \mir{\MPlc_n}
}\end{gathered}
 \end{equation}
Yet, we need to complete the lower part of the diagram -- the object $\mir{\MPlc_n}$.

Note also that by Remark \ref{rmk:clk.factor} and Diagram \eqref{eq:diag.0}, the monoid  homomorphism $\tA_n^* \To \coclkM_n$ factors through the plactic monoid
$$ \tA_n^* \ONTO \tA_n^*/_ {\pcong} \ONTO \tA_n^*/ _{\cclcong}.$$
We will return to this factorization later.
\end{remark}

Despite \eqref{eq:d.clk.rep} provides a faithful representation of $\coclkM_n$ in terms of monoids, i.e., as a matrix submonoid of $\mfAml_n$, the images $\dclkrep(\lt_\ell)$ of the generators $\lt_\ell$ of $\coclkM_n$ do not generate the image $\dclkrep(\coclkM_n)$ as a tropolactic matrix algebra. To achieve this attribute we define the matrix algebra
\begin{equation}\label{eq:mat.co.alg}
   \mir{\mfA}_n := \genr{\Oix{\chA}{1}, \dots, \Oix{\chA}{n}}  \subset \MatnT \; ,
\end{equation}
 generated
by the matrices $ \Oix{\chA}{1}, \dots, \Oix{\chA}{n} $ introduced   in \eqref{eq:co.cl.A.mat},
which we call the \textbf{co-algebra} of $\mfA_n$.  The multiplicative submonoid
$\dMPlcn$ of $\mir{\mfA}_n $ is called the \textbf{co-monoid} of $\MPlc_n$, for which
 we obtain the monoid homomorphism
\begin{equation}\label{eq:co.cl.rep}
  \dclkrep :  \coclkM_n \TO \dMPlcn, \qquad  \lt_{\ell} \mTo   \Oix{\chA}{\ell }, \ e \mTo E,
\end{equation}
i.e., $\dclkrep(\lt_\ell) = \ipv\dclkrep_\pv(\lt_\ell)$ for every  $\lt_\ell \in \tA_n$. Accordingly we see that
\begin{equation}\label{eq:iso.1}
 \vartheta: \sMon{ (\MPlc_n)}{1} \ISOTO \dMPlcn, \qquad \Oix{M}{\ell} \mTo \Oix{\chA}{\ell }\ (= \ipv\Oix{M}{\ell})  ,
\end{equation}
is a monoid isomorphism, and the composition
\begin{equation}\label{eq:mat.2.dmat}
 \matcmat := \vartheta \circ \pMap_1 :  \MPlc_n \ONTO \dMPlcn, \qquad \Oix{A}{\ell} \mTo \Oix{\chA}{\ell } ,
\end{equation}
is a surjective homomorphism.

Applying the negation map \eqref{eq:math.neq} in  Remark \ref{rmk:dual.mat},
$$\mfN: \Oix{\chA}{\ell} \longmapsto \Oix{\htA}{\ell} := -\Oix{\chA}{\ell}, \qquad \cha_{i,i} \mapsto -(\cha_{i,i}) =  \inv{(\cha_{i,i})},$$
 due to the special structure of $\Oix{\chA}{\ell}$, it  only
 replaces the entry $\cha_{\ell',\ell'} = \ipv$   by $\hta_{\ell',\ell'}  := \pv$, where $\ell' = n - \ell +1$).  The images  matrices $\Oix{\htA}{\ell}$ are then considered as matrices in $ \dMatnT$, whose multiplication are now induced by the semiring operations $ + $ and $\-$.
 Nevertheless, as in the sequel we want be consistent with the multiplicative operation of the monoid $\mfAml_n$, we stick with the matrices $\Oix{\chA}{\ell}$ for which the matrix multiplications are taken with respect to semiring operations $+$ and $\+$ (jointly with negation),
  while the additive operation $A \- B$ can be computed dually as $-(-A \+ -B)$.

\begin{example}\label{exmp:coclk.3}   The \ckmon\ $\coclkM_3$ of rank 3 is linearly represented by the homomorphism  $\dclkrep: \coclkM_3 \To \mir{ \mfAml_3}$ determined
by  the generators'  mapping
\begin{align*}
\lt_1 \mTo
 \Oix{\chA}{1} = \mresize{
\begin{bmatrix}
 0 & 0 & 0 \\
  & 0  & 0  \\
   &   & \ipv
\end{bmatrix}},   \quad
\lt_2 \mTo   \Oix{\chA}{2} = \mresize{
\begin{bmatrix}
 0 & 0 & 0 \\
  & \ipv & 0 \\
   &   & 0
\end{bmatrix}}  ,
 \quad
\lt_3 \mTo  \Oix{\chA}{3}  = \mresize{
\begin{bmatrix}
 \ipv & 0 & 0 \\
  & 0 & 0  \\
   &   & 0
\end{bmatrix}  }.
\end{align*}
(Recall that $\one := 0$, and the empty space stand for $\zero := -\infty$.)
  For these generators we have the products
$$ \Oix{\chA}{1} \Oix{\chA}{2}  = \mresize{
\begin{bmatrix}
 0 & 0 & 0 \\
  & \ipv & 0  \\
   &   & \ipv
\end{bmatrix}  }, \quad \Oix{\chA}{2} \Oix{\chA}{1}  = \mresize{
\begin{bmatrix}
 0 & 0 & 0 \\
  & \ipv & \ipv  \\
   &   & \ipv
\end{bmatrix}  }, \quad \Oix{\chA}{3} \Oix{\chA}{1}  = \mresize{
\begin{bmatrix}
 0 & 0 & 0 \\
  & \ipv & 0  \\
   &   & \ipv
\end{bmatrix}  } .  $$
Note that $\Oix{\chA}{1} \Oix{\chA}{2} = \Oix{\chA}{1} \- \Oix{\chA}{2} $,  $\Oix{\chA}{3} \Oix{\chA}{1} = \Oix{\chA}{3} \- \Oix{\chA}{1} $, but  $\Oix{\chA}{2} \Oix{\chA}{1} \neq \Oix{\chA}{2} \- \Oix{\chA}{1} $
\end{example}
Considering the digraph view of the matrix  $W = \dclkrep(w)$ assigned to
a word $w \in \tA_n^*$, cf. \S\ref{ssec:digraph},   a nonzero entry $w_{i,j}$ of $W$ corresponds to a path $[\pth_\setW]_{i,j}$  of maximal (nonpositive) weight from $i$ to $j$ of length $\len{w}$ in $\grph_\setW$ (Notation \ref{nott:1}). In terms of words,  it means the length of the complement of subword of $w$ over $\{ \lt_i, \dots, \lt_j\}$ of maximal length where each letter has (negative)  length $\ipv$.

\begin{theorem}\label{thm:co.plac.mat.relations} The matrix co-algebra  $\mir{\mfA}_n := \genr{ \Oix{\chA}{1}, \dots, \Oix{\chA}{n} }$
is a dual troplactic algebra $\dAPlc_n$ (Definition \ref{defn:d.trop.plc.alg}), and thus
every triplet of matrices $$\chA= \Oix{\chA}{p}, \quad  \chB = \Oix{\chA}{q}, \quad  \chC=\Oix{\chA}{r}, \qquad p < q < r,$$  admit the Knuth relations \eqref{eq:knuth.rel}.
%
\end{theorem}
\begin{proof}
We verify the axioms of Definition \ref{defn:d.trop.plc.alg} by a straightforward technical computation.
 By construction $$p' = n- p + 1, \qquad  q' = n- q + 1, \qquad  r' = n- r + 1,$$ are the only $\ipv$ entries in $\chA$, $\chB$, and $\chC$, with $p' > q' > r'$. All other nonzero entries are $\one$.
\begin{enumerate}\ealph
  \item The product $\chA \chB$ has exactly two $\ipv$ entries, the $(q',q')$-entry and the $(p',p')$-entry,  where all  other nonzero entries are $\one$.
  \item  The product $\chB \chA$ has at least two $\ipv$ entries, the $(q',q')$-entry and the $(p',p')$-entry,  all other nonzero entries are $\one$.  In general, these are the only $\ipv$ entries, except in the case that $p' = q' + 1$, in which also the $(q',p')$-entry has value  $\ipv$. (See e.g. Example \ref{exmp:coclk.3}.)
\end{enumerate}
The other two letter products involving $\chC$ have the analogues properties.
\begin{description} \eroman \dispace
 \item[\dPXa] ($\mfa  = \mfe  \d+ \mfa$.) \quad Immediate, by  comparison of entries.

  \item[\dPXb] ($ \mfa \mfb = \mfa \d+ \mfb$  when  $\mfb >  \mfa $.) \quad
  A direct implication of property (a), since $p' > q'$, which is also the structure of the matrix $\chA \- \chB$.

 \item[\dPXc] ($ (\mfb  \d+  \mfc) \mfa = \mfb \mfa  \d+ \mfc$.) \quad By property (b) the $\ipv$ entries of $\chB \chA \- \chC \chA$ are $(r',r')$, $(q',q')$, and $(p',p')$, and possibly $(r',p')$ or $(q',p')$.      If the $(q',p')$-entry in $\chB \chA$ is $\ipv$ then it is also $\ipv$ in $\chB \chA \- \chC \chA$, independently on $\chC \chA$.
      If the $(r',p')$-entry  is $\ipv$, then  the  $(q',p')$-entry  is  also $\ipv$ and hence  $\chB = \chC$. Thus, in all cases,
$(\chB  \- \chC) \chA = \chB \chA \- \chC $.

  \item[\dPXd] ($\mfc(\mfb  \d+  \mfa)  = \mfa \d+ \mfc \mfb $.) \quad
  By property (b) the $\ipv$ entries of $\chC \chB \- \chC \chA$ are $(r',r')$, $(q',q')$, and $(p',p')$, and possibly $(r',p')$ or $(r',q')$.
   If the $(r',q')$-entry in $\chC \chB$ is $\ipv$ then it is also $\ipv$ in $\chC \chB \- \chC \chA$, independently on   $\chC \chA$.
  If the $(r',p')$-entry  is $\ipv$, then the $(r',q')$-entry  is also $\ipv$,
  which implies  $\chA = \chB$. Thus, in all cases,
$\chC(\chB  \- \chA)  = \chA \-  \chC \chB $.
\end{description}
The proof is then completed by Theorem \ref{thm:d.plc.Alg.1}.
\end{proof}

\begin{theorem}\label{thm:co.clk.represntation}
The map \begin{equation}\label{eq:co.clk.rep}
\dclkrep :  \coclkM_n \TO \dMPlcn, \qquad \dclkrep: \lt_{\ell} \mTo  \Oix{\chA}{\ell }, \ e \mTo E, \end{equation}
defined by generators' mapping,  i.e.,
$$ \dclkrep(w) = \dclkrep(\lt_{\ell_1}) \cdots \dclkrep(\lt_{\ell_m}), \qquad w = \lt_{\ell_1} \cdots \lt_{\ell_m} \in \tA_n^*,$$
 is a monoid isomorphism -- a  faithful linear representation of the \ckmon\ $\coclkM_n$ of rank $n$.

\end{theorem}

\begin{proof}
The monoid homomorphism $\dclkrep_\pv: \coclkM_n \To \mfAml_n$ is injective by Remark \ref{rem:embed.coc},  providing the isomorphism
$\smon{\dclkrep_\pv}{1}: \coclkM_n \Isoto \sMon{(\mfAml_n)}{1}$,
which translates to
 $\dclkrep :  \coclkM_n \To \dMPlcn$ by the isomorphism \eqref{eq:iso.1}.
\end{proof}

By this theorem we see that
$u \cclcong v$ in $\coclkM_n$ implies that  $\dclkrep(u) \kcong \dclkrep(v)$ in $\mir{\mfAml_n}$, but the converse does not hold as seen later in
Example \ref{exmp:tfrm.2}.

\begin{corollary}\label{cor:co.clk.knuth} The \ckmon\ $\coclkM_n$ satisfies the Knuth relations  \eqref{eq:knuth.rel}.
\end{corollary}
\begin{proof}
$\dclkrep :  \coclkM_n \To \dMPlcn$ is an isomorphism, by Theorem \ref{thm:co.clk.represntation},  where $\dMPlcn$ by itself  admits  the Knuth relations by Theorem \ref{thm:co.plac.mat.relations}.
  \end{proof}

\begin{corollary}\label{cor:co.clk.id} The \ckmon\ $\coclkM_n$ admits all the semigroup identities satisfied by  $\TMatnT$, in particular the semigroup identities  \eqref{eq:iduniv2.2}.
\end{corollary}
\begin{proof}
  $\coclkM_n$ is faithfully represented by $\dMPlcn$ -- a submonoid of $\TMatnT$ -- which by Theorem \ref{thm:trMat.Id} satisfies the identity \eqref{eq:iduniv2.2}.
\end{proof}

\section{Configuration tableaux }\label{sec:configuration.tab}

Towards the  goal of introducing linear  representations of the plactic monoid $\plcM_I$ (Definition \ref{def:plactic.mon}), we need to establish its  precise  linkage to the \kmon\ $\clkM_n$. To do so we utilize the known bijective correspondence of $\plcM_I$ to Young tableaux \cite{Sagan}, therefore by this study we also obtain a useful algebraic description  of the latter combinatorial objects.
To simplify indexing notations we denote arbitrary letters in $\tA_I$ by $x,y,$ and $z$.

\subsection{Young tableaux}\label{ssec:y.tab} \sSkip

A \textbf{Young diagram} (also called Ferrers diagram) is a finite collection of cells, arranged in left-justified rows, such that the number of cells at each row (i.e., the row's length) is less or equal than its predecessor. In this paper we use the France convention in which the longest row is the bottom row, and  rows are enumerated from bottom to top  \cite{Las2}.
Over this setting, we define three types of (proper) steps  from a cell to its
 neighbor:
$$ \begin{array}{c}
     \ \downarrow \hskip -4.5mm \begin{tabular}{|c|c|}
\cline{1-1}
 $\bullet$    \\[1mm] \cline{1-2}
$\bullet$ &  \phantom{w}   \\ \cline{1-2}
\end{tabular} \\[5mm] \text{(a)}
\end{array}
\qquad \qquad
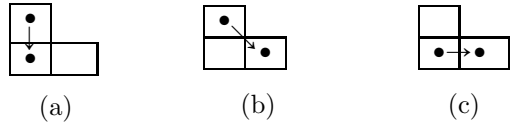
\begin{array}{c}
\searrow \hskip -8mm \begin{tabular}{|c|c|}
\cline{1-1}
  $\bullet$  \\ \cline{1-2}
 &  $\bullet$  \\ \cline{1-2}
\end{tabular} \\[5mm] \text{(b)}
\end{array} \qquad  \qquad
\begin{array}{c}
\begin{tabular}{|c|c|}
\cline{1-1}
    \\ \cline{1-2}
$\bullet \hskip -1mm \rightarrow \hskip -3.5mm$ &  $\bullet$ \  \\ \cline{1-2}
\end{tabular}\\[5mm] \text{(c)}
\end{array} $$

 \begin{enumerate} \ealph
   \item a \textbf{\vstep} -- a move from a cell to its
 neighbor below;

 \item a \textbf{\dstep} -- a move from a cell to its
 its bottom right neighbor;

 \item a \textbf{\hstep} -- a move from a cell to its
right neighbor.

 \end{enumerate}
 These steps are combined to walks and covers in a Young diagram:
$$
\begin{array}{c}
\begin{tabular}{|c|c|c|}
\cline{1-1}
  $\bullet$   \\[2pt] \cline{1-2}
$\bullet$ &     \\ \cline{1-3}
& $\bullet$ &   \\ \cline{1-3}
\end{tabular} \begin{array}{c}
                \hskip -18mm \downarrow  \\
                \hskip -13mm \searrow
              \end{array}  \\[5mm]  \text{(A)}
\end{array}
\qquad \qquad
\begin{array}{c}
\begin{tabular}{|c|c|c|}
\cline{1-1}
    \\ \cline{1-2}
$\bullet \hskip -1mm \rightarrow \hskip -3.5mm$ &  $\bullet$   \\[1pt] \cline{1-3}
&  &  $\bullet$  \\ \cline{1-3}
\end{tabular} \begin{array}{c}
                 \\
                \hskip -13mm \searrow
              \end{array}  \\[5mm]  \text{(B)}
\end{array}
\qquad \qquad
\begin{array}{c}
\begin{tabular}{|c|c|c|}
\cline{1-1}
$\bullet$    \\ \cline{1-2}
 &    \\ \cline{1-3}
& $\bullet \hskip -1mm \rightarrow \hskip -3.5mm$ &  $\bullet$  \\ \cline{1-3}
\end{tabular}\\[5mm]  \text{(C)}
\end{array}  $$
\begin{enumerate}\dispace
  \item[(A)] a \textbf{\vwalk} is a  continues composition of sequential \vstep s or  \dstep s;
  \item[(B)] a \textbf{\dwalk} is a  continues composition of sequential \hstep s or \dstep s;
  \item[(C)] a \textbf{\hwalk} is a collection of cells, one from each column, taken from a nonincreasing subset of rows.
\end{enumerate}
(Note that all of these are ``non-left oriented''.)
A \textbf{proper walk} is either a \vwalk\ or a \dwalk.
This terminology is consistent  with the content of the Young diagram, as described next.
\emph{In what follows, all walks and covers are assumed to be proper.}
\pSkip

A \textbf{Young tableau} is a Young diagram whose cells are filled by symbols -- a single letter in each cell -- taken from a given   alphabet ~$\tA_I$, usually required to be finite and totally ordered.

\begin{definition}\label{def:wght.walk}
The \textbf{track} $\trk{\wlk}$  of a (proper) walk $\wlk$ in a Young tableau $\tb$ is the word defined by concatenating the symbols appearing in  cells  along $\wlk$.
 The \textbf{length} $\len{\wlk}$ of a walk $\wlk$ is the number of cells along  $\wlk$.

 When the content of cells in $\tb$ are numbers, the \textbf{weight} of $\wlk$, denoted $\tbwt{\wlk}$, is defined as the sum  of cells' values along  $\wlk$.
\end{definition}

While a walk between two cells needs  not be  unique, except when the walk is along the bottom row or the left column, all the tracks  between pair of cells always have the same length; different tracks may have a same weight.

Young tableaux are combinatorial structure, used intensively in group theory and representation theory~ \cite{Fulton}. They provide a useful machinery that led to  an elegant algorithmic solution (Schensted 1961) for the following problem:
\begin{problem*}
Given a (finite) word $w \in \tA_I^+$ on the (totally) ordered
alphabet $\tA_I$, find the length of its longest nondecreasing subwords of $w$.
\end{problem*} \noindent The central advantage of Schensted's method is that it avoids a precise identification of a longest  nondecreasing subword, and presents  a useful linkage between semigroups and Young tableaux, as summarized below.

A nondecreasing word $w \in \tA_I^+$ is called a \textbf{row}, while a (strictly) decreasing word is called a \textbf{column}.
We say the row  $u = x_1 \cdots x_s$  \textbf{dominates} the  row $v = y_1 \cdots y_t$, written $u \dom v $,  if $s \leq t$ and $x_i > y_i$ for all  $i =1, \dots, s$.
Any  word $w \in  \tA_I^+ $ has a unique factorization $ w =  \rw_p \cdots \rw_1$ as a concatenation  of rows $\rw_i$ of maximal length.

\begin{definition}
A \textbf{semi-standard tableau} $\tb$ is a tableau  that corresponds (by rows) to a word $w \in \tA_I^+$
such that $\rw_p \dom \rw_{p-1} \dom \cdots \dom \rw_1$, where $\rw_1$ is the bottom row and $\rw_p$  is the top row of $\tb$.
The set of all semi-standard tableaux is denoted by $\Tab(\tA_I)$, written $\Tab_I$ for short.

We formally adjoin the \textbf{empty tableau}, denoted $\tb_0$, considered as a semi-standard tableau.

\end{definition}
\noindent
Namely,  a tableau is  semi-standard if the entries in each row are nondecreasing and the entries in  each column are increasing. In the sequel we assume that $I$ is  finite and write $\Tab_n$ for $\Tab_I$.
 A tableau is called \textbf{standard} if the entries in each row and in each column are increasing, which implies that each letter may appear only once. The subset of all standard tableaux over a finite alphabet $\tA_n$ is denoted by $\STab_n.$

We denote the columns of a given  tableau $\tb$ by $\cl_1, \dots, \cl_q$, enumerated from left to right. Reading the columns from bottom to top, each column  $\cl_{j+1}$ is increasing but not necessarily dominates the column $\cl_{j}$ unless, $\tb$ is standard. In latter   case, we define the \textbf{transpose tableau} $\tb^\trn$ of $\tb$ by writing the columns as rows, i.e., $ \cl_q \cdots \cl_1$, for which  $(\tb^\trn)^\trn = \tb$.

The nonincreasing sequence
$\lm := (\lm_1 \geq  \lm_2 \geq \cdots \geq \lm_k)$ of the rows' length $\lm_i$ of a tableau $\tb$ defines the \textbf{shape} of   $\tb$ and  introduces a \textbf{partition} of the integer $|\lm| = \lm_1 + \lm_2 +  \cdots + \lm_k$, which in its turn determines the graphical representation of
the corresponding  Young diagram. (The shape of the empty tableau is set to be $0$.)

Schensted's algorithm associates each word $w \in \tA_n^+$ with a tableau $\tb = \tab(w)$, via the map \footnote{In the literature, the map $\tab$ is also  denoted  by ~$P$.}
$$ \tab: \tA_n^+ \TO \Tab_n,$$ defined inductively as $\tab(wx) = \tab(\tab(w)x)$ for an arbitrary $x \in \tA_n$. It is recursively  defined in terms of tableaux as
\begin{equation}\label{eq:recursive.tab}
 \tab(\tb x ) = \left\{
                 \begin{array}{ll}
                   \tb x   & \hbox{if $\rw_1 x$ is a row}, \\[1mm]
                    \tab(\rw_p \cdots \rw_{2} y) \rw'_1 & \hbox{if $\tab(\rw_1 x) = y \rw'_1$.}
                 \end{array}
               \right.
\end{equation}
where  $\rw_p \cdots \rw_1$ are the row decomposition of $\tb$, $y$ is the left most letter of $\rw'_1$ which is strictly grater than ~$x$, and $\rw'_1$ is obtained from $\rw_1$ by replacing $y$ with $x$.

In a more lingual  terms,  the map $\tab$ can be described in algorithmic terms.
\begin{algr}[Bumping Algorithm]\label{algr:1} Given a word $w = x_1 \cdots x_m$, $x_i \in \tA_n$,  start with $\{x_1\}$, which is a Young tableau. Suppose $x_1,\dots, x_k$ have already been inserted, and $\tb$ is the current tableau.  To insert $x_{k+1}$, start with the first row of $\tb$ and search for the first letter which is greater than $x_{k+1}$. If there is no such element, append $x_{k+1}$ to the end of first row. If there is such an element (say, $x_j$), exchange $x_j$ by $x_{k+1}$,
and proceed inductively to insert $x_j$ to the next row.
\end{algr}

For three letters $x< y< z$ in $\tA_n^+$, the algorithm produces the familiar tableaux:
$$ \tab(xzy)  = \tab(zxy)  =
\begin{tabular}{|c|c|c}
\cline{1-1}
 z  \\ \cline{1-2}
x &  y  & , \\ \cline{1-2}
\end{tabular}
\qquad  \tab(yxz) = \tab (yzx) =
\begin{tabular}{|c|c|c}
\cline{1-1}
 y  \\ \cline{1-2}
x &  z  & . \\ \cline{1-2}
\end{tabular}
$$
For two distinct letters $x < y$ it gives:
$$
 \tab(xyx) = \tab (yxx)  =
\begin{tabular}{|c|c|c}
\cline{1-1}
 y   \\ \cline{1-2}
x &  x & , \\ \cline{1-2}
\end{tabular}\qquad
\tab (yxy) = \tab(yyx)  =
\begin{tabular}{|c|c|c}
\cline{1-1}
 y   \\ \cline{1-2}
x &  y  & . \\ \cline{1-2}
\end{tabular}
$$
One immediately sees that the upper tableaux are just realization of the Knuth relations \eqref{eq:knuth.rel}.
\begin{theorem}[Schensted 1961]\label{thm:Schen} The length of a longest nondecreasing (reps. decreasing)
subword of~$w$ is equal to the length of the bottom row $\rw_1$ (resp. height of the first column $\cl_1$) of $\tab(w)$.
\end{theorem}

This tableau form   introduces on $\tA_n^+$  the  equivalence relation
$$ u \tcong v \Iff \tab(u) = \tab(v).$$
We call $\tab(w)$ the \textbf{(canonical) tableau form} of the word $w \in \tA_n^+$, and   write $$\tb_w: = \tab(w).$$ Conversely, every tableau $\tb \in \Tab_n$ can be  identified  with $\tb_w = \tab(w)$ for some (not necessarily unique) word $w \in \tA^+_n$. Therefore, we may think of $\tb_w$ also as a word having a canonical form in  $\tA_n^+$.
The extension of $ \tab: \tA_n^+ \To \Tab_n$  to
$$ \tab: \tA_n^* \TO \Tab_n$$
is defined naturally by sending $e \mTo \tb_0$.

\begin{theorem}[Knuth 1970]\label{thm:Knuth} The equivalence relation $\tcong$ coincides with the plactic
congruence $\pcong$ which is the Knuth congruence $\kcong$.  In particular,  each plactic class contains exactly one tableau.
\end{theorem}
In our notation, the theorem reads as
$$ \tb_w = [w]_\plc \ (= [w]_\knu), \qquad \text{for every } w \in \tA_n,$$
and hence  we may alternate between these notations.
Therefore   $\Tab_n$ can be realized as an algebraic structure, equipped with a multiplicative operation.
\begin{remark}\label{rmk:tab.mon} Set theoretically,    $\Tab_n$ bijectively corresponds  to the plactic monoid $\plcM_n$, thus it can be considered as a monoid whose operation is now induced  from  Algorithm \ref{algr:1}, or equivalently from ~\eqref{eq:recursive.tab}. That is,
the product of tableaux $\tb_u \cdot \tb_v$, with say $\tb_v = y_1 \cdots y_m$,   is defined by the sequential insertion (by reverse row ordering) of the  letters $y_1, \dots, y_m$  of $\tb_v$ to $\tb_u$.
\end{remark}

In order to extract numerical invariants  from the combinatorial structure of semi-standard tableaux we define the following functions.
\begin{enumerate} \ealph
  \item  The function
\begin{equation}\label{eq:cont.func.letters}
    \ltno{i}{\lt_\ell}: \Tab_n \TO \N_0, \qquad \ell, i = 1,\dots, n,  \end{equation}
 counts the occurrences of the letter $\lt_\ell \in \tA_n^+$ in the $i $'th row  of a tableau $\tb \in \Tab_n$.

  \item  The function
\begin{equation}\label{eq:cont.func.swords}
    \sword{J}{}: \Tab_n \TO \N_0, \qquad  \ J \subset \{ 1,\dots, n \}, \ J \neq \emptyset, \end{equation}
assigns  a tableau with length of its longest nondecreasing subword that involves only  letters
$\al_\ell \in \tB_J$ from the sub-alphabet  $\tB_J \sqsubseteq \tA_n$.
\end{enumerate}
Our main interested is in convex   sub-alphabet $\tA_J \cnxsset \tA_n$, i.e., $J = \{i, i+1, \dots,j \} $, with $i \leq j$, which we denote as $\itoj{\lt_i}{\lt_j}$ or $\itoj{i}{j}$, for short.

 When   $i = j = \ell$, the function  $\sword{\itoj{\ell}{\ell}}{\itoj{\ell}{\ell}}(\tb)$ gives the number of occurrences of the letter $\lt_\ell$ in the tableau ~$\tb$, and thus
$$ \sword{\itoj{\ell}{\ell}}{}(\tb) = \sum_{t=1}^n \ltno{t}{\al_\ell}(\tb), \qquad{\text{for every }} \ell  = 1, \dots, n.$$
We write $\tbrest{\tb}{i}{j}$ for the restriction of the tableau $\tb$ to  rows $i, \dots , j$.

\begin{remark}\label{rmk:max.susse} The equivalence relation $\clcong$ on tableaux  $\tb_u $ and  $\tb_v $ in  $ \Tab_n $ (cf. \eqref{eq:clk.cong}) reads in terms of~ \eqref{eq:cont.func.swords} as
$$\tb_u \clcong \tb_v {\Iff} \sword{J}{}(\tb_u) = \sword{J}{}(\tb_v) \quad \text{ for every  } J \cnxsset \nnn. $$
(Note that if    $\tb_u \clcong \tb_v$, then $\tb_u $ and  $\tb_v$ must contains exactly the same number of each letter.)
\end{remark} \noindent

We start with our running example (appeared also in Remark \ref{rem:reveras.clack}) that points out special pathologies along our exposition.
\begin{example}\label{exmp:run:1}  The words
$ u = c \ b^2 c^2 \ a^2 b^2 c $ and    $v= c^2 \ b^2 c \ a^2 b^2 c $  have the following tableau realizations:
$$u = c \ b^2 c^2 \ a^2 b^2 c  \dss {\mTo} \tb_u =
\begin{array}{|l|l|l|l|l|l}
\cline{1-1}
  \cc   \\ \cline{1-4}
  \bb & \bb &  \cc & \cc  \\ \cline{1-5}
  \aa & \aa  &  \bb & \bb &     \cc  & , \\ \cline{1-5}
\end{array}
  $$
and
$$v  = c^2 \ b^2 c \ a^2 b^2 c  \dss \mTo  \tb_v =
\begin{array}{|l|l|l|l|l|l}
\cline{1-2}
  \cc   & \cc \\ \cline{1-3}
  \bb & \bb &  \cc  \\ \cline{1-5}
  \aa & \aa  &  \bb & \bb &     \cc & .  \\ \cline{1-5}
\end{array}
 $$
It easy to check that $u \clcong v$, where $\sword{\itoj{b}{c}}{}(\tb_u) = \sword{\itoj{b}{c}}{}(\tb_v) = 5$,  but $u \not \tcong v$.
\end{example}
Hence, from this example,  we learn  that $\clcong$  does not imply $\tcong$ (or equivalently $\kcong$), but the converse holds.
\begin{proposition}\label{prop:tab-to-cloc} If $u \kcong v$, or equivalently $u \tcong v$,  then  $u \clcong v$, for any $u,v \in \tA^*_n$.
Thus $u \clcong v$ iff $\tb_u \clcong \tb_v$.

\end{proposition}\label{prop:plc2clk}
\begin{proof}
By Theorem \ref{thm:Knuth}, $u \kcong v$  iff $u \tcong v$ iff $\tb_u = \tb_v$, then it follows from  Remark \ref{rmk:max.susse} that $u \clcong v.$
\end{proof}

\begin{corollary}\label{cor:plc.to.clk}
The map
$$ \plktoclk: \plcM_n \ONTO \clkM_n, \qquad [w]_\plc \mTo [w]_\clk,$$
is a surjective  homomorphism.
\end{corollary}

\begin{proof} Both $\plcM_n$ and $\clkM_n$ admit the Knuth relations \eqref{eq:knuth.rel}, cf. Corollary \ref{cor:clk.knuth}, where $\kcong$ implies $\clcong$ by Proposition \ref{prop:plc2clk}. Thus $\plktoclk$ is a well defined monoid homomorphism. Surjectivity  is clear.
\end{proof}

Accordingly the diagram
 \begin{equation*}\label{eq:diag.2} \begin{gathered}
 \xymatrix{
\tA_n^* \ar@{->>}[rr] \ar@{->>}[rrd]& &  \plcM_n := \tA_n^* /_{\kcong}  \ar@{->>}[d]^\plktoclk \\ &  &
\qquad \clkM_n := \tA_n^* /_{\clcong}
}\end{gathered}
    \end{equation*}
commutes.

\subsection{Configuration tableaux}  \sSkip
We  introduce a new class of tableaux, carrying nonnegative integers.
These tableaux record lengths of tracks in semi-standard tableaux as weights, and are used as axillary tableaux, playing  a major role  in this paper.
The cells' indexing  in these tableaux are arranged in a way that each diagonal corresponds to a letter, as explained  below. As will be seen later, this indexing system  makes sense in the passage to matrices which establish linear  representations of Young tableaux.

\begin{definition}\label{def:ctab}
An $n$-\textbf{configuration tableau} $\ctb = (\lm_{i,j})$ is an isosceles tableau with $n$ rows and $n$ columns
of fixed shape $ (n >  n-1 > \cdots > 1)$ whose cells are indexed as

\begin{equation*}\label{eq:ctab}
 \ctb =  \quad
\begin{array}{|l|l|l|l|l|l|l}
\cline{1-1}
 {\cent_{n,n}}   \\ \cline{1-2}
 {\cent_{n-1,n-1} }  &   {\cent_{n-1,n} }    \\ \cline{1-3}
 \ { \vdots}  &   {}   &  {\ddots}  \\ \cline{1-4}
  {\cent_{3,3}  }  &   {  } & &   { \cent_{3,n} } \\ \cline{1-5}
  {\cent_{2,2}} &   { \cent_{2,3} }  &   {  } &  { \cent_{2,n-1} } &   { \cent_{2,n} } \\ \cline{1-6}
  {\cent_{1,1}} &   { \cent_{1,2} }  &   { \cent_{1,3} \ \ } & \  { \cdots } &   { \cent_{1,n-1} }  &   { \cent_{1,n} } & , \\ \cline{1-6}
\end{array}
\end{equation*}\pSkip
the rows are enumerated from bottom to top and columns from left to right.
The index $i$ of   $\lm_{i,j}$ stands for $i$'th row, while $j$ refers to the $j$'th diagonal (enumerated for bottom left).

The content of cells are  nonnegative integers $\lm_{i,j}$ satisfying  the following \textbf{configuration laws}  for all   $i \leq j \leq n$
\begin{align}\label{eq:ctab.order}
       \sum_{t=0}^k \cent_{j,j+t}  \leq  \sum_{t=0}^k \cent_{i,i+t} ,  & \qquad  \text{for every  }  k \leq n-j.
\end{align}
 The \textbf{null configuration tableau}, denoted  $\ctb_0$,  is the $n$-configuration tableau with $\lm_{i,j} =0$ for all $i,j$.
We denote the set of all $n$-configuration tableaux by $\CTab_n$.

The \textbf{$\ell$'th diagonal} of $\ctb$, for fixed $\ell$, is the the collection of cells $\cent_{t,\ell}$, $t = 1, \dots, \ell$.  We define
\begin{equation}\label{eq:Om.order}
    \Cent_\ell(\ctb) := \sum_{t = 1}^\ell \cent_{t,\ell}, \qquad \ell = 1, \dots, n,
\end{equation}
which we call the \textbf{$\ell$'th trace} of $\ctb$.

An  $n$-configuration tableau is called \textbf{standard} if  $\Cent_\ell = 1$ for every $\ell = 1, \dots, n.$ The set of all standard $n$-configuration tableaux is denoted by $\SCTab_n$.

\end{definition}
\noindent Literally,  the configuration laws \eqref{eq:ctab.order} assert  that the sum of cells' values of a
left adjacent subrow of $\rw_i$ is greater or equal than the sum of cells' values of any left adjacent subrow of the same length placed above $\rw_i$.

Note that \emph{configuration tableaux by themselves  need  not be semi-standard} as they may contain  $0$-valued cells. Nevertheless, as proved  later, they bijectively correspond to semi-standard tableaux.
The fact that all  $n$-configuration tableaux have the same fixed shape,  together with a canonical indexing system of cells,   permits uniform formal references, and  provides a suitable notation for  walks between cells.

Given an  $n$-configuration tableau $\ctb = (\cent_{i,j})$, we write $\wlk_{i,j}$ for a proper walk  from cell $\cent_{i,i}$ to cell $\cent_{1,j}$; $\Wlk_{i,j}$ denotes the set of all such proper walks $\wlk_{i,j}$. Due to the cells' indexing  of configuration tableaux, since $\wlk_{i,j}$ is proper,
\begin{itemize}\dispace
  \item
$\wlk_{i,j}$ is a \dwalk\ iff $i \leq j $,
  \item  $\wlk_{i,j}$ is a
\vwalk\ iff $i \geq j $.

\end{itemize}
 The \textbf{weight} $\tbwt{\wlk_{i,j}}$ of the walk $\wlk_{i,j}$ is the sum of its cells' values, i.e., $$ \tbwt{\wlk_{i,j}} := \sum_{\cent_{s,t} \in \wlk_{i,j}}\cent_{s,t}.$$
 We write $\mxtbwt{}_{i,j}(\ctb)$  for the maximal weight over all walks $\wlk_{i,j} \in \Wlk_{i,j}$, that is
 \begin{equation}\label{eq:max.inc.walk}
 \mxtbwt{}_{i,j}(\ctb) = \bigp_{\gm_{i,j} \in \Wlk_{i,j}} \tbwt{\wlk_{i,j}}, \qquad (i \leq j) \;,
 \end{equation}
 which determines the function
 \begin{equation}\label{eq:stp.0}
\mxtbwt{}_{i,j}:\CTab \TO \N_0, \qquad i,j = 1, \dots,n.
\end{equation}
We write $\mntbwt{}_{i,j}(\ctb)$  for the minimal weight over all walks $\wlk_{i,j} \in \Wlk_{i,j}$, that is
 \begin{equation}\label{eq:min.dnc.walk}
 \mntbwt{}_{i,j}(\ctb) = \bigm_{\gm_{i,j} \in \Wlk_{i,j}} \tbwt{\wlk_{i,j}}, \qquad (i\geq j) \; .
\end{equation}

 When $i=j $, the set  $\Wlk_{i,i}$ consists of exactly one proper walk $\wlk_{i,i}$  which is both a \dwalk \ and a \vwalk, and thus (cf. \eqref{eq:Om.order}) $$\mxtbwt{}_{i,i}(\ctb) = \mntbwt{}_{i,i}(\ctb)= \tbwt{\wlk_{i,i}} = \Cent_i$$  for any $i = 1,\dots, n$.
On the other hand, for fixed $i =1$, the weight
\begin{equation}\label{eq:ctab.row.1}
\tbwt{\wlk_{1,j}} = \mxtbwt{}_{1,j}(\ctb) =   \sum_{t =1}^j \lm_{1,j}
\end{equation}  is  determined by a unique (bottom) walk, which precisely encodes the first row $\rw_1$ of $\ctb$.
Hence, the weights of the walks $\wlk_{1,j}$  record the full data on $\rw_1$.

\pSkip

Let $\HC_{i,j}(\ctb)$ be the set of all \hwalk s $\cov_{i,j}$ of columns $1, \dots, j$ by rows $1,\dots, i$ with $i \geq j$, cf.~ \S\ref{ssec:y.tab}. Namely, $\cov_{i,j}$ is a collection of cells
$\cent_{i_1,1}, \dots, \cent_{i_j,j} $ in $\ctb := (\cent_{i,j})$ such that  $i_1 \geq i_2 \geq \cdots \geq i_j$ with $i_t \in \{ 1,\dots, j \} $;  in particular
$\cov_{i,i} = \lm_{i,i}$ for each $i = 1,\dots, n.$

 The \textbf{weight} $\tbwt{\cov_{i,j}}$ of a cover $\cov_{i,j}$ is the sum of its cells' values, i.e., $$ \tbwt{\cov_{i,j}} := \sum_{\cent_{s,t} \in \cov_{i,j}}\cent_{s,t} \; .$$
Using these weights we  introduce a new function on $n$-configuration tableaux:
\begin{equation}\label{eq:stp.1}
\jmp{i,j}{}:\CTab \TO \N_0, \qquad i,j = 1, \dots,n,
\end{equation}
defined for $i \geq j$ as
\begin{equation}\label{eq:stp.2}
   \jmp{i,j}{}(\ctb) = \bigm_{\cov_{i,j} \in \HC_{ij}  } \tbwt{\cov_{i,j}},
\end{equation}
and $\jmp{i,j}{}(\ctb) := 0$ whenever $i < j$. Explicitly, in terms of subsequences,  the value of the function $\jmp{i,j}{}$ can be  computed directly as
\begin{equation*}\label{eq:stp.3}
   \jmp{i,j}{}(\ctb) = \bigm_{S_j \sqsubseteq \itoj{1}{i}  } \sum_{\arrs{t =1}{ s_t \in S_j}}^j \cent_{t,s_t}\, , \qquad i \geq  j \; .
\end{equation*}
 In particular, for $i = j$ we have
$ \jmp{i,i}{}(\ctb) = \Cent_i $, while  by the configuration law \eqref{eq:ctab.order} for every $i = 1, \dots, n$ we obtain that $$ \jmp{i,1}{}(\ctb) = \cent_{i,i} \, . $$ So, we see that $\jmp{i,1}{}(\ctb)$, $i = 1,\dots,n$, records explicitly  the full data on the  first column $
\cl_1$ of $\ctb$.

The fixed structure of configuration tableaux enables a straightforward execution of operations that resize the tableaux -- shrink or extend them.
\begin{remark}\label{rmk:ctab.del}
  The deletion of the bottom row or the $n$'th diagonal of an $n$-configuration tableau results in  a new proper $(n-1)$-configuration tableau. This does not hold for the deletion of  the left column,  as  seen in Example \ref{exmp:ctab.1} below.
\end{remark}
On the other hand, one can enlarge the size of a tableaux through tableau injections, fully  preserving its content.
 The possible injections are determined by column mappings, depending on the size of the target tableau, while row mappings are  always one-to-one.

\begin{remark}\label{rmk:ctab.injection}
  Among the possible injections of configuration tableaux of different sizes, we are interested in   the \textbf{right injection}, that is  the map
  $$ \RInj: \CTab_m \TO \CTab_n, \qquad m \leq n,$$
  that embeds  an $m$-configuration tableau $\ctb \in \CTab_m$ in the  $n$-configuration tableau whose $m \times m$ right part is identically  $\ctb$ and all its cells in the left columns $1, \dots , n-m$ are of value $0$.
\end{remark}

\subsection{Semi-standard tableaux vs. configuration tableaux}\label{ssec:tab.vs.ctab}\sSkip

Every semi-standard tableau  $\tb  \in \Tab_n $ is associated to an   $n$-configuration tableau $\ctb \in \CTab_n$  by the  map
\begin{equation}\label{map:tabctab}
    \tabctab: \Tab_n \TO \CTab_n,
\end{equation}
defined as
\begin{equation}\label{map:tabctab.2}
    \tabctab: \tb \longmapsto \ctb, \qquad \text{where } \cent_{i,\ell} := \ltno{i}{\lt_\ell}(\tb),
\end{equation}
i.e., the cell  $\cent_{i,\ell}$ of $\ctb = (\cent_{i,j})$ is assigned with the occurrence number of the letter $\lt_\ell$  in the row $\rw_i$ of~ $\tb$, and $\tb_0 \mTo \ctb_0$.
By this setting  the letter $\lt_\ell$ in  $\tb$ corresponds to the $\ell$'th diagonal of $\ctb$, enumerated starting from bottom left. Accordingly,  the  $\ell$-trace
$\Cent_\ell$ of $\ctb$, cf. \eqref{eq:Om.order}, gives the total occurrences  of the letter $\al_\ell$ in $\tb$, which  is read off from  the $\ell$'th diagonal of $\ctb$.

\begin{example}\label{exmp:ctab.1} Let $w = \lt_4 \, \lt_3 \lt_4 \,
  {\lt_2}    {\lt_2}     {\lt_4}  {\lt_4} \,
  {\lt_1}    {\lt_1}      {\lt_1}  {\lt_3}     {\lt_3}     {\lt_4} $ be a standard tableau for which the  map  $\tabctab: \tb_w \longmapsto \ctb$ is as follows:
$$ \tb_w =
\begin{array}{|l|l|l|l|l|l|l|}
\cline{1-1}
  {\lt_4}   \\ \cline{1-2}
  {\lt_3}  &   {\lt_4} \\ \cline{1-4}
  {\lt_2} &   {\lt_2}  &   {\lt_4} &  {\lt_4} \\ \cline{1-6}
  {\lt_1}  &    {\lt_1} &     {\lt_1}  & {\lt_3} &    {\lt_3} &    {\lt_4} \\ \cline{1-6}
  \multicolumn{6}{c}{ \begin{array}{l} \\ \\ \end{array}}
\end{array}
\qquad \To  \qquad
 \ctb = \begin{array}{|l|l|l|l|l|l|l|}
\cline{1-1}
  {1}   \\ \cline{1-2}
  {1} &  {1}  \\ \cline{1-3}
  {2} &   {0}  &   {2}  \\ \cline{1-4}
  {3}  &    {0} &     {2}  &  {1}  \\ \cline{1-4}
\multicolumn{5}{c}{\begin{array}{l}
                                      \nwarrow \ \nwarrow \  \nwarrow  \  \nwarrow \\
                                      \quad \ \lt_1 \; \lt_2 \; \lt_3 \; \lt_4
                                     \end{array}
}
\end{array}
$$
The traces of $\ctb$ are then   $\Cent_1(\ctb)=3$, $\Cent_2(\ctb)=2$, $\Cent_3(\ctb)=3$, and $\Cent_4(\ctb)=5$.

\end{example}
This example also  shows that tableaux are not necessarily semi-standard, e.g., see the second column that is not (strictly) increasing.

\begin{remark}\label{rmk:tabctab}
For every semi-standard tableau $\tb \in \Tab_n$ assigning  \eqref{map:tabctab.2} admits the configuration laws ~ \eqref{eq:ctab.order} as for every $i < j $ the $j$'th row $\rw_j $ of $\tb$ dominates its $i$'th row $\rw_i$.
\end{remark}

An important property of the map 
\eqref{map:tabctab}
 is that it also precisely records  the shape of semi-standard tableaux. That is,
  the shape $(\lm_1, \dots, \lm_m)$ of a tableau $\tb$ is encoded in its   image $ \tabctab(\tb)$ as the sums of cells' values of the rows:
   $$ \lm_i = \sum_{t=i}^n \cent_{i,t}.$$
   In general, an $n$-configuration tableau records all the information that its pre-image semi-standard tableau carries, in a numerical way.

\begin{theorem}\label{thm:configuration} The map $\tabctab: \Tab_n \To \CTab_n$ in \eqref{map:tabctab} is bijective.
\end{theorem}
\begin{proof} Let $\rw_i = \al_{\ell_1}^{q_1} \al_{\ell_2}^{q_2} \cdots \al_{\ell_m}^{q_m}$,
where $\ell_1, \dots, \ell_m \in L \subseteq  \{i,\dots, n\}$ and $q_t \in \N$, be the $i$'th row of the tableau $\tb \in \Tab(\tA_n)$.
Then~ $\rw_i$ can be rewritten uniquely as $\al_{i}^{q_i} \al_{i+1}^{q_{i+ 1}} \cdots \al_{n}^{q_n}$
with $q_t \in \N_0$  for every $\ell_t \in \{i,\dots, m\}$ and $q_t = 0$ when
$t \in \{i,\dots, n\} \sm L.$  Thus, the row mapping
\begin{equation}\label{eq:str}
\tabctab|_{\rw_i}: \al_{i}^{q_i} \; \al_{i+1}^{q_{i+ 1}} \; \cdots \; \al_{n}^{q_n} \dss \mTo  q_i \;  q_{i+1} \; \cdots \; q_n \tag{$*$}
\end{equation}
is injective, i.e., $\cent_{i,t} = q_t$, $t = i, \dots, n$.   As this holds for the restriction $\tabctab|_{\rw_i}$ of $\tabctab$ to any row $\rw_i$, and $\tabctab$ maps row-to-row, we deduce that  $\tabctab$ is  injective over  the whole tableau (and also respects the configuration laws \eqref{eq:ctab.order} by Remark \ref{rmk:tabctab}).

To see that $\tabctab$ is surjective, use \eqref{eq:str} to reproduce each row $\rw_i$ of $\tb$ from the $i$'th row of $\ctb$, which together, due to the configuration laws  \eqref{eq:ctab.order}, obey the dominance relations in  semi-standard tableaux.
\end{proof}

The next  conclusions are now  immediate.

\begin{corollary}\label{cor:sctab} The restriction \begin{equation}\label{map:stabsctab}
    \stabsctab: \STab_n \TO \SCTab_n
\end{equation}
 of $\tabctab: \Tab_n \To \CTab_n$  to standard tableaux is bijective.
\end{corollary}

In  view of  Theorem \ref{thm:configuration}, and \eqref{eq:recursive.tab}, we define the map $$\ctab: \tA_n^+ \TO \CTab_n, \qquad w \mTo \tabctab(\tab(w)), $$  and write $$\ctb_w:= \ctab(w)$$  for short. For the empty word $e \in \tA_n^*$, we formally set
 $\ctb_e := \ctb_0$ to be the null configuration tableau $\ctb_0$ (Definition \ref{def:ctab}).

\begin{corollary}\label{cor:plc.to.ctab} The map
$$ \scP_\ctab: \plcM_n := \tA_n^*/_{\pcong} \TO \CTab_n, \qquad [w]_\plc \longmapsto \ctb_w , $$
is bijective, and $u \pcong v$ iff $\ctb_u = \ctb_v$, for any $u, v ,w \in \tA_n^*$.
\end{corollary}
\begin{proof}
Compose Theorem  \ref{thm:configuration} with Theorem \ref{thm:Knuth}.
\end{proof}
Similar to Remark \ref{rmk:tab.mon}, by the latter  corollary, the map $\scP_\ctab: \plcM_n  \To \CTab_n$ can be realized as a monoid homomorphism, where $\CTab_n$ is a monoid whose operation (insertion by concatenating) is induced by the next algorithm.
This view is  compatible with the monoid structure of $\Tab_n$ via the bijection $\tabctab: \Tab_n \To \CTab_n$ -- reads now a tableau isomorphism.

\begin{algorithm}[Encoding Algorithm]\label{algr:conf}
  To encode a word $w \in \tA_n^+$ in an $n$-configuration tableau $\ctb$, start with an empty $n$-configuration tableau $\ctb := \ctb_e$ and perform the following letter by letter.

 To encode the letter $\lt_\ell \in \tA_n$ in $\ctb$,  start with the first row $i=1$
  \begin{itemize}\dispace
    \item incitement  $\cent_{i,\ell}$ by $1$, i.e., $\cent_{i,\ell} \leftarrowtail \cent_{i,\ell} + 1$;
    \item decrement $\cent_{i,k_i}$ by $1$, where $k_i >\ell$ is the minimal index with  $\cent_{i,k_i} > 0$.
  \end{itemize}
  Repeat  this same procedure to insert $\lt_{k_i}$ to row $i +1$ of $\ctb$, as long  as $i < n$.

This encoding  of a word  $w$ in  $\ctb$ is denoted by  $w \insrt \ctb$, resulting in the configuration tableau $\ctb_w$.

\end{algorithm}

By  this algorithm we  see  that the restriction of $\tA^+_{\itoj{1}{n}}$ to $\tA^+_{\itoj{\ell}{n}}$, i.e., to words over the convex sub-alphabet $\tA_{\itoj{\ell}{n}} : = \{ \lt_\ell, \dots, \lt_n \} \cnxsset \{ \lt_1, \dots, \lt_n \}$, reads in terms of configuration tableaux as the image of right injection $\RInj: \CTab_{n-\ell+1} \Into \CTab_n  $, cf.
Remark~\ref{rmk:ctab.injection}.

\pSkip

Due to the special structure of configuration tableaux, provided by  the configuration laws~ \eqref{eq:ctab.order}, we obtain the next  important property.

\begin{observation}\label{obs:ctab.wlk} The bijection
 $\tabctab: \Tab_n \To \CTab_n$ (Theorem \ref{thm:configuration}), together with our construction, provides the equality
$$\sword{\itoj{i}{j}}{}\big(\tb_w \big) = \mxtbwt{}_{i,j}(\ctb_w), \qquad i \leq j, $$ for any $w \in \tA_n^*$, cf. \eqref{eq:cont.func.swords} and \eqref{eq:max.inc.walk}. Namely,
weights of  \dwalk s $\wlk_{i,j}$ in  a configuration tableaux ~$\ctb_w $ are translated to subwords over the convex sub-alphabet $\tA_{\itoj{i}{j}} \subset \tA_{\itoj{1}{n}} = \tA_n$ in $\tb_w$, and vice versa.
  That is, the weight $\mxtbwt{}_{i,j}(\ctb_w)$ gives the length of longest nondecreasing subword in $\tb_w$ with letters in $\tA_{\itoj{i}{j}}$.

 The weight of a \dwalk \ $\wlk_{i,j}$ from cell $(i,i)$ to cell $(1,j)$ in $\ctb \in \CTab_n$, $j \geq i$,
 is greater than the weight of every \dwalk \ $\wlk_{k,j}$ from $(k,k)$ to  $(1,j)$ for $k \geq i$, i.e., $\tbwt{\wlk_{i,j}} \geq \tbwt{\wlk_{k,j}}$ for any $i \leq k \leq j$.
Indeed, any such  \dwalk \ $\wlk_{k,j}$ starts with a sub-row whose cells lay above those of $\wlk_{i,j}$, which by the configuration law \eqref{eq:ctab.order} has a lower weight, and this argument applies inductively.

\begin{equation*}
 \tbrest{\ctb}{i}{k} =  \quad
\begin{array}{c|c|c|c|c|c|c|c|}
 \cline{2-4}
 k & \circ  &      &   \\ \cline{2-5}
 &    &   \circ &   \circ &    \\ \cline{2-6}
  &  &      &  &   ..  &   \\ \cline{2-7}
  & &     &     & ..    &    \circ  &    \\ \cline{2-8}
 i & \bullet  &  \bullet    &   \bullet  & ..    &   \bullet   &   \bullet  & \bullet \\ \cline{2-8}
 \multicolumn{8}{r}{j}
\end{array}
\end{equation*}
(This  property fits well with Theorem \ref{thm:Schen}, as configuration tableaux record numerically Young tableaux.)

\end{observation}

\subsection{Reversal of words}
\sSkip
 Recall that the reversal $\rvs{w}$ of a word $w= \std{w}$ is the rewriting  of $w$ from right to left (Definition \ref{def:wordreversing}).
%
%
Tableaux of  reversals of words   in general do not preserve the equivalence  relation $\tcong$,  in sense that we may have
$\tab(u) \neq  \tab(v)$ with $\tab(\rvs{u}) =  \tab(\rvs{v})$, or vice versa.  The same may happen for configuration tableaux, cf. Theorem \ref{thm:configuration}.
\begin{example}
  For example, let $ w =  bc \, aabc$,
  $u = b cc \, aa b$ and $ v =  c \, bc \, aab$, for which $\tb_u \neq \tb_v $, i.e., $u \not \tcong v$.
  But the tableau $\tb_w$ of $w$  is the same as the tableau of the reversal of both $u$ and $v$, that is
  $\tb_w= \tab(\rvs{u})= \tab(\rvs{v}).$

\end{example}
Howeverm  for the case of standard tableaux we do have the following useful correspondence.

\begin{proposition}
  \label{prop:reverse.std}
Suppose $\tb_u = \tab(v)$ and $\tb_v = \tab(v)$ are standard tableaux in $\STab_n$, then
$$  \tb_u = \tb_v \Iff \tb'_{u} = \tb'_{v}, $$
where $\tb'_u = \tab(\rvs{u})$ and $\tb'_v = \tab(\rvs{v})$.
\end{proposition}
\begin{proof}
Let $\tb^\trn_u$ and  $\tb^\trn_v$ be the transpose tableaux of $\tb_u$ and  $\tb_v $, respectively.    Then $\tb_u^\trn = \tab(\rvs{u})$ and $\tb_v^\trn = \tab(\rvs{v})$, by  \cite[Theorem 3.2.3]{SaganBook}, and the  assertion follows at once.
\end{proof}

In comparison to the case of standard tableaux, due to their configuration laws, standard configuration tableaux do not always admit transposition in classical terms.
\begin{example}\label{exmp:permutation}
The word $\std{u} = bdac$ and its reversal $\rvs{u} = cadb$ are    expressed by configuration tableaux as
$$\ctab(u) =  \mresize{\begin{array}{|l|l|l|l|l|l|l|}
\cline{1-1}
  0  \\ \cline{1-2}
  0 &  0  \\ \cline{1-3}
  1 &   0  &   1  \\ \cline{1-4}
  1 &   0  &   1  & 0 \\ \cline{1-4}
\end{array}} \dss\To  \ctab(\rvs{u}) = \mresize{
\begin{array}{|l|l|l|l|ll}
\cline{1-1}
  0  \\ \cline{1-2}
  0 &  0  \\ \cline{1-3}
  0 &   1  &   1  \\ \cline{1-4}
  1 &   1  &   0  & 0 & .\\ \cline{1-4}
\end{array}}
$$
The  reversal of  $\std{v} = dbac$ is $\rvs{v} = cabd$, and they described in terms of  configuration tableaux as
$$ \ctab(v) = \mresize{ \begin{array}{|l|l|l|l|l|l|l|}
\cline{1-1}
  0  \\ \cline{1-2}
  0 &  0  \\ \cline{1-3}
  1 &   0  &   1  \\ \cline{1-4}
  1 &   0  &   1  & 0 \\ \cline{1-4}
\end{array}} \dss\To \ctab(\rvs{v}) = \mresize{
\begin{array}{|l|l|l|l|l|l|l|}
\cline{1-1}
  0  \\ \cline{1-2}
  0 &  0  \\ \cline{1-3}
  0 &   1 &   0  \\ \cline{1-4}
  1 &   1  &   0  & 1 \\ \cline{1-4}
\end{array}}
$$
Thus,  $u \neq v$  with $\ctab(u) = \ctab(v)$, but $\ctab(\rvs{u}) \neq \ctab(\rvs{v})$.
\end{example}
 So we see that reversing of words is not translated to transposition of standard configuration tableaux,  yet we can employ configuration tableaux for the study of reversed words.

\begin{lemma}\label{lem:reverse.2} Suppose $\ctb_w = \ctab(w)$ is
 a standard $n$-configuration tableau in $\SCTab_n$  with nonzero cells $\cent_{i,k} =1 $ and $\cent_{j,k+1} =1$,  where  $i < j$. Then, for the nonzero cells $\cent'_{i',k}$ and $\cent'_{j',k+1}$ of   $\ctb_w' = \ctab(\rvs{w})$, we have  $j' \leq i'$.
  \end{lemma}
\begin{proof} Since $j> i$ where  $\cent_{i,k} =1 $ and $\cent_{j,k+1} =1$, the translation of  $\ctb_w = \ctab(w)$ to $\tb_w = \tab(w)$ by Theorem \ref{thm:configuration}  means that in $\tb_w$ the letter $\lt_{k+1}$ is on the left to $\lt_{k}$, which implies that $\lt_{k+1}$ is on the right to $\lt_{k}$ in $\tb_w' = \tab(\rvs{w})$ by Proposition \ref{prop:reverse.std}. Then, the one-to-one correspondence  $\tabctab:\tb_w' \mTo \ctb_w'$,by Theorem \ref{thm:configuration}, gives the required.
\end{proof}

Now we are first  able to directly tied the underlying congruence  $\clcong$ of the
\kmon\ (Definition~ \ref{def:cloaktic.mon}) to tableau equivalence $\tcong$ in the class of standard tableaux.

\begin{theorem}\label{thm:tab.wxr(w)}
Let  $\tb_u = \tab(u)$ and $\tb_v = \tab(v)$ be   standard tableaux in $\STab_n$, with $u,v \in \tA^+_n$, then
$$ \tb_u = \tb_v \quad  (\Leftrightarrow \tb'_u = \tb'_v) \Iff  u \clcong v \ \text{ and } \  \rvs{u} \clcong \rvs{v}, $$
where $\tb'_u = \tab(\rvs{u})$ and $\tb'_v = \tab(\rvs{v})$.
\end{theorem}

\begin{proof}  Recall that by definition $u \tcong v $ iff $\tb_u = \tb_v $. The part  $(\Leftrightarrow \tb'_u = \tb'_v)$  has already been  proven in Proposition \ref{prop:reverse.std}, which also asserts that  $u \clcong v$ iff $\tb_u \clcong \tb_v$. We prove the rest.
\pSkip
$(\Rightarrow)$: Immediate since $\tcong$ implies $\clcong$, by Proposition \ref{prop:tab-to-cloc}.
\pSkip
$(\Leftarrow)$:   Let $\ctb_u = \ctab(u) $, $\ctb_v = \ctab(v)  $, $\ctb'_u = \ctab(\rvs{u}) $,  and $\ctb'_v = \ctab(\rvs{v})  $,  and denote their cells respectively by $[\cent_u]_{i,j}$, $[\cent_v]_{i,j}$, $[\cent'_u]_{i,j}$,  and $[\cent'_v]_{i,j}$.
Assume that $\tb_u \neq \tb_u$, and thus $\tb'_{u} \neq \tb'_{v}$ by Proposition~ \ref{prop:reverse.std}, implying  that $\ctb_u \neq \ctb_v$ and $\ctb'_u \neq \ctb'_v$, by Corollary~ \ref{cor:sctab}. Recall that
 each diagonal of a standard $n$-configuration tableau has a unique nonzero cell of value $1$.

Proof by induction on $n$.
 The cases of $n =1$ and $n=2$ are clear. Assuming  the implication holds for $\tA_{n-1} \cnxsset \tA_n$, the only change happens by including the extra letter $\lt_n$ is at the $n'$th diagonal of
 standard $n$-configuration tableaux.  Then, by the induction assumption, $\ctb_u$ and $\ctb_v$ have the same $(n-1)$'th diagonal whose  nonzero cell is say at
 $\cent_{k,n-1} =1$.

 Suppose that $[\cent_u]_{i,n}=1$  and $[\cent_v]_{j,n}=1$, say for  $j > i$,  and  assume that $u \clcong v$  (or equivalently $\tb_u \clcong \tb_v$), which  in terms of the function \eqref{eq:cont.func.swords} implies
  that the equality  $\sword{\itoj{s}{t}}{}(\tb_u) = \sword{\itoj{s}{t}}{}(\tb_v)$ holds for all $1 \leq s \leq t \leq n $. We deduce that $ j > i > k $, since otherwise we would get
 $\sword{\itoj{i}{n}}{}(\tb_u) > \sword{\itoj{i}{n}}{}(\tb_v)$.

Similarly, consider   the $n$-configuration tableaux $\ctb'_u$ and $\ctb'_v$  of $\rvs{u}$ and $\rvs{v}$ respectively, and  let  $[\cent'_u]_{k',n-1}$,
$[\cent'_u]_{i',n},$  and $[\cent'_v]_{j',n}$ be respectively the nonzero cells of $\ctb'_u$ and $\ctb'_v$, say with $i' > j'$.  Then, assuming that $\rvs{u} \kcong \rvs{v}$,  by  the same argument as above  we obtain   $i' > j' > k'$,
since otherwise we would have   $\sword{\itoj{j'}{n}}{}(\ctb'_v) > \sword{\itoj{j'}{n}}{}(\ctb'_u)$.
But, since $ j > i > k $, by Lemma \ref{lem:reverse.2}  we should have
  $j' \leq k'$ and $i' \leq k'$ -- a contradiction.
\end{proof}

\section{Representations of tableaux and of the plactic monoid}\label{sec:plc.rep}
In this section we utilize the tropical representations of the \kmon\ (\S\ref{ssec:clk.rep}) and the \ckmon\ (\S\ref{ssec:co.clk.rep}),  which essentially record lengths of longest subwords over convex sub-alphbets, to construct linear  representations of semi-standard tableaux.
Linear representations of the plactic monoid $\PLC_n$ are then follow  from the correspondence between the elements of $\PLC_n$ and $n$-configuration tableaux (Corollary ~\ref{cor:plc.to.ctab}).  The latter correspond uniquely to semi-standard tableaux via the map $\tabctab: \Tab_n \Isoto \CTab_n$ (Theorem \ref{thm:configuration}). Our next step is to establish the two maps
\begin{align*}
  \ctabmat : \  & \CTab_n \ONTO \MPlc_n,  \\
   \ctabcomat: \  &  \CTab_n \ONTO \dMPlcn,
\end{align*} from $n$-configuration tableaux to tropical matrices. Then their compositions with the bijection $\tabctab: \Tab_n \To \CTab_n$ provide the maps
\begin{align*}
 \tabmat :=  \ctabmat \circ \tabctab : \  & \Tab_n \ONTO \MPlc_n,  \\
  \tabcomat := \ctabcomat \circ \tabctab : \  &  \Tab_n \ONTO \dMPlcn.
\end{align*}
 At this point, the digraph realization  of  tropical matrices (cf. \S \ref{ssec:digraph}) is of major importance.
 To make our image matrices more comprehensible, in what follows, for simplicity, \emph{we set  the formal variable  $\pv$ in  the troplacitic matrix algebra $\mfA_n$ to have fixed value  $\pv := 1$, cf. \eqref{eq:mAlg}.}

\pSkip

The  study in  this section is accompanied  with a collection of pathological examples that  demonstrate the difficulties towards faithfully representing tableaux. Our development is supported  by the use   of configuration tableaux that  enables an easier analysis and helps  to better understanding the combinatorial arguments. But before that, in order to complete these representations,  we need another crucial competent.

\subsection{The co-plactic monoid}\label{ssec:co.plc} \sSkip

 \Cmirr ing  of words  (Definition  \ref{def:mirror}) leads to the following monoid structure,  drown from the plactic monoid $\PLC_n$ (Definition \ref{def:plactic.mon}). Recall that $\pcong$ is an additional notation for the underlying congruence ~$\kcong$ of $\PLC_n$.

\begin{definition}\label{def:co.plactic.mon}
The \textbf{co-plactic monoid} is the monoid $\coplcM_n := \CPLC(\tA_n)  $ generated by a finite ordered  set of elements $\tA_n : = \{\aa_1, \dots,  \aa_n \}$,  subject to the equivalence relation $\cpcong$, defined as
$$ u \cpcong v \Iff \lcmp{}{v} \pcong \lcmp{}{u}.$$
Namely $\coplcM_n := \tA_n^*/ _{\cpcong}$.
    We say that $\CPLC_n$ is of \textbf{rank} $n$, and write $\CPLC_n$ for $\CPLC(\tA_n)$ when $\tA_n$ is arbitrary.
  \end{definition}

In other words, the relation $\cpcong$ is satisfied  if the words $\aa'_\ell:= \lcmp{n}{\aa_\ell}$ in $\tA_n^*$, (lexicographically) ordered as $\aa_1' < \aa'_2 < \cdots < \aa'_n$ (cf. Remark \ref{rmk:mir.smon}), admit the Knuth relations \eqref{eq:knuth.rel}. Note that here we initially consider only finitely generated monoid, to have the \cmirr ing  map  well defined.

\begin{theorem}\label{thm:coplc2plc}
The \cmirr s \eqref{eq:word.mirr} of the generators $\lt_1, \dots , \lt_n$ of the plactic monoid  $\plcM_n$,
$$\lt'_\ell :=  \lcmp{}{\lt_\ell}, \qquad \ell = 1, \dots,n ,  $$
ordered
as $ \lt'_1 < \lt'_2 < \cdots < \lt'_n,  $
admit  the Knuth relations \eqref{eq:knuth.rel} and thus the congruence  $\pcong$ of  $\plcM_n$ implies the equivalence $\cpcong$.
\end{theorem}

Configuration  tableaux correspond bijectively to elements of $\plcM_n$ (Corollary \ref{cor:plc.to.ctab}) and are  utilized    to prove the theorem by a straightforward computation, heavily  based  on the Encoding Algorithm \ref{algr:conf}. The proof is performed by   induction on $n$, where the induction step assumes the implication for $\{ \lt_2, \dots, \lt_n\}$, which corresponds to assuming the implication for the upper $n-1$ rows of an $n$-configuration tableau.
The proof is fairly technical, including several cases, and appears in its full details in Appendix \ref{apx:A}, together with some  additional examples.

Recall from
Remark \ref{rmk:mir.smon} that $\sMon{\tA^*_n}{1}$ is the finitely monoid generated by $\lcmp{}{\lt_1}, \dots, \lcmp{}{\lt_n}$, and associated with the homomorphism $\pMap_1: \tA^*_n \To  \sMon{\\tA^*_n}{1}$.

\begin{corollary}\label{cor:coplc2plc}
  The congruence  $\pcong$ implies  the equivalence $\cpcong$ (which by itself is then a congruence), the map
    $$ \pMap:= \pMap_1: \plcM_n \ISOTO \sMon{\plcM_n}{1}$$  is a monoid isomorphism, and
 the plactic monoid $\plcM_n := \tA^*_n/_{\pcong}$ of rank $n$ is also a co-plactic monoid $\coplcM_n$ of rank $n$ for which the map
 \begin{equation}\label{eq:plc.to.cplc}
\plccplc:    \plcM_n \ONTO \coplcM_n, \qquad [w]_\plc  \mTo [w]_\coplc
 \end{equation}
 is a surjective homomorphism.
 \end{corollary}

\begin{corollary}
  The plactic monoid $\plcM_n = \genr{\lt_1, \dots, \lt_n}$ contains the inductive chain of plactic submonoids $$\plcM_n = \sMon{\plcM_n}{0} \ds\supset \sMon{\plcM_n}{1} \ds\supset \sMon{\plcM_n}{2} \ds\supset \cdots  $$ of rank $n$, generated by $\smon{\lt_1}{i} =  \lcmp{n}{\smon{\lt_1}{i-1}}, \dots, \smon{\lt_n}{i} =  \lcmp{n}{\smon{\lt_n}{i-1}}$. Furthermore, the maps
  $$ \pMap_i: \sMon{\plcM_n}{i-1} \ISOTO \sMon{\plcM_n}{i}, \qquad  i = 1,2, \dots ,$$
  cf. Remark \ref{rmk:mir.smon}, are monoid isomorphisms.
\end{corollary}

\begin{remark}\label{rmk:ccoplc.to.coclk}
The surjective  homomorphism $ \plktoclk: \plcM_n \onto{1} \clkM_n$
(Corollary \ref{cor:plc.to.clk}) induces the  surjective  homomorphism
  $ \plktocclk: \plcM_n \To \coclkM_n$ via the diagram
 \begin{equation}\label{eq:diag.1} \begin{gathered}
 \xymatrix{
  \plcM_n   \ar@{->>}[r]^{ \plktoclk} \ar@{->>}[dr]^{\co{\ \plktoclk}} \ar@{->>}[d]   & \clkM_n  \ar@{->}[r]^\Sim _\clkrep &  \MPlc_n   &     \\
 \coplcM_n \ar@{->>}[r]^{ \plktoclk} & \coclkM_n \ar@{->}[r]^\Sim_\dclkrep& \mir{\MPlc_n} \ar@{_{(}->}[u]  &
 }\end{gathered}
    \end{equation}
where $\clkM_n$ and $\coclkM_n$ are linearly represented  by  $\MPlc_n $ and $\dMPlcn $, respectively.

\end{remark}

\subsection{Representations of configuration tableaux}\label{ssec:rep}
\sSkip

We start by defining  our  first map
\begin{equation}\label{map:ctabtfrm.1}
\ctabmat: \CTab_n \TO \MPlc_n,
\end{equation}
that sends an $n$-configuration tableau $\ctb = (\cent_{i,j})$  to the matrix  $U= (u_{i,j})$ determined by
\begin{equation}\label{map:ctabtfrm.2}
u_{i,j} := \left\{\begin{array}{llll}
\mxtbwt{}_{i,j}(\ctb)  & \qquad\text{if  } i \leq  j, \\[1mm]
\underline{\null} \zero & \qquad\text{if  } i > j, \\
\end{array}\right.
\end{equation} where $\mxtbwt{}_{i,j}$ is given by \eqref{eq:max.inc.walk} in \S\ref{ssec:tab.vs.ctab}.
This map is then realized as a monoid homomorphism.

\begin{lemma}\label{lem:ctab.to.mat}
The matrix $W = \ctabmat(\ctb)$ indeed belongs to $\MPlc_n$, cf.  \eqref{eq:mAlg}, for any $\ctb \in \CTab_n$.

\end{lemma}

\begin{proof}
 $\tabctab: \Tab_n \To \CTab_n$ is  bijective  (Theorem \ref{thm:configuration}),  with  $\sword{\itoj{i}{j}}{}\big(\tb_w \big) = \mxtbwt{}_{i,j}(\ctb_w)$  by Observation ~\ref{obs:ctab.wlk}, while  $\Tab_n$  and $\CTab_n$ are isomorphic (as monoids) to the plactic monoid $\plcM_n$ (Remark ~\ref{rmk:tab.mon} and Corollary \ref{cor:plc.to.ctab}, respectively).
  Furthermore, the map
$ \plktoclk: \plcM_n \To \clkM_n$, $[w]_\plc \mto [w]_\clk,$
is a surjective monoid  homomorphism  (Corollary \ref{cor:plc.to.clk}), while
$\clkrep: \clkM_n \To \mfAml_n$, $\lt_\ell \mto \Pv{\Oix{A}{\ell}}$ for $\ell  = 1, \dots, n$, is an  isomorphism
(Theorem \ref{thm:clk.represntation}). Thus we only need to show that $\ctabmat$ maps the  $n$-configuration tableaux assigned  to  letters $\lt_\ell$ in ~ $\tA_n$ to the generating matrices  $\Oix{A}{\ell}$ of ~$\mfAml_n$, namely that
$\ctabmat(\ctb_{\lt_\ell}) = \Oix{A}{\ell}$ for each  $\ell  = 1, \dots, n$.

Drawing the $n$-configuration tableau of a letter $\lt_\ell$, which is of the  form
\begin{equation}\label{eq:C.l}
\ctb_{\lt_\ell} = \ctab(\lt_\ell) = \begin{array}{c} \footnotesize \begin{array}{|l|l|l|l|l|l|l|l}
\cline{1-1}
  {0}   \\ \cline{1-2}
  {0}  & 0 \\ \cline{1-3}
  \vdots &  \vdots & \ddots \\  \cline{1-4}
   0   &     &  & 0 \\ \cline{1-5}
  \vdots &    &   \cdots & &  \ddots \\ \cline{1-6}
   0 &  \cdots   &    &  & \cdots & 0 \\ \cline{1-7}
  {0} &   \cdots &   0 & 1 & 0 & \cdots & 0 & ,\\ \cline{1-7}
 \end{array}  \\  \begin{array}{llllllllll} & \   \ell & & \end{array}  \end{array}
\end{equation}
 by the definition of \dwalk s (\S\ref{ssec:y.tab}) and their weights  \eqref{eq:max.inc.walk}  we observe that
\begin{enumerate}
  \eroman
  \item $\mxtbwt{}_{i,j}(\ctb_{\lt_\ell}) = 0$ for $i \leq j < \ell$;
  \item $\mxtbwt{}_{i,j}(\ctb_{\lt_\ell}) = 1$ for $i \leq  \ell \leq j $;
  \item $\mxtbwt{}_{i,j}(\ctb_{\lt_\ell}) = 0$ for $ \ell < i \leq j$.
\end{enumerate}
Accordingly, the matrix $ \ctabmat(\ctb_{\lt_\ell})$ defined  by \eqref{map:ctabtfrm.2}  is  precisely the triangular matrix
$\Oix{A}{\ell} := \Oix{F}{\ell} \+ E$ defined in ~\eqref{eq:F.mat.2}
and \eqref{eq:mAlg},  with $\pv=1$ and $E$ as in \eqref{eq:E.mat}.
\end{proof}
By Remark \ref{rmk:tab.mon} we conclude that:
\begin{theorem}\label{thm:ctb.mat}
  $\ctabmat: \CTab_n \To \MPlc_n$, given in \eqref{map:ctabtfrm.2}, is a surjective map, realized a as monoid homomorphism.   \end{theorem}

\begin{proof}
We know that  $ \scP_\ctab: \plcM_n  \To \CTab_n$ is a bijection (Corollary \ref{cor:plc.to.ctab}) realized as a monoid homomorphism, $ \plktoclk: \plcM_n \To \clkM_n$ is a surjective monoid homomorphism (Corollary \ref {cor:plc.to.clk}),
  and $\clkrep: \clkM_n \To \MPlc_n$ is an isomorphism (Theorem \ref{thm:clk.represntation}), thus the diagram
 \begin{equation*}\label{eq:diag.2} \begin{gathered}
 \xymatrix{
 \plcM_n  \ar@{<->}[rr]^{\SIM{\scP_\ctab}} \ar@{->>}[drr]^{\plktoclk} & & \CTab_n \ar@{..>}[drr]^\ctabmat \\
 & & \clkM_n \ar@{->}[rr]^{\clkrep \quad \Sim } && \MPlc_n
}\end{gathered}
    \end{equation*}
commutes by Lemma \ref{lem:ctab.to.mat}.
\end{proof}
\begin{example}\label{exmp:tfrm.2} Applying the map $\ctabmat$ in \eqref{map:ctabtfrm.1} to the $3$-configuration tableau  $\ctb_w$ with  $w$  the word in  Example \ref{exmp:ctab.1}, we have
$$ \tb_w = \mresize{
\begin{array}{|l|l|l|l|l|l|l|}
\cline{1-1}
  {\lt_4}   \\ \cline{1-2}
  {\lt_3}  &   {\lt_4} \\ \cline{1-4}
  {\lt_2} &   {\lt_2}  &   {\lt_4} &  {\lt_4} \\ \cline{1-6}
  {\lt_1}  &    {\lt_1} &     {\lt_1}  & {\lt_3} &    {\lt_3} &    {\lt_4} \\ \cline{1-6}
\end{array}}
\dss \mTo  \ctb_w = \mresize{
\begin{array}{|l|l|l|l|l|l|l|}
\cline{1-1}
  {1}   \\ \cline{1-2}
  {1} &  {1}  \\ \cline{1-3}
  {2} &   {0}  &   {2}  \\ \cline{1-4}
  {3}  &    {0} &     {2}  &  {1}  \\ \cline{1-4}
\end{array}}
\dss \mTo  \ctabmat(\ctb_w) = \mresize{
\begin{bmatrix}
3 & 3  & 5 &  6 \\
& 2 & 4 & 5 \\
& & 3 & 5 \\
& & & 5
\end{bmatrix}},
$$ where
$\tb_w$ is the Young tableau of $w$.
\end{example}

In terms of configuration tableaux, and the map $\ctabmat: \CTab_n \To \MPlc_n$,
 Example \ref{exmp:mat.clk.3x3} now reads as follows.
\begin{example}\label{exmp:tfrm.1}
  The representation $\ctabmat:\CTab_3 \To \MPlc_3$ of $3$-configuration tableaux assigned to  the three letter alphabet $\tA_3 = \genr{a,b,c}$ is determined by the  generators' mapping
\begin{align*}\ctb_a := \ctab(a) = \mresize{
\begin{array}{|l|l|l|l|l|l|l|}
\cline{1-1}
  {0}   \\ \cline{1-2}
  {0} &  {0}  \\ \cline{1-3}
  {1} &   {0}  &   {0}  \\ \cline{1-3}
\end{array}
\ds  \mTo  A =
 \begin{bmatrix}
 1 & 1 & 1 \\
 & 0 & 0 \\
 & & 0
\end{bmatrix} }\ , &
\qquad \ctb_b := \ctab(b) =   \mresize{
\begin{array}{|l|l|l|l|l|l|l|}
\cline{1-1}
  {0}   \\ \cline{1-2}
  {0} &  {0}  \\ \cline{1-3}
  {0} &   {1}  &   {0}  \\ \cline{1-3}
\end{array}
 \ds \mTo B =
 \begin{bmatrix}
 0 & 1 & 1 \\
 & 1 & 1 \\
 & & 0
 \end{bmatrix} }\ ,
\\
\ctb_c := \ctab(c) = \mresize{
\begin{array}{|l|l|l|l|l|l|l|}
\cline{1-1}
  {0}   \\ \cline{1-2}
  {0} &  {0}  \\ \cline{1-3}
  {0} &   {0}  &   {1}  \\ \cline{1-3}
\end{array} \ds \mTo C =
 \begin{bmatrix}
 0 & 0 & 1 \\
 & 0 & 1 \\
 & & 1
\end{bmatrix} }\  . \end{align*}
 The Knuth-equivalence \eqref{eq:knuth.rel} for triplets in $\tA_3$  are described in $\CTab_3$ and $\MPlc_3$ as
$$ \tab(acb)  = \tab(cab)  = \mresize{
\begin{tabular}{|c|c|}
\cline{1-1}
 c  \\ \cline{1-2}
a &  b  \\ \cline{1-2}
\end{tabular}
\dss \mTo
\begin{array}{|l|l|l|l|l|l|l|}
\cline{1-1}
  {0}   \\ \cline{1-2}
  {0} &  {1}  \\ \cline{1-3}
  {1} &   {1}  &   {0}  \\ \cline{1-3}
\end{array}
\dss \mTo
 \begin{bmatrix}
 1 & 2 & 2 \\
 & 1 & 1 \\
 & & 1
\end{bmatrix} }\ , $$
while
$$   \tab(bca) = \tab (bca) = \mresize{
\begin{tabular}{|c|c|}
\cline{1-1}
 b  \\ \cline{1-2}
a &  c  \\ \cline{1-2}
\end{tabular} \dss \mTo
\begin{array}{|l|l|l|l|l|l|l|}
\cline{1-1}
  {0}   \\ \cline{1-2}
  {1} &  {0}  \\ \cline{1-3}
  {1} &   {0}  &   {1}  \\ \cline{1-3}
\end{array}
\dss \mTo
 \begin{bmatrix}
 1 & 1 & 2 \\
 & 1 & 2 \\
 & & 1
\end{bmatrix}} \ .
$$
For the increasing word $abc \in \tA_3^+$ we get
$$    \tab (abc) =  \mresize{
\begin{tabular}{|c|c|c|}
 \cline{1-3}
a & b & c  \\ \cline{1-3}
\end{tabular} \dss \mTo
\begin{array}{|l|l|l|l|l|l|l|}
\cline{1-1}
  {0}   \\ \cline{1-2}
  {0} &  {0}  \\ \cline{1-3}
  {1} &   {1}  &   {1}  \\ \cline{1-3}
\end{array}
\dss \mTo
 \begin{bmatrix}
 1 & 2 & 3 \\
 & 1 & 2 \\
 & & 1
\end{bmatrix} },
$$
and for the decreasing word $cba \in \tA_3^+$ we have
$$    \tab (cba) =  \mresize{
\begin{tabular}{|c|}
 \cline{1-1}
c \\ \cline{1-1} b \\ \cline{1-1} a  \\ \cline{1-1}
\end{tabular} \dss \mTo
\begin{array}{|l|l|l|l|l|l|l|}
\cline{1-1}
  {1}   \\ \cline{1-2}
  {1} &  {0}  \\ \cline{1-3}
  {1} &   {0}  &   {0}  \\ \cline{1-3}
\end{array}
\dss \mTo
 \begin{bmatrix}
 1 & 1 & 1 \\
 & 1 & 1 \\
 & & 1
\end{bmatrix} }.
$$

For a generic  3-letter tableau $ \tb_w := \cc^{k_3} \, \bb^{j_2} \cc^{k_2} \,\aa^{i_1} \bb^{j_1} \cc^{k_1}$, by a direct computation,  we obtain
$$ \ctb_w = \begin{array}{|l|l|l|l|l|l|l|}
\cline{1-1}
  k_3  \\ \cline{1-2}
  j_2 &  k_2  \\ \cline{1-3}
  i_1 &   j_1  &   k_1  \\ \cline{1-3}
\end{array}\dss{\mTo} \ctabmat(\ctb_w) =  \left[\begin{array}{rrr}
                               i_1 & i_1 j_1 & i_1 j_1 k_1 \\
                                & j_1j_2 & j_2(j_1 \+ k_2) \\
                                 &  & k_1  k_2  k_3
                             \end{array}\right].
$$
Then taking $k_2', k_2''$ such that  $k_2', k_2'' \leq j_2$ and $k'_2 k_3' = k''_2 k_3''$, we can produce two different tableaux with a same image, what shows that in general  the map $\ctabmat$ not injective.

\end{example}

Returning to our running example (cf. Example \ref{exmp:run:1}), it gives an explicit numerically example for the non-injectivity of the map $\ctabmat:\Tab_n \To \MPlc_n$.
\begin{example}\label{exmp:inject.2}  Taking the explicit words of Example \ref{exmp:run:1}, we have
$$u = c \ b^2 c^2 \ a^2 b^2 c  \dss \mTo \ctb_u =  \mresize{
\begin{array}{|l|l|l|l|l|l|l|}
\cline{1-1}
  {1}   \\ \cline{1-2}
  {2} &  {2}  \\ \cline{1-3}
  {2} &   {2}  &   {1}  \\ \cline{1-3}
\end{array}
 \dss \mTo A_u =
 \begin{bmatrix}
 2 & 4 & 5 \\
 & 4 & 5 \\
 & & 4
\end{bmatrix}}, $$
where on the other also
$$v  = c^2 \ b^2 c \ a^2 b^2 c  \dss \mTo \ctb_v =  \mresize{
 \begin{array}{|l|l|l|l|l|l|l|}
\cline{1-1}
  {2}   \\ \cline{1-2}
  {2} &  {1}  \\ \cline{1-3}
  {2} &   {2}  &   {1}  \\ \cline{1-3}
\end{array}
 \dss \mTo A_v =
 \begin{bmatrix}
 2 & 4 & 5 \\
 & 4 & 5 \\
 & & 4
\end{bmatrix} }. $$
Hence  $\ctabmat(\ctb_u) = \ctabmat(\ctb_v)$, but $u \not \tcong v $, or equivalently
 $u \not \pcong v $ by Theorem \ref{thm:Knuth}. Namely, the map  $\ctabmat: \CTab_3 \To \MPlc_3$ is not injective
\end{example}

Nevertheless,  we do have a partial (row) injection:

\begin{lemma}\label{lem:row.1} The restriction
\begin{equation}\label{eq:ctabtfrm.resmap}
\ctabmat|_{\rw_1}: \row_{1}(\ctb)
\TO  \row_{1}(W), \qquad W = \ctabmat(\ctb),
\end{equation}
of the map $\ctabmat$ in \eqref{map:ctabtfrm.1} to the bottom row is a bijective map.
\end{lemma}
\begin{proof} By Theorem \ref{thm:configuration},  $\row_{1}(\ctb)$ corresponds uniquely to a  nondecreasing word in $\tA_n^+$, which is mapped bijectively to $\row_{1}(W)$ of $W \in \MPlc_n$  by Lemma \ref{lem:injec.ndc}.
\end{proof}

\subsection{Co-representations of configuration tableaux}\label{ssec:co.rep}
\sSkip

We now turn to our second map
\begin{equation}\label{map:ctabcomat.1}
\ctabcomat: \CTab_n \TO \dMPlcn,
\end{equation}
which we  called a \textbf{co-representation}, that uses as its target the matrix co-monoid $\dMPlcn$, cf.  \eqref{eq:mat.co.alg}.
This map sends an  $n$-configuration tableau $\ctb \in \CTab_n$  to the matrix  $V = (v_{i,j})$ defined by
\begin{equation}\label{map:ctabcomat.2}
v_{i,j} := \left\{\begin{array}{llll}
-\jmp{i', j'}{}(\ctb)   & \qquad\text{if  } i \leq  j, \\[1mm]
\underline{\null} \zero & \qquad\text{if  } i > j. \\
\end{array}\right.
\end{equation}
where $i' = n -i +1$, $j' = n- j +1$, and $\jmp{i',j'}{}$ is the function given by \eqref{eq:stp.2} in \S\ref{ssec:tab.vs.ctab}.
The  map $\ctabcomat$ is then realized as a monoid homomorphism.
(Recall that,  by Theorem \ref{thm:co.clk.represntation}, $\dMPlcn$ is isomorphic to the \ckmon\ ~$\coclkM_n$.)

\begin{lemma}\label{lem:ctab.to.comat}
The matrix $V = \ctabcomat(\ctb)$ indeed belongs to $\mir{\MPlc_n} $ for any $\ctb \in \CTab_n$.

\end{lemma}

\begin{proof} We follow similar arguments as in the proof of Lemma \ref{lem:ctab.to.mat}, with an additional step.  First,
 $\tabctab: \Tab_n \To \CTab_n$ is a bijection (Theorem \ref{thm:configuration}) with  $\Tab_n$  and $\CTab_n$ isomorphic (as monoids) to the plactic monoid $\plcM_n$ (Theorem \ref{thm:configuration}, Remark ~\ref{rmk:tab.mon},  and Corollary \ref{cor:plc.to.ctab}, respectively).
  On the other hand,
$ \plktoclk: \plcM_n \To \clkM_n$, $[w]_\plc \mto [w]_\clk,$
is a surjective  homomorphism  (Corollary \ref{cor:plc.to.clk}), while
$\clkrep: \clkM_n \To \mfAml_n$, $\lt_\ell \mto \Pv{\Oix{A}{\ell}}$ is an isomorphism
(Theorem \ref{thm:clk.represntation}). Finally we have the surjective homomorphism
 $\matcmat :  \MPlc_n \To \dMPlcn,$ $\Oix{A}{\ell} \mto \Oix{\chA}{\ell }$, cf. \eqref{eq:mat.2.dmat}.
Thus we only need to show that $\ctabcomat$ maps each $n$-configuration tableau assigned to the letter $\lt_\ell$ in $\tA_n$ to the matrix $\Oix{A}{\ell}$ of $\mfAml_n$, namely that
$\ctabcomat(\ctb_{\lt_\ell}) = \Oix{\chA}{\ell}$.

Observing the $n$-configuration tableau $\ctb_{\lt_\ell}$ assigned to the letter $\lt_\ell$, see  \eqref{eq:C.l},
 by the  definition of \hwalk s and \eqref{eq:stp.2} we obtain that:
\begin{enumerate}
  \eroman
  \item $\jmp{i,j}{}(\ctb) = 1$ for $i = j = \ell$;
  \item $\jmp{i,j}{}(\ctb) = 0$ otherwise.
\end{enumerate}
Then the image  matrix $\ctabcomat(\ctb_{\lt_\ell})$ defined  by   \eqref{map:ctabcomat.2} is precisely  the triangular matrix
$\Oix{\chA}{\ell} := \Oix{F}{\ell} \+ E$ in~ \eqref{eq:co.cl.A.mat},  with $\pv=1$ and $\one = 0$, as usual.
\end{proof}

\begin{theorem}\label{thm:ctb.comat}
  $\ctabcomat: \CTab_n \To \dMPlcn$, given by \eqref{map:ctabcomat.2}, is a surjective map, realized a as monoid homomorphism.   \end{theorem}

\begin{proof}
We know that  $ \scP_\ctab: \plcM_n  \To \CTab_n$ is a bijection (Corollary \ref{cor:plc.to.ctab}) realized as a monoid homomorphism,
$ \pMap: \plcM_n \onto{1} \coplcM_n$ is a surjective homomorphism  (Corollary \ref{cor:coplc2plc}),
$ \plktocclk: \coplcM_n \To \coclkM_n$ is a surjective  homomorphism induced from  $ \plktoclk: \coplcM_n \To \coclkM_n$ (Corollary \ref{cor:plc.to.clk}),
  and $\dclkrep: \clkM_n \Isoto \mir{\MPlc_n}$ is an isomorphism (Theorem \ref{thm:co.clk.represntation}). Hence,  by Lemma \ref{lem:ctab.to.comat}, the diagram
 \begin{equation*}\label{eq:diag.2} \begin{gathered}
 \xymatrix{
 \plcM_n  \ar@{<->}[rr]^{\SIM{\scP_\ctab}} \ar@{->>}[d]^{\pMap} & & \CTab_n \ar@{..>}[drr]^\ctabcomat \\
\coplcM_n \ar@{->>}[rr]^{\plktoclk} & & \coclkM_n \ar@{->}[rr]^{\dclkrep \quad \Sim} && \mir{ \MPlc_n}
}\end{gathered}
    \end{equation*}
commutes.
\end{proof}

In comparison to the representation $\ctabmat:\CTab_3 \To \MPlc_3$ in  Example \ref{exmp:tfrm.1},  for the co-representation
$\ctabcomat:\CTab_3 \To \mir{\MPlc_3}$ we have the following.

\begin{example}\label{exmp:tfrm.2}
 The $3$-configuration tableaux corresponding to the three letter alphabet $\tA_3 = \genr{a,b,c}$ are co-represented by $\ctabcomat$ in $\mir{\MPlc_3}$  by the matrices
\begin{align*} \ctb_a := \ctab(a) =  \mresize{
\begin{array}{|l|l|l|l|l|l|l|}
\cline{1-1}
  {0}   \\ \cline{1-2}
  {0} &  {0}  \\ \cline{1-3}
  {1} &   {0}  &   {0}  \\ \cline{1-3}
\end{array}
 \ds \mTo
  \chA =
 \begin{bmatrix}
 0 & 0 & 0 \\
 & 0 & 0 \\
 & & -1
\end{bmatrix} },  &
\qquad \ctb_b := \ctab(b) =  \mresize{
\begin{array}{|l|l|l|l|l|l|l|}
\cline{1-1}
  {0}   \\ \cline{1-2}
  {0} &  {0}  \\ \cline{1-3}
  {0} &   {1}  &   {0}  \\ \cline{1-3}
\end{array}
 \ds \mTo
 \chB =
 \begin{bmatrix}
 0 & 0 & 0 \\
 & -1 & 0 \\
 & & 0
 \end{bmatrix} },
\\
\ctb_c :=
\ctab(c) =    \mresize{
\begin{array}{|l|l|l|l|l|l|l|}
\cline{1-1}
  {0}   \\ \cline{1-2}
  {0} &  {0}  \\ \cline{1-3}
  {0} &   {0}  &   {1}  \\ \cline{1-3}
\end{array} \ds \mTo
  \chC  =
 \begin{bmatrix}
 -1 & 0 & 0 \\
 & 0 & 0 \\
 & & 0
\end{bmatrix}  }.  \end{align*}
 The Knuth-equivalence \eqref{eq:knuth.rel} for triplets in $\tA_3$  are described in $\mir{\MPlc_3}$ as
$$ \tab(acb)  = \tab(cab)  =  \mresize{
\begin{tabular}{|c|c|}
\cline{1-1}
 c  \\ \cline{1-2}
a &  b  \\ \cline{1-2}
\end{tabular}
\dss \mTo
 \begin{array}{|l|l|l|l|l|l|l|}
\cline{1-1}
  {0}   \\ \cline{1-2}
  {0} &  {1}  \\ \cline{1-3}
  {1} &   {1}  &   {0}  \\ \cline{1-3}
\end{array}
\dss \mTo
 \begin{bmatrix}
 -1 & -1 & 0 \\
 & -1 & 0 \\
 & & -1
\end{bmatrix} } ,  $$
while
$$   \tab(bca) = \tab (bca) =  \mresize{
\begin{tabular}{|c|c|}
\cline{1-1}
 b  \\ \cline{1-2}
a &  c  \\ \cline{1-2}
\end{tabular} \dss \mTo
\begin{array}{|l|l|l|l|l|l|l|}
\cline{1-1}
  {0}   \\ \cline{1-2}
  {1} &  {0}  \\ \cline{1-3}
  {1} &   {0}  &   {1}  \\ \cline{1-3}
\end{array}
\dss \mTo
 \begin{bmatrix}
 -1 & 0 & 0 \\
 & -1 & -1 \\
 & & -1
\end{bmatrix} }.
$$
For the increasing word $abc \in \tA_3^*$ we get
$$    \tab (abc) =  \mresize{
\begin{tabular}{|c|c|c|}
 \cline{1-3}
a & b & c  \\ \cline{1-3}
\end{tabular} \dss \mTo
\begin{array}{|l|l|l|l|l|l|l|}
\cline{1-1}
  {0}   \\ \cline{1-2}
  {0} &  {0}  \\ \cline{1-3}
  {1} &   {1}  &   {1}  \\ \cline{1-3}
\end{array}
\dss \mTo
 \begin{bmatrix}
 -1 & 0 & 0 \\
 & -1 & 0 \\
 & & -1
\end{bmatrix} },
$$
and for the decreasing word $cba \in \tA_3^*$ we have
$$    \tab (cba) =  \mresize{
\begin{tabular}{|c|}
 \cline{1-1}
c \\ \cline{1-1} b \\ \cline{1-1} a  \\ \cline{1-1}
\end{tabular}
\dss \mTo
\begin{array}{|l|l|l|l|l|l|l|}
\cline{1-1}
  {1}   \\ \cline{1-2}
  {1} &  {0}  \\ \cline{1-3}
  {1} &   {0}  &   {0}  \\ \cline{1-3}
\end{array}
\dss \mTo
 \begin{bmatrix}
 -1 & -1 & -1 \\
 & -1 & -1 \\
 & & -1
\end{bmatrix} }.
$$
But
$$   \tab(aca) = \tab (caa) =  \mresize{
\begin{tabular}{|c|c|}
\cline{1-1}
 c  \\ \cline{1-2}
a &  a  \\ \cline{1-2}
\end{tabular}
\dss \mTo
\begin{array}{|l|l|l|l|l|l|l|}
\cline{1-1}
  {0}   \\ \cline{1-2}
  {0} &  {1}  \\ \cline{1-3}
  {2} &   {0}  &   {0}  \\ \cline{1-3}
\end{array}
\dss \mTo
 \begin{bmatrix}
 -1 & 0 & 0 \\
 & 0 & 0 \\
 & & -2
\end{bmatrix}},
$$
where  we also  have
$$   \tab(aac) =  \mresize{
\begin{tabular}{|c|c|c|}
 \cline{1-3}
a &  a  & c\\ \cline{1-3}
\end{tabular}
\dss \mTo
\begin{array}{|l|l|l|l|l|l|l|}
\cline{1-1}
  {0}   \\ \cline{1-2}
  {0} &  {0}  \\ \cline{1-3}
  {2} &   {0}  &   {1}  \\ \cline{1-3}
\end{array}
\dss \mTo
 \begin{bmatrix}
 -1 & 0 & 0 \\
 & 0 & 0 \\
 & & -2
\end{bmatrix}},
$$
which shows that  in general the co-representation  $\ctabcomat: \CTab_n \To \mir{\MPlc_n}$  is not injective.
\pSkip

For a general 3-letter tableau $ \tb_w := \cc^{k_3} \, \bb^{j_2} \cc^{k_2} \,\aa^{i_1} \bb^{j_1} \cc^{k_1}$, by a direct computation,   we have
$$ \ctb_w = \begin{array}{|l|l|l|l|l|l|l|}
\cline{1-1}
  k_3  \\ \cline{1-2}
  j_2 &  k_2  \\ \cline{1-3}
  i_1 &   j_1  &   k_1  \\ \cline{1-3}
\end{array}\dss{\mTo}  \chA_w =  \left[\begin{array}{rrr}
                               -(k_1 k_2 k_3) & -k_3(-j_1 \+ -k_2) & -k_3 \\
                                & -(j_1 j_2) & -j_2 \\
                                 &  & -i_1
                             \end{array} \right] \ ,$$
 showing that the values $j_1, k_2$ make the map $\ctabcomat$ non-injective. Recall that in Example \ref{exmp:tfrm.1}, the cause for non-injectivity  was $j_2, k_2$.
\end{example}

Yet, in other situations, the use of co-representations $\ctabcomat: \CTab_n \To \mir{\MPlc_n}$  can be useful.
\begin{example}\label{exmp:inject.3}  Applying the co-representation  $\ctabcomat: \CTab_3 \To \mir{\MPlc_3}$  to the $3$-configuration tableaux  in  Example \ref{exmp:inject.2}, we have
$$u = c \ b^2 c^2 \ a^2 b^2 c  \dss \longrightarrow \ctb_u =  \mresize{
\begin{array}{|l|l|l|l|l|l|l|}
\cline{1-1}
  {1}   \\ \cline{1-2}
  {2} &  {2}  \\ \cline{1-3}
  {2} &   {2}  &   {1}  \\ \cline{1-3}
\end{array}
 \dss \longrightarrow
 \begin{bmatrix}
 -4 & -3 & -1 \\
 & -4 & -2 \\
 & & -2
\end{bmatrix}},
$$
where on the other
$$v  = c^2 \ b^2 c \ a^2 b^2 c  \dss \longrightarrow \ctb_v =  \mresize{
 \begin{array}{|l|l|l|l|l|l|l|}
\cline{1-1}
  {2}   \\ \cline{1-2}
  {2} &  {1}  \\ \cline{1-3}
  {2} &   {2}  &   {1}  \\ \cline{1-3}
\end{array}
 \dss \longrightarrow
  \begin{bmatrix}
 -4 & -3 & -2 \\
 & -4 & -2 \\
 & & -2
\end{bmatrix}}. $$
Hence  $\ctabcomat(\ctb_u) \neq \ctabcomat(\ctb_v)$, with
 $u \not \tcong v $.
Recall from Example \ref{exmp:inject.2}  that for these tableaux  we had  $\ctabmat(\ctb_u) = \ctabmat(\ctb_v)$ which has shown that  $\ctabmat: \CTab_n \To \MPlc_n$ is not injective.
\end{example}

\begin{corollary}\label{cor:plc2clkcoclk}
  $u \kcong v$ implies both $u \clcong v$ and $u \cclcong v$.
\end{corollary}
\begin{proof}
The implication for the equivalence $\clcong$  has been already proven in Proposition \ref{prop:tab-to-cloc}, while that for the equivalence  $\cclcong$ follows from Theorem \ref{thm:ctb.comat}.
\end{proof}

\subsection{Linear representations of the plactic monoid} \sSkip

  The representations of $n$-configuration tableaux as constructed previously in \S\ref{ssec:rep} and \S\ref{ssec:co.rep} lead directly to a representation of semi-standard tableaux, via the one-to-one correspondence $\tabctab: \Tab_n \To \CTab_n$  (Theorem ~ \ref{thm:configuration}), and thus to a linear representation of the plactic monoid (cf. Remark ~\ref{rmk:tab.mon}).

\begin{theorem} \label{thm:plc.rep} The map
\begin{equation}\label{eq:map.plcrep.1}
\begin{array}{lll}
  \plcrep_n : & \plcM_n \TO \MPlc_n \times \mir{\MPlc_n}, \\[2mm]
  &  \tb_w \mTo  (\ctabmat(\ctb_w), \,  \ctabcomat(\ctb_w)),
\end{array}
   \end{equation}
is a surjective monoid homomorphism -- a linear representation of the plactic monoid $\plcM_n$.

\end{theorem}
\begin{proof} Compose the bijection
$ \scP_\ctab: \plcM_n  \To \CTab_n$ (Corollary \ref{cor:plc.to.ctab}) separately with the two
surjections
  $\ctabmat: \CTab_n \To \MPlc_n$ given by \eqref{map:ctabtfrm.2} (Theorem \ref{thm:ctb.mat})
and
  $\ctabcomat: \CTab_n \To \dMPlcn$ given by \eqref{map:ctabcomat.2} (Theorem ~\ref{thm:ctb.comat}), realized  as monoid homomorphisms  by Remark \ref{rmk:tab.mon}.
\end{proof}

Let $u = bdac$ and $v = dbac$,  relating  respectively  the  $4$-configuration tableaux
$$\ctb_u =  \mresize{\begin{array}{|l|l|l|l|l|l|l|}
\cline{1-1}
  0  \\ \cline{1-2}
  0 &  0  \\ \cline{1-3}
  1 &   0  &   1  \\ \cline{1-4}
  1 &   0  &   1  & 0 \\ \cline{1-4}
\end{array}} \dss \neq  \ctb_v = \mresize{
\begin{array}{|l|l|l|l|ll}
\cline{1-1}
  0  \\ \cline{1-2}
  0 &  1  \\ \cline{1-3}
  1 &   0  &   0  \\ \cline{1-4}
  1 &   0  &   1  & 0 & , \\ \cline{1-4}
\end{array}}
$$  or equivalently   $\tb_u \neq \tb_v$,  and therefore $u \not \pcong v$. On the other hand, their  image under  the representation
$\plcrep_4 : \plcM_4 \To \MPlc_4 \times \mir{\MPlc_4}$ is the same
 $$ \plcrep_4(u) = \plcrep_4(v) = \mresize{
\begin{bmatrix}
    1 & 1 & 2& 2\\
      & 1 & 2 &2 \\
      &   & 1 & 1 \\
      &   &   & 1 \\
\end{bmatrix} }
\times \mresize{
\begin{bmatrix}
    -1 & -1 & 0& 0\\
      & -1 & 0 &0 \\
      &   & -1 & -1 \\
      &   &   & -1 \\
\end{bmatrix} },
$$ which shows that in general  $\plcrep_n$
 is not injective for $n > 3. $
Yet, this representation  is faithful  for  the case of the plactic monoid ~$\plcM_3$ of rank $3$.

\begin{theorem}\label{thm:plc.3.rep} The
homomorphism $\plcrep_3:  \plcM_3 \To \MPlc_3 \times \mir{\MPlc_3}$ in \eqref{eq:map.plcrep.1} is an isomorphism -- a faithful linear representation of the plactic monoid $\plcM_3$.
\end{theorem}

\begin{proof}
Proof by a straightforward   computation. Suppose   $\tb_w \in \Tab_n$
is a tableau of the form
$$ \tb_w := \cc^{k_3} \, \bb^{j_2} \cc^{k_2} \,\aa^{i_1} \bb^{j_1} \cc^{k_1}.$$
corresponding  uniquely to the $3$-configuration tableau
$\ctb_w = \ctab(w)$, by  Theorem~ \ref{thm:configuration}.
As in Theorem \ref{thm:plc.rep}, define $\plcrep_3$ to have the generators' mapping
$$ \ctb_a \mTo (A, \chA), \quad  \ctb_b \mTo (B, \chB), \quad \ctb_c \mTo (C, \chC), $$
where $A, B, C \in \MPlc_3$  and $\chA, \chB, \chC \in \mir{\MPlc_3}$ are described in Example \ref{exmp:tfrm.1} and   \ref{exmp:tfrm.2}, respectively.

For simplicity, write  $i = i_1,$ $j = j_1 j_2$,  $ k = k_1 k_2 k_3$.
 The direct computation of the matrix  products
   $$ X = C^{k_3} \, B^{j_2} C^{k_2} \, A^{i_1} B^{j_1} C^{k_1}, \qquad Y = \chC^{k_3} \, \chB^{j_2} \chC^{k_2} \,\chA^{i_1} \chB^{j_1} \chC^{k_1},$$
  in $\TMat(\Trop)$  results in the correspondence
 $$ \ctb_w = \begin{array}{|l|l|l|l|l|l|l|}
\cline{1-1}
  k_3  \\ \cline{1-2}
  j_2 &  k_2  \\ \cline{1-3}
  i_1 &   j_1  &   k_1  \\ \cline{1-3}
\end{array}\dss{\mTo} X  \times Y = \left[\begin{array}{rcr}
                               i & i_1 j_1 & i_1 j_1 k_1 \\
                                & j & j_2(j_1 \+ k_2) \\
                                 &  & k
                             \end{array}\right] \times \left[\begin{array}{rcr}
                               -k & -k_3(-j_1 \+ -k_2) & -k_3 \\
                                & -j & -j_2 \\
                                 &  & -i
                             \end{array} \right] \,  ,$$
and we need to prove that $\ctb_w$ can be recover uniquely from the matrices $X = (x_{i,j})$ and $Y= (y_{i,j})$.

 First  $i_1$, $j_2$, and $k_3$ can be read off immediately from the right column of $Y$, while
 $j_1$ and $k_1$ are extracted recursively from the top row of $X$. So, it remains
   to compute $j_2$ and $k_2$.

   Since $j_1$ and $j_2$ have  been already recovered, from $x_{2,3}$ we can figure out whether  $k_2 \geq j_1$ or not.
   If $k_2 \geq j_1$, then $x_{2,3} = j_2 k_2$, which provides $k_2$.  Otherwise $k_2 < j_1$, which implies $-k_2 > -j_1$, and thus  $y_{1,2} = (-k_2) (-k_3)$,  which gives $k_2$, as $k_3$ is already known.
(Note that $k_2$ can also be computed  as $(-k)/((-k_1)(- k_3))$.)
\end{proof}

\begin{corollary}\label{cor:plc.id} The plactic monoid $\plcM_3$ admits all the semigroup identities satisfied by  the monoid $\TMat_3(\Trop)$ of triangular tropical matrices, in particular the semigroup identities  \eqref{eq:iduniv2.2} with $n =3 $
\begin{equation*} \Id_{(C,2,2)}: \quad
    \ll2 \ds{\underline{x}}  \ll2 \ds =  \ll2 \ds{\underline{x}}
     \ll2 \
\end{equation*}
by letting  $x = uv$ and $y=vu$, for any   $u,v \in  \plcM_3$, cf. \eqref{eq:id2} in Construction \ref{construct:main}.
\end{corollary}
\begin{proof}
  $\plcM_3$ is isomorphic to the product  $\MPlc_3 \times \mir{\MPlc_3}$ (Theorem \ref{thm:plc.3.rep}), both $\MPlc_3 $ and $ \mir{\MPlc_3}$ are    submonoids  of $\TMat_3(\Trop)$, which by Theorem \ref{thm:trMat.Id} satisfies the identity \eqref{eq:iduniv2.2} for $n =3$.
\end{proof}

The maps $\ctabmat$ and $\ctabcomat$
By our construction, one sees that for any $w \in \plcM_n$  the characters  (cf. \eqref{eq:char.add} and \eqref{eq:char.mlt})
$$\chra(\ctabmat(\ctb_w)) = \chrp(- \ctabcomat(\ctb_w)), \qquad \chrp(\ctabmat(\ctb_w)) = - \chrp (\ctabcomat(\ctb_w)),
$$
and thus, embedding  $\plcrep_n(w)$ trivially  in  $2n \times 2n$ matrix $C_w$ with $\ctabmat(\ctb_w) $ and $\ctabcomat(\ctb_w)$ as its diagonal blocks,
by~ \eqref{eq:map.plcrep.1} we have
$$  \chra(\plcrep_n(w)) = \chra(\ctabmat(\ctb_w)), \qquad \chrp(\plcrep_n(w)) = 0.   $$
Recall that $\chra$ provides the maximal occurrence of a letter in $w$, cf. \S\ref{sec:troplactic.m.alg}. Then, $\{ C_w \ds | w \in \tA_n^*\} \subset \TMat_{2n}(\Trop)$ is a multiplicative monoid whose members all have multiplicative trace  $\mtr = 0$ (Definition ~\ref{def:matOper}).

\subsection{Standard Young tableaux and the symmetric group}
\sSkip

The special structure of standard tableaux $\STab_n$, via the corresponding standard $n$-configuration tableaux  (Lemma \ref{lem:reverse.2}), allows their faithful  realization in terms of the tropical  representation of the \kmon\ $\clkM_n$ (Theorem \ref{thm:clk.represntation}). (We use the terminology ``realization'' as $\STab_n$ does not form a monoid, and our maps here are set-theoretical maps.)
\begin{theorem}[Realization of standard tableaux]\label{thm:stab.rep} Writing $\tb_w = \tab(\std{w})$ and $
\tb'_w = \tab(\rvs{w})$  for standard tableaux of $\std{w}$ and $\rvs{w}$, i.e., $\std{w}$ is a permutation of $\lt_1, \dots, \lt_n$,  then
the map
\begin{equation}\label{eq:map.stdxrvs}
  \stdxrvs_n : \STab_n \TO \MPlc_n \times \MPlc_n, \qquad \tb_w \longmapsto (\tabmat(\tb_w), \tabmat(\tb'_{w})),
\end{equation}
is  injective.

\end{theorem}
\begin{proof} 
  Compose Proposition \ref{prop:reverse.std} with  Theorem \ref{thm:tab.wxr(w)}.
\end{proof}

By the bijective correspondence of the symmetric group $S_n$ to standard Young tableaux \cite{SaganBook},  we define the set-theoretic  map
\begin{equation}\label{eq:sym.rep}
  \symrep_n: S_n \TO \MPlc_n \times \MPlc_n, \qquad \sig \mTo \stdxrvs_n(\tb_\sig),
\end{equation}
induced from \eqref{eq:map.stdxrvs}, and thus obtain a tropical matrix realization of $S_n$.

A supplementary tropical view to $S_n$ is briefly as follows.
Let $s_i\in S_n$ be the simple transposition that permutes $i$ with $i+1$ and leaves the other elements of $\{1, \cdots, n\}$ unchanged. The set $\{ s_1, \dots, s_{n-1}\}$ of all these  transpositions (called Coxeter transpositions) generates the symmetric group $S_n$, and thus every permutation $\sig \in S_n$ can be written in terms of $s_i$'s as a word $s_{i_1} s_{i_2} \cdots s_{i_m}$ with $i_1, \dots, i_m \in \{ 1, \dots, n-1 \}$. To get a reduced form of $\sig$ we take $s_{i_1} s_{i_2} \cdots s_{i_m}$ with minimal  $m$.

Consider the transposition $s_i$ as a word over $\{ 1, \dots, n \}$. It consists of exactly $n$ letters, each of which appears once. Applying the map \eqref{eq:sym.rep} to transpositions $s_i$,  we obtain the induced (set-theoretic) map
\begin{equation}\label{eq:sym.rep.2}
  \symrepB_n: S_n \TO \MPlc_n \times \MPlc_n, \qquad s_i  \mTo \stdxrvs_n(\tb_{s_i}),
\end{equation}
determined now by the generators' mapping. This realization has special properties, for example all
$\symrepB_n(s_i)$ are  quasi-idempotents, i.e., $\symrepB_n(s_i) = q(\symrepB_n(s_i))^2$ with  a fixed $q \in \Real$, establishing an important linkage to Heacke algebras \cite{IZSym}.

\section{Remarks and open problems}\label{sec:problems}

The first question arises from our results concerns a generalization of Theorem \ref{thm:plc.3.rep}.

\begin{problem}
  Is there a faithful linear representation of the plactic monoid of rank $n$ by  triangular (or nonsingular) tropical matrices?
\end{problem}
\noindent
By Theorem \ref{thm:trMat.Id}, a positive answer to this question solves the problem (cf. \cite{KOk.1}):
\begin{problem}
 Does the plactic monoid of rank $n$ satisfy  a nontrivial semigroup identity?
\end{problem}

This paper shows that the plactic monoid
$\plcM_3$ of rank $3$ admits all the semigroup identities satisfied by the monoid $\TMat_3(\Trop)$ of $3 \times 3$ triangular tropical matrices (Corollary~ \ref{cor:plc.id}).
This naturally leads to the converse question:
\begin{problem}
 Do  $\plcM_3$ and  $\TMat_n(\Trop)$
satisfy exactly the same identities?
\end{problem}

A few fundamental facts on the symmetric group $S_n$ are well  known:
\begin{itemize} \ealph
  \item the irreducible representations of $S_n$
are in canonical bijection with partitions $\lm$ of $n$;
  \item the dimension of the
irreducible representation corresponding to a given partition $\lm$ of $n$ is exactly
$d_\lm$ --  the number of standard Young tableaux of shape $\lm$.
\end{itemize}
Actually, these facts were
the original motivation for studying Young tableaux, and many combinatorial
results on  Young tableaux can be expressed by using representations of
$S_n$, e.g., see \cite{SaganBook}.

The bijective correspondence of elements of  $S_n$ to standard Young tableaux, which in their turn are faithfully   realizable in $\MPlc_n \times \MPlc_n$ (Theorem \ref{thm:clk.represntation}),  leads to a new approach for   studying representations of~ $S_n$.
Hook's formula computes the numbers $d_\lm$ in combinatorial terms of Young tableaux, these numbers are  possibly reflected in their tropical matrix  realization.
\begin{problem}
  Can the $d_\lm$'s  be extracted from tropical realizations in an algebraic way?
\end{problem}

\bibliographystyle{abbrv}

\begin{appendix}
\section{Proofs and completions to \S\ref{ssec:co.plc}}\label{apx:A}

In this appendix we bring additional details for \S\ref{ssec:co.plc}, in particular  the full detailed proof of Theorem ~ \ref{thm:coplc2plc}. We use the terminology in \S\ref{ssec:semigroups}.
\begin{example}\label{exmp:co.plc.3} Let $\plcM_3 := \genr{a,b,c}$ be the plactic monoid of rank $3$, for which the \cmirr s \eqref{eq:word.mirr} of  generators  are
$$a' :=  \lcmp{}{a} = ba, \qquad b':= \lcmp{}{b} = ca, \qquad c' := \lcmp{}{c} = cb, $$
   ordered as $a' <   b' <   c'$. Then $a',b',c'$ admit  the Knuth relations \eqref{eq:knuth.rel} and thus the congruence   $\pcong$ of  $\plcM_3$ implies the equivalence $\cpcong$.
  Indeed, for \KXa \ we have \footnote{The middle terms are  presented in their canonical tableau form, that is as   elements of the plactic monoid, as  discussed in \S\ref{ssec:y.tab}. }
  \begin{align}\label{eq:ctab.coplac.A.1}
    \aa' \; \cc' \; \bb'  = ba \; cb \; ca = c\; bb \; aab \;
 = cb \; ba \; ca =  \cc' \; \aa' \; \bb',  & & \ctb = \mresize{ \begin{array}{|l|l|l|l}
\cline{1-1}
  {1}   \\ \cline{1-2}
  {2} &  0  \\ \cline{1-3}
  {2} &   0  &   {1}  &,  \\ \cline{1-3}
\end{array}}
  \end{align}
and
  \begin{align}\label{eq:ctab.coplac.A.2}
    \aa' \; \bb' \; \aa'  = ba \; ca \; ba = c\; bb \; aaa \;
 = ca \; ba \; ba =  \bb' \; \aa' \; \aa',  & & \ctb = \mresize{\begin{array}{|l|l|l|l}
\cline{1-1}
  1   \\ \cline{1-2}
  {2} &  0  \\ \cline{1-3}
  3 &   0 &   0  &,  \\ \cline{1-3}
\end{array}}
  \end{align}
  when $\aa' < \bb'$.

  For the second relation \KXb \ we have \footnotemark[5]
  \begin{align}\label{eq:ctab.coplac.B.1}
    \bb '  \;\aa '  \; \cc'  = ca \; ba \; cb  = c\; bc \; aab \;   = ca \; cb \; ba = \bb' \; \cc' \; \aa',  &
     & \ctb = \mresize{ \begin{array}{|l|l|l|l}
\cline{1-1}
  {1}   \\ \cline{1-2}
  1 &  1  \\ \cline{1-3}
  {2} &   1  &   0  & , \\ \cline{1-3}
\end{array}}
  \end{align} where
  \begin{align}\label{eq:ctab.coplac.B.2}
    \cc '  \;\bb '  \; \cc'  = cb \; ca  \; cb  = c\; bc \; abc \;   = cb \; cb \; ca  = \cc' \; \cc' \; \bb' , &
     & \ctb = \mresize{ \begin{array}{|l|l|l|l}
\cline{1-1}
  1  \\ \cline{1-2}
  1 &  1  \\ \cline{1-3}
  1 &   1  &   1  & , \\ \cline{1-3}
\end{array}}
  \end{align}
  for $\bb' < \cc'$.

  For the two other 3-letter words we get
  \begin{align*}
   \aa '  \;  \bb '  \; \cc'  = ba \; ca \;  cb  = bcc \; aa b \; , & & \ctb =  \mresize{ \begin{array}{|l|l|l|l}
\cline{1-1}
  2  \\ \cline{1-2}
  2 &  0  \\ \cline{1-3}
  2 &   0  &   0  & , \\ \cline{1-3}
\end{array}} \\
   \cc' \;  \bb '  \;\aa '    = cb \; ca \; ba   = cc \; bb \; aa \;, &&  \ctb =  \mresize{ \begin{array}{|l|l|l|l}
\cline{1-1}
  0  \\ \cline{1-2}
  1 &  2  \\ \cline{1-3}
  2 &   1  &   0  & . \\ \cline{1-3}
\end{array}}
  \end{align*}
So we see that the Knuth relations \eqref{eq:knuth.rel} are ``faithfully'' preserved.

\end{example}

Let $\lt'_\ell :=  \lcmp{}{\lt_\ell}$ be the \cmirr \ of $\lt_\ell \in \tA_n$, which is a word in $\tA_n^*$,   and set $\ell'= n - \ell +1$. The $n$-configuration tableau of $\lt'_\ell$ is tableau
\begin{equation}\label{eq:C.2}
\ctab(\lt'_\ell) = \begin{array}{c} \\ \\ \\ \ell' \\ \\ \\ \\ \end{array} {\small \begin{array}{|l|l|l|l|l|l|l|}
\cline{1-1}
  {0}   \\ \cline{1-2}
  {0}  & 1 \\ \cline{1-3}
  \vdots &  \vdots & \ddots \\  \cline{1-4}
   0   &   {1}  & 0 & \\ \cline{1-5}
  {1} &   {0}  &   \cdots & &  0 \\ \cline{1-6}
  \vdots &   \vdots  &  \ddots  &  & & \ddots\\ \cline{1-7}
  {1} &   {0}  &   \cdots & 0 & & \cdots & 0 \\ \cline{1-7}
 \end{array}}
\end{equation}
whose  $\ell'$-diagonal is zero;  in the case that $\ell = n$ (i.e., when $\ell' = 1$) the left column is empty.

\thmcite{Theorem \ref{thm:coplc2plc}}{The \cmirr s \eqref{eq:word.mirr} of the generators $\lt_1, \dots , \lt_n$ of the plactic monoid  $\plcM_n$,
$\lt'_\ell :=  \lcmp{}{\lt_\ell},$ $\ell = 1, \dots,n ,  $
ordered
as $ \lt'_1 < \lt'_2 < \cdots < \lt'_n,  $
admit  the Knuth relations \eqref{eq:knuth.rel} and thus the congruence  $\pcong$ of  $\plcM_n$ implies the equivalence $\cpcong$.}

\begin{proof}[Proof of Theorem \ref{thm:coplc2plc}] Proof by induction on $n$.
The case of $n =3$ has been  proven  in Example \ref{exmp:co.plc.3}. Assuming the implication holds for $n-1$, we prove it by cases for $n$, heavily basing on the Encoding Algorithm ~\ref{algr:conf} of configuration tableaux.

For an easy exposition,  we write the words $\aa_1', \dots, \aa_n'$ in $\tA_n^*$,  each is of length $(n-1)$,  as a table
\begin{equation}\label{eq:mirr.2}
  \begin{array}{l|l|l|l|l|l|l|l }
\cline{2-7}
  \aa_1' : &      & \aa_{n-1} & \aa_{n-2} & \cdots  & \aa_2 & \aa_1 \\ \cline{2-7}
  \aa_2' : &   \aa_n  & & \aa_{n-2} & \cdots &\aa_2   & \aa_1 \\ \cline{2-7}
  \aa_3' : &   \aa_n  &  \aa_{n-1}   &  &\ddots & &  \vdots \\ \cline{2-7}
  \vdots &  \vdots &    &\ddots & & \aa_2 & \aa_1  \\ \cline{2-7}
  \aa_{n-1}' :  &  \aa_n  &   \cdots & &  \aa_3 & & \aa_1  \\ \cline{2-7}
  \aa_n' : &  \aa_n  &   \cdots & & \aa_3 & \aa_2 & &,   \\ \cline{2-7}
 \end{array}
\end{equation}
where the empty spaces stand for absent letters.
To indicate that an  $n$-configuration tableau $\ctb$ is considered with respect to  the  sub-alphabet $\{ \lt_k, \dots, \lt_n \} \subset \tA_n$, we denote it by $\ctb^{(k:n)}$;  $\ctb^{(n)}_w := \ctb^{(1:n)}_w$ denotes the $n$-configuration tableau of a word $w \in \tA^*_n$ over the whole alphabet $\tA_n = \{ \lt_1, \dots, \lt_n \}$.

 As the word  $\aa'_{\ell}$ has a decreasing order, the encoding of its $i$'th letter in $\ctb^{(n)}$ ``bumps'' all the pervious letters,  resulting in a tableau of the form~  \eqref{eq:C.2}, which is similar to  encoding  its $(i-1)$-prefix in a configuration tableau of a fewer letters.
 Thus, we  obtain the following observations for the encoding  $ \lt'_\ell \insrt \ctb$ of $\lt'_\ell$ in   $\ctb$. \begin{itemize}\dispace
  \item[(A)] If $\ell < n $  then  $\lt'_\ell = u \lt_1 $ and the upper part of $\ctb_{u \lt_1}^{(n)}$ is equal to
  $\ctb_{u }^{(2:n)}$.
  \item[(B)] When $\ell =n$, the right part of $\ctb_{\lt'_n}^{(n)}$ is equal to $\ctb_{\lt'_n}^{(2:n)}$.
\end{itemize}
Let $\aa' = \aa_p', \bb' = \aa_q'$, $\cc' = \aa_r'$, where
$p \leq q \leq r$. By \eqref{eq:knuth.rel} we have to prove the following equalities
\begin{description}\dispace
\item[\KXa] $\aa' \; \cc' \; \bb' = \cc' \; \aa' \; \bb' \  \text{if }   p \leq q < r $,
\item[\KXb] $\bb' \;\aa' \; \cc' = \bb' \; \cc' \; \aa' \  \text{if }  p <  q \leq r    $.
\end{description}
To do so we proceed by cases, depending on the values of $p$, $q$, and $r$.

\pSkip
\textbf{Case I.} $r < n$: \quad Then also $p,q < n$, and we can write $\aa'_p = u_p \aa_1$, $\aa'_q = u_q \aa_1$, and $\aa'_r = u_r \aa_1$. All
have the same terminating  letter $\aa_1$ (i.e., the $1$-suffix), and observation (A) holds sequentially for concatenations of $\aa'_p$, $\aa'_q$, and $\aa'_r$,  e.g. see \eqref{eq:ctab.coplac.A.2}.
Therefore
$$\ctb^{(n)}_{\aa'_i \aa'_j \aa'_k} = \mresize{ \begin{array}{|l|l|l|l|l|}
\cline{1-1}
    \\ \cline{2-2}
    \multicolumn{2}{|l|}{\embox}  \\ \cline{3-3}
   \multicolumn{3}{|l|}{\ctb^{(2:n)}_{u_i u_j u_k}}     \\ \cline{4-4}
   \multicolumn{4}{|l|}{\embox}     \\ \cline{1-5}
    {3} &   0  &  \ldots & 0 &  0 \\ \cline{1-5}
\end{array}}
$$for  $i,j,k \in \{ p,q,r\}$.
Thus \begin{equation}\label{eq:u.v.w}
\begin{array}{ccc}
\ctb _{\aa' \cc' \bb'}^{(n)} = \ctb_{\cc' \aa'  \bb'}^{(n)} &  \Iff \quad & \ctb_{u_p u_r u_q}^{(2:n)} = \ctb_{u_r u_p u_q}^{(2:n)}, \\[2mm]
\ctb _{\bb' \aa' \cc'}^{(n)} = \ctb_{\bb' \cc' \aa'  }^{(n)} & \Iff \quad  & \ctb_{u_q u_p u_r }^{(2:n)} = \ctb_{u_q u_r u_p }^{(2:n)},
\end{array}
\end{equation}
and this case is completed by induction.

\pSkip
\textbf{Case II.} $r = n $: \quad We have several  sub-cases, depending on the indices $p, q$, which we recall satisfy  $p  \leq q \leq n$.

\begin{enumerate}\ealph
  \item
 $p \leq q < n-1 < r = n$:
\quad
Write $\aa'_p = u_p \aa_1$, $\aa'_q = u_q  \aa_1$,   $\aa'_r = u_r \aa_2$.
Then, $\aa_2$ is the  $1$-suffix  of $\aa'_r$, and it is the smallest letter in $\aa'_r$, which is  also the $1$-suffix of  both $u_p$ and ~$u_q$, since $p,q < n-1$.  Thus,
 observation~(A) holds inductively, and hence
$$ \ctb^{(n)}_{\aa'_i \aa'_j \aa'_k} = \mresize{ \begin{array}{|l|l|l|l|l|l}
\cline{1-1}
    \\ \cline{2-2}
    \multicolumn{2}{|l|}{\embox}  \\ \cline{3-3}
   \multicolumn{3}{|l|}{\ctb^{(2:n)}_{u_i u_j u_k}}     \\ \cline{4-4}
   \multicolumn{4}{|l|}{\embox}     \\ \cline{1-5}
    2 &   1  & 0 & \ldots &  0  & .\\ \cline{1-5}
\end{array}}$$
This case is then completed by induction as in  \eqref{eq:u.v.w}.

\pSkip
\item   $p <  q = n-1 < r = n$:  \quad
Write  $\aa'_p = u  \aa_2 \aa_1$, $\aa'_q =  v\aa_3\aa_1$, and $\aa'_r = v \aa_3 \aa_2 $ (e.g. see \eqref{eq:ctab.coplac.A.1} and  \eqref{eq:ctab.coplac.B.1}).

\pSkip
\KXa:  The first row of $\ctb_{\aa'_p  \aa'_r} $ is $11 0\cdots 0$. The encoding  $\aa'_q \insrt \ctb_{\aa'_p  \aa'_r}  $
increments $\lm_{1,3}$ by   $\aa_3 \insrt \ctb_{\aa'_p  \aa'_r v}  $   to obtain $\ctb_{\aa'_p   \aa'_r v \aa_3 } |_{\row_1} = 111 0 \cdots 0 $, and then $\aa_1 \insrt \ctb_{\aa'_p  \aa'_r v \aa_3}$ gives $\ctb_{\aa'_p   \aa'_r \aa'_q } |_{\row_1} = 2010 \cdots 0$. For $\ctb_{\aa'_r   \aa'_p }$ we have the same  $\ctb_{\aa'_r   \aa'_p } |_{\row_1} = 110 \cdots 0$, and as before
 $\aa'_q \insrt \ctb_{\aa'_r   \aa'_p }$ results in  $\ctb_{\aa'_r   \aa'_p  \aa'_q}|_{\row_1} = 2010 \cdots 0$. Then, by induction on $\aa'' = u \aa_2$, $\bb'' = v \aa_3$, $\cc'' = v \aa_3$, we have
$$   \text{\KXa} \ (\aa \; \cc \; \bb = \cc \; \aa \; \bb):  \qquad   \ {\small \begin{array}{|l|l|l|l|l|}
\cline{1-1}
    \\ \cline{2-2}
    \multicolumn{2}{|l|}{\embox}  \\ \cline{3-3}
   \multicolumn{3}{|l|}{\ctb^{(2:n)}_{u \aa_2 v \aa_3 v \aa_3 }}     \\ \cline{4-4}
   \multicolumn{4}{|l|}{\embox}     \\ \cline{1-5}
    2 &   0  & 1 & \ldots &  0 \\ \cline{1-5}
\end{array}} \dss = {\small \begin{array}{|l|l|l|l|l|l}
\cline{1-1}
    \\ \cline{2-2}
    \multicolumn{2}{|l|}{\embox}  \\ \cline{3-3}
   \multicolumn{3}{|l|}{\ctb^{(2:n)}_{v \aa_3 u \aa_2 v \aa_3 }}     \\ \cline{4-4}
   \multicolumn{4}{|l|}{\embox}     \\ \cline{1-5}
    2 &   0  & 1 & \ldots &  0 & . \\ \cline{1-5}
\end{array}}  $$

\pSkip
\KXb:  The first row of $\ctb_{\aa'_q  \aa'_p} $ is $2 0\cdots 0$, then $\aa'_r \insrt \ctb_{\aa'_q  \aa'_p}  $
increments $\lm_{1,2}$   to obtain $\ctb_{\aa'_q   \aa'_p  \aa'_r } |_{\row_1} = 21 0 \cdots 0 $. For $\ctb_{\aa'_q   \aa'_r }$ we have   $\ctb_{\aa'_q   \aa'_r } |_{\row_1} = 110 \cdots 0$,
and  $\aa'_p \insrt \ctb_{\aa'_q  \aa'_r }$ results in  $\ctb_{\aa'_q   \aa'_r  \aa'_p}|_{\row_1} = 210 \cdots 0$, since it increments $\lm_{1,2}$ and immediately decrements it by the encoding of $\aa_1$, which increments $\lm_{1,1}$.  Then, by induction on $\aa'' = u \aa_2$, $\bb'' = v \aa_3$, $\cc'' = v \aa_3$, we have

$$   \text{\KXb} \  (\bb \;\aa \; \cc = \bb \; \cc \; \aa  ): \qquad
{\small \begin{array}{|l|l|l|l|l|}
\cline{1-1}
    \\ \cline{2-2}
    \multicolumn{2}{|l|}{\embox}  \\ \cline{3-3}
   \multicolumn{3}{|l|}{\ctb^{(2:n)}_{v \aa_3 u \aa_2 v \aa_3}}     \\ \cline{4-4}
   \multicolumn{4}{|l|}{\embox}     \\ \cline{1-5}
    2 &   1  & 0 & \ldots &  0 \\ \cline{1-5}
\end{array}} \dss =
{\small \begin{array}{|l|l|l|l|l|l}
\cline{1-1}
    \\ \cline{2-2}
    \multicolumn{2}{|l|}{\embox}  \\ \cline{3-3}
   \multicolumn{3}{|l|}{\ctb^{(2:n)}_{v \aa_3 v \aa_3 u \aa_2 }}     \\ \cline{4-4}
   \multicolumn{4}{|l|}{\embox}     \\ \cline{1-5}
    2 &   1  & 0 & \ldots &  0 & . \\ \cline{1-5}
\end{array}}
$$

\pSkip
\item $p =  q = n-1< r = n$:  \quad
Write  $\aa'_p = \aa'_q =  u \aa_3\aa_1$ and $\aa'_r = v \aa_3 \aa_2 $ (e.g. \eqref{eq:ctab.coplac.A.2}). The first row of $\ctb_{\aa'_p  \aa'_r} $ is $11 0\cdots 0$, then $\aa'_q \insrt \ctb_{\aa'_p  \aa'_r}  $
increments $\lm_{1,3}$ and  $\lm_{1,1}$, and increments $\lm_{1,2}$ to obtain $\ctb_{\aa'_p   \aa'_r \aa'_q } |_{\row_1} = 2010 \cdots 0 $.
On the other hand  $\ctb_{\aa'_r   \aa'_p } |_{\row_1} = 101 \cdots 0$,
and  $\aa'_q \insrt \ctb_{\aa'_r   \aa'_p }$ gives $\ctb_{\aa'_r   \aa'_p  \aa'_q}|_{\row_1} = 2010 \cdots 0$. But $u = v$,  and thus

$$    \text{\KXa} \ (\aa \; \cc \; \bb = \cc \; \aa \; \bb):  \qquad  {\small \begin{array}{|l|l|l|l|l|}
\cline{1-1}
    \\ \cline{2-2}
    \multicolumn{2}{|l|}{\embox}  \\ \cline{3-3}
   \multicolumn{3}{|l|}{\ctb^{(2:n)}_{u \aa_3 v\aa_3 \aa_2 u}}     \\ \cline{4-4}
   \multicolumn{4}{|l|}{\embox}     \\ \cline{1-5}
    2 &   0  & 1 & \ldots &  0 \\ \cline{1-5}
\end{array}} \dss =
{\small \begin{array}{|l|l|l|l|l|l}
\cline{1-1}
    \\ \cline{2-2}
    \multicolumn{2}{|l|}{\embox}  \\ \cline{3-3}
   \multicolumn{3}{|l|}{\ctb^{(2:n)}_{v \aa_3 u \aa_3 \aa_2 u}}     \\ \cline{4-4}
   \multicolumn{4}{|l|}{\embox}     \\ \cline{1-5}
    2 &   0  & 1 & \ldots &  0 & . \\ \cline{1-5}
\end{array}}
$$
(In this case,  \KXb\ is not relevant here.)
\pSkip
\item $ p < n-1 < q = r = n$: \quad
Let $\aa'_p = u  \aa_2 \aa_1$ and   $\aa'_q =  \aa'_r = v \aa_3 \aa_2 $. The first row of $\ctb_{\aa'_q  \aa'_p} $ is $11 0\cdots 0$, then $\aa'_r \insrt \ctb_{\aa'_q  \aa'_p}  $
increments  $\lm_{1,2}$ to obtain $\ctb_{\aa'_q   \aa'_p \aa'_r } |_{\row_1} = 120 \cdots 0 $.
On the other hand $\ctb_{\aa'_q   \aa'_r } |_{\row_1} = 020 \cdots 0$,
and  $\aa'_p \insrt \ctb_{\aa'_q   \aa'_r }$ gives $\ctb_{\aa'_q   \aa'_r  \aa'_p}|_{\row_1} = 120\cdots 0$. Thus
$$
\text{\KXb} \  (\bb \;\aa \; \cc = \bb \; \cc \; \aa  ): \qquad
 {\small \begin{array}{|l|l|l|l|l|}
\cline{1-1}
    \\ \cline{2-2}
    \multicolumn{2}{|l|}{\embox}  \\ \cline{3-3}
   \multicolumn{3}{|l|}{\ctb^{(2:n)}_{v \aa_3 u \aa_2 v \aa_3}}     \\ \cline{4-4}
   \multicolumn{4}{|l|}{\embox}     \\ \cline{1-5}
    1 &   2  & 0 & \ldots &  0 \\ \cline{1-5}
\end{array}} \dss =
{\small \begin{array}{|l|l|l|l|l|l}
\cline{1-1}
    \\ \cline{2-2}
    \multicolumn{2}{|l|}{\embox}  \\ \cline{3-3}
   \multicolumn{3}{|l|}{\ctb^{(2:n)}_{v \aa_3 v \aa_3 u \aa_2 }}     \\ \cline{4-4}
   \multicolumn{4}{|l|}{\embox}     \\ \cline{1-5}
    1 &   2  & 0 & \ldots &  0 & .\\ \cline{1-5}
\end{array}} $$
by induction on $\aa'' = u \aa_2$, $\bb'' = \cc'' = v \aa_3$.
(In this case,  \KXa\ is not relevant here.)

\pSkip
\item $p = n-1 < q = r= n, $: \quad
Write  $\aa'_p = u  \aa_3 \aa_1$ and   $\aa'_q =  \aa'_r = v \aa_3 \aa_2 $ (e.g. see \eqref{eq:ctab.coplac.B.2}). The first row of $\ctb_{\aa'_q  \aa'_p} $ is $101 0\cdots 0$, then $v \insrt \ctb_{\aa'_q  \aa'_p}  $ fills the forth column by $1$, which is bumped by $\aa_3 \insrt \ctb_{\aa'_q  \aa'_p v }  $. So $\ctb_{\aa'_q  \aa'_p v \aa_3}|_{\row_1} = 1020 \cdots 0$, and $\aa_2 \insrt \ctb_{\aa'_q  \aa'_p v \aa_3}  $ gives $\ctb_{\aa'_q  \aa'_p \aa'_r}|_{\row_1} = 1110 \cdots 0$, as $\lm_{1,3}$ is  decremented by encoding $\aa_2$. On the other hand
$\ctb_{\aa'_q   \aa'_r } |_{\row_1} = 020 \cdots 0$,  and $\aa'_p \insrt \ctb_{\aa'_q   \aa'_r }$ gives $\ctb_{\aa'_q   \aa'_r  \aa'_p}|_{\row_1} = 1110\cdots 0$. Thus
$$
\text{\KXb} \  (\bb \;\aa \; \cc = \bb \; \cc \; \aa   ): \qquad  \ {\small \begin{array}{|l|l|l|l|l|}
\cline{1-1}
    \\ \cline{2-2}
    \multicolumn{2}{|l|}{\embox}  \\ \cline{3-3}
   \multicolumn{3}{|l|}{\ctb^{(2:n)}_{\aa_q' u \aa_3 v }}     \\ \cline{4-4}
   \multicolumn{4}{|l|}{\embox}     \\ \cline{1-5}
    1 &   1  & 1 & \ldots &  0 \\ \cline{1-5}
\end{array}} \dss = {\small \begin{array}{|l|l|l|l|l|l}
\cline{1-1}
    \\ \cline{2-2}
    \multicolumn{2}{|l|}{\embox}  \\ \cline{3-3}
   \multicolumn{3}{|l|}{\ctb^{(2:n)}_{\aa_q' v \aa_3 u }}     \\ \cline{4-4}
   \multicolumn{4}{|l|}{\embox}     \\ \cline{1-5}
    1 &   1  & 1 & \ldots &  0 & , \\ \cline{1-5}
\end{array}}  $$
since $u = v$. (In this case,  \KXa\ is not relevant here.)
\end{enumerate}
The above cases show that involving a new letter reflects inductively on extending the tableaux by a bottom row,  and the proof is completed.
\end{proof}

\end{appendix}
\end{document}